\documentclass{article}

\usepackage{stmaryrd}

\usepackage{amsmath}

\usepackage[utf8]{inputenc}
\usepackage[T1]{fontenc}
\usepackage[british,UKenglish,USenglish,english,american]{babel}

\usepackage{hyperref}
\usepackage{amssymb}
\usepackage{amsthm,amscd}
\usepackage{mathrsfs}
\usepackage{amsfonts}

\usepackage{esint}
\usepackage{array}
\usepackage{latexsym}

\usepackage{lmodern}

\usepackage{graphicx}
\usepackage{graphics}
\usepackage{epsfig}
\usepackage{color}
\DeclareGraphicsExtensions{.eps,.pdf,.jpeg,.png}

\usepackage{titlesec}
\usepackage{multirow}

\usepackage[all]{xy}
\usepackage{dsfont}

\usepackage{hyperref}
\usepackage{bigints}

\usepackage{color}
\definecolor{darkgreen}{cmyk}{1,0,1,.2}
\definecolor{m}{rgb}{1,0.1,1}
\definecolor{green}{cmyk}{1,0,1,0}
\definecolor{darkred}{rgb}{0.55, 0.0, 0.0}
\definecolor{test}{rgb}{1,0,0}
\definecolor{cmyk}{cmyk}{0,1,1,0}
\definecolor{orange}{rgb}{1,0.5,0}

\usepackage[normalem]{ulem}
\renewcommand\sout{\bgroup\markoverwith
{{\rule[0.7ex]{3pt}{1.4pt}}}\ULon}

\def\Av{\operatorname{Av}}
\def\End{\operatorname{End}}
\def\END{\operatorname{END}}

\def\ind{\operatorname{ind}}
\def\End{\operatorname{End}}

\def\id{\operatorname{id}}
\def\Id{\operatorname{Id}}

\def\Hom{\operatorname{Hom}}

\def\Ind{\operatorname{Ind}}

\def\Supp{\operatorname{Supp}}

\def\SU{\operatorname{SU}}

\def\C{\mathbb C}

\def\N{\mathbb N}
\def\R{\mathbb R}

\def\Z{\mathbb Z}
\def\T{\mathbb T}

{\theoremstyle {definition} \newtheorem {definition} {Definition} [section] }

{\theoremstyle {definition} \newtheorem {defi} {Definition} [section] }
{\theoremstyle {plain}  \newtheorem {thm} [defi] {Theorem}}
{\theoremstyle {plain}  \newtheorem {cor} [defi]{Corollary}}
{\theoremstyle {plain} \newtheorem {prop} [defi]{Proposition}}
{\theoremstyle {plain} \newtheorem {lem}[defi] {Lemma}}

{\theoremstyle {definition} }

{\theoremstyle {definition} \newtheorem{remarque}[defi]{Remark}}
{\theoremstyle {definition} }
{\theoremstyle {definition} }
{\theoremstyle {definition}  }
{\theoremstyle {definition} }
{\theoremstyle {plain}  }

\def\Hom{{\mathrm{Hom}}}
\def\End{{\mathrm{End}}}

\def\g{{\mathfrak{g}}}

\def\K{{\mathrm{K_G}}}

\def\k{{\mathrm{K}}}
\def\ind{{\mathrm{Ind^{\mathcal{F}}}}}

\def\ext{{\mathrm{ext}}}

\def\interieur{{\mathrm{int}}}

\usepackage{authblk}
\newcommand{\email}[1]{\href{mailto:#1}{#1}}

\newcommand\maA{\mathcal A}

\newcommand\maN{\mathcal N}
\newcommand\maE{\mathcal E}
\newcommand\maF{\mathcal F}

\newcommand\maG{\mathcal G}
\newcommand\maK{\mathcal K}

\newcommand\maL{\mathcal L}

\newcommand\maV{\mathcal V}
\newcommand\maU{\mathcal U}

\newcommand{\Spin}{\operatorname {Spin}}

\newcommand{\SO}{\operatorname {SO}}

\newcommand\ep{\epsilon}

\title{The index of leafwise $G$-transversally \\elliptic operators on foliations\\{\em{\small{Dedicated  to Varghese Mathai  on the occasion of his sixtieth birthday}}}}
\author[1]{Alexandre Baldare}
\affil[1]{Institut für Analysis, Welfengarten 1, 30167 Hannover,
Germany, \email{alexandre.baldare@math.uni-hannover.de}}
\author[2]{Moulay-Tahar Benameur}
\affil[2]{IMAG, UMR 5149 du CNRS, Universit\'e de Montpellier, France,
\email{moulay.benameur@umontpellier.fr}}
\date{}
\begin{document}

\maketitle

%
%
%
%
%

\begin{abstract}
We introduce and study the index morphism for  leafwise $G$-transversally elliptic operators on smooth closed foliated manifolds. We prove the usual axioms of excision, multiplicativity and induction for closed subgroups.  In the case of free actions, we relate our index class with the Connes-Skandalis index class of the corresponding leafwise elliptic operator on the quotient foliation. Finally we prove the compatibility of our index morphism with the Gysin morphism and reduce its computation to the case of tori actions. We also construct a topological candidate for an index theorem using the Kasparov Dirac element for euclidean $G$-representations.
\end{abstract}



\tableofcontents

\bigskip


\section*{Introduction}

The present paper is devoted to the index theory for leafwise pseudodifferential operators on smooth foliations, which are $G$-invariant and leafwise $G$-transversally elliptic, for a given leafwise action of a compact Lie group $G$.  Since $G$-invariant elliptic operators are $G$-transversally elliptic, our results encompass the equivariant Connes-Skandalis index theory for leafwise elliptic operators \cite{ConnesSkandalis, benameur1994theoreme} as well as the  classical  index theory for $G$-transversally elliptic operators \cite{atiyah1974elliptic, BV:IndEquiTransversal} with its fibered version obtained in \cite{baldare:KK, baldare:H}. Due to the lack of full ellipticity, the kernel of such operators is infinite dimensional in general. However, the  $G$-invariance of the operator along the orbits together with its ellipticity in the directions transverse to these orbits ensures that this kernel  contains  irreducible representations only with finite multiplicities. 

In \cite{atiyah1974elliptic}, Atiyah showed  that the index of a $G$-transversally elliptic operator is well defined  as a central distribution  which actually lives  in the $-\dim (G)$ Sobolev space of $G$.  He also proved a list of important axioms and reduced  the  computation of the distributional index to the case of linear actions of tori. An important observation is that even in the elliptic case, by embedding the operators in the wider $G$-transversally elliptic class, one can benefit from their better functoriality properties and it is often easier  to deform  when only $G$-transversal ellipticity is prescribed. This idea  proved fruitful in many situations explained below, in particular in some recent approaches to the differential $[Q, R]=0$ problem, see for instance \cite{ParadanVergne17}. In the case of locally free actions of tori, Atiyah succeded proving a signature formula for the singular quotient, which was the starting point for the famous Kawasaki work on orbifolds \cite{kawasaki1981index}. Later on, N. Berline and M. Vergne proved a delocalized cohomological formula for $G$-invariant $G$-transversally elliptic operators, see \cite{BV:formuleloc:Kirillov, BV:ChernCharactertransversally,BV:IndEquiTransversal}. This formula computes the  Atiyah index distribution around a given  $s\in G$ as an integral of equivariant characteristic classes, by using the Kirillov localization principle.  It is worthpointing out the close relation of these delocalized index formulae with the Duistermaat-Heckman theorem for tori actions on compact symplectic manifolds \cite{DuistermaatHeckman} which actually motivated these index formulae, see also \cite{AtiyahBottMoment}. More recently in \cite{Kasparov:KKindex}, Kasparov applied  the classical ''Bott $\leftrightarrow$ Dirac'' proof of the Atiyah-Singer theorem for elliptic operators,  to investigate the index problem now for  $G$-transversally elliptic operators. His approach is new and  computes the $\k$-homology index class  \cite{julg1982induction} in terms of the stable homotopy class of the principal symbol. Kasparov actually considered the more general setting of locally compact Lie groups acting properly and cocompactly on smooth manifolds and succeeded in this wide generality to compute the  index class as a cup product of the symbol class by a fundamental Dirac element. This is an important progress and it increased interest in $G$-transversally elliptic operators. Previous results were also obtained for bounded geometry manifolds in \cite{kordyukov1994transv, kordyukov1995transv}.  Here are some well known constructions  where $G$-transversally elliptic operators  play a significant part:
\begin{itemize}
\item the Mathai-Melrose-Singer fractional index \cite{MMS1,MMS2} is the evaluation of the distributional index of a $G$-transversally elliptic operator on a specific  test function localizing at the neutral element, here $G=\SU (N)$, see as well \cite{paradan:projective};
\item projective elliptic genera as introduced by Han-Mathai \cite{HanMathai} turn out to also be related with a $G$-transversally elliptic operator, here $G=\Spin (N)$;   
\item the basic index problem for riemannian foliations \cite{Kacimi}, see the recent developements \cite{GorokhovskyLott, BrueningKamberRichardson}, reduces via Molino's theory to the computation of the index of a $G$-transversally elliptic operator. Here $G=\SO(N)$;
\item For locally free $G$-actions, one recovers a first affordable index problem for transversally elliptic operators on foliations \cite{ConnesMoscovici}. Then our more general setting gives a first example towards the harder index problem for, $F_1$-leafwise elliptic and $F_2$-transversally elliptic operators, as studied in \cite{hilsum1987morphismes} for a bifoliation $F_2\subset F_1$. 
\item as explained above, the Kawasaki index formula for orbifolds \cite{kawasaki1981index} can be recovered as  a corollary  of the Berline-Vergne cohomological formula for $G$-transversally elliptic operators, see  \cite{VergneOrbifold};
\item the Guillemin-Sternberg principle \cite{MeinrenkenSjamaar} and more specifically its spin$^c$ version as studied in \cite{ParadanVergne17} relies on ideas from the index theory of $G$-transversally elliptic operators. See also the  interesting extension to proper actions by Hochs-Mathai in \cite{HochsMathai};
\item etc...
\end{itemize}

The first observation for the foliation case treated in the present paper is that, using a holonomy invariant transverse measure, leafwise $G$-transversally elliptic operators do have well defined measured  distributional indices, similar to the Atiyah distribution, and defined using Connes' machinary \cite{Connes:integration:non:commutative} and the Murray-von Neumann dimension theory. 
Exactly as $G$-transversally elliptic operators on closed manifolds provide type I spectral triples,  the holonomy invariant measure allows here to see any  leafwise $G$-transversally elliptic operator as a semi-finite spectral triple  on the convolution algebra $C^\infty (G)$, in the sense of \cite{BenameurFack}.  Now when such holonomy invariant measure does not exist, one is naturally led to the construction of an index theory taking place in appropriate bivariant $\k$-theory groups. More precisely, given a smooth foliation $\maF$ of a closed riemannian manifold $M$ together with a smooth isometric action of a compact Lie group $G$ by leaf-preserving diffeomorphisms, any  leafwise pseudodifferential operator $P$ which is $G$-invariant and leafwise $G$-transversally elliptic is shown to have an index class living in the bivariant group $\k\k (C^*G, C^*(M, \maF))$, where $C^*(M, \maF)$ is Connes' $C^*$-algebra. Therefore, when evaluated at  irreducible representations of $G$, the index of $P$ embodies  a $\k$-multiplicity map 
$$
m_P:{\widehat G}\longrightarrow \k (C^*(M, \maF)).
$$
When $\maF$ is top dimensional,  this is the usual integer valued multiplicity map \cite{atiyah1974elliptic}. In fact and as expected, if we denote by $F_G$ the space of leafwise tangent vectors which are orthogonal to the orbits of the $G$-action, then  our index construction induces the index morphism
$$
\Ind^\maF : \K(F_G) \longrightarrow \k\k (C^*G, C^*(M, \maF)).
$$
We prove here  a list of  axioms which reduce, as in the classical case, the computation of the index class to the simplest case of linear tori actions. \\

Let us explain  more in detail some results. 
Due to the high complexity of the transverse geometry of foliations, the expected axioms are surprisingly hard  to formulate and to prove, without exploiting Kasparov's theory and the powerful tool of the cup product. First notice that when the $G$-invariant operator is leafwise elliptic, its symbol still defines a class in $\K (F_G)$ which is the restriction of the usual class in $\K (F)$, and its index as a $G$-transversally elliptic operator can then be deduced from the classical  elliptic index class in $\k\K (\C, C^*(M, \maF))$ using composition with the trivial representation of $G$. 
Moreover, assume that $(M, \maF) \to (B, \maF^B)$ is a principal $G$-equivariant bundle so that $G$ preserves the leaves of $\maF$ and induces the foliation $\maF^B$ downstairs in $B$, and that $P$ is a $G$-invariant leafwise $G$-transversally elliptic operator on $(M, \maF)$ which corresponds through its symbol to a leafwise elliptic operator $P_0$ on $(B ,\maF^B)$ then we show that the Connes-Skandalis leafwise index of $P_0$ can be recovered from the index class of $P$ upstairs by evaluation at the trivial representation, modulo  composition with a standard Morita extension morphism. These results explain the compatibility of the leafwise $G$-transversally elliptic theory with the elliptic one. We prove here many axioms for our index morphism, such as 
multiplicativity and excision, see Theorems \ref{thm:excision} and \ref{thm:multiplicativité:indice}. Let us state now the compatibility of our index map with the Gysin morphisms of foliations.

\medskip

\begin{thm}
{Let $(M', \maF')$ be a smooth $G$-foliation. Let $\iota: M \hookrightarrow M'$ be a $G$-equivariant embedding of a closed $G$-manifold $M$ which is transverse to the foliation $\maF'$ and denote by $\maF=\iota^*\maF'$  the pull back foliation.}  Then for any $j\in \Z_2$, the following diagram commutes:
$$
\xymatrix{\k^j_\mathrm{G}(F_G) \ar[r]^{\iota_!} \ar[d]_{\Ind^\maF}& \k^j_\mathrm{G}(F'_G) \ar[d]^{\mathrm{Ind}^{\mathcal{F}'}} \\
\k\k^j(C^*G ,C^*(M,\mathcal{F}))\; \; \ar[r]_{\underset{C^*(M, \maF)}{\otimes}\epsilon_\iota}&\; \;    \k\k^j(C^*G ,C^*(M',\mathcal{F}')).}$$
\end{thm}

\medskip

The class $\epsilon_\iota$ is a quasi-trivial Morita extension, see Section \ref{Naturality}. Another important feature of the index morphism   is the generalized reciprocity formula for closed subgroups as well as its good behaviour with respect to the restriction to a maximal torus. Assuming $G$ connected with a maximal torus $\T$, we obtain for instance the following theorem which allows to reduce the computation of the index morphism to the case of tori actions. 

\medskip

\begin{thm}
Denote by 
$r^G_\T: \k_\mathrm{G}^j(F_G)\rightarrow \k^j_{\T}(F_\T)$ the map defined in Section \ref{NaturalityInduction} using the Dolbeault operator associated with the complex {G}-structure on $G/\T$. Then for $j\in \Z_2$ the following diagram commutes: 
$$\xymatrix{\k_\mathrm{G}^j(F_G) \ar[r]^{r^G_\T} \ar[d]_{\mathrm{Ind}^{\mathcal{F}}}& \k^j_\T (F_\T)\ar[d]^{\mathrm{Ind}^{\mathcal{F}}}\\
\;\;\;\; \k\k^j(C^*G,C^*(M,\mathcal{F})) \;\;&\;\;\; \k\k^j(C^*\T,C^*(M,\mathcal{F}))\;\; \;\;  \ar[l]^{[i] \underset{C^*\T}{\otimes} \bullet }
}$$
where $[i]\in \k\k (C^*G, C^*\T)$ is the induction class.
\end{thm}

\medskip

In the last section, we  provide a topological candidate for an index theorem in our setting.   Given a $G$-embedding of $M$ in a euclidean $G$-representation $E$, we show that there exists a topological transversal $\mathcal{N}_G$ for the foliated space $(M\times T_G(E), \maF \times 0)$ together  a $\k$-oriented $G$-map $\iota: F_G\hookrightarrow \maN_G$. Hence we use the  Gysin morphism 
$$
\iota_! : \k^j_\mathrm{G }(F_G) \longrightarrow \k^j_\mathrm{G} (\mathcal{N}_G).
$$
composed  with the Morita extension morphism $\K (\mathcal{N}_G)\to  \k_j^\mathrm{G} (C^*(M\times T_G(E), \maF))$ to define the $R(G)$-morphism
$$
\k^j_\mathrm{G} (F_G) \longrightarrow  \k_j^\mathrm{G} \left( C_0(T_G(E))\otimes C^*(M, \maF)\right).
$$
Now, the topological index morphism is obtained by composition of this morphism with the Dirac morphism defined in \cite{Kasparov:KKindex} on $E$, this latter step is a replacement for the Bott periodicity in the elliptic case. 
Even in the case of closed fibrations, this topological construction is new and completes the bivariant approach to the families Atiyah problem that was investigated in \cite{baldare:KK}. \\

In order to keep this paper in a reasonnable size, we have chosen to restrict ourselves to {\em{compact}} Lie group actions although the interested reader can easily check that most of the constructions are immediately extendable to the setting of proper cocompact   actions, as carried out in  \cite{Kasparov:KKindex} for the top-dimensional case. Also, the higher distributional approach will be dealt with in a forthcoming paper where we also develop the cohomological viewpoint in the spirit of \cite{BV:ChernCharactertransversally, BV:IndEquiTransversal}, using Haefliger cohomology and results from \cite{BenameurHeitschI,BenameurHeitschICorrigendum, BenameurHeitschJDG}. We point out that these latter cohomological results have already been carried out by the first author in the case of closed fibrations in \cite{baldare:H}, where the family Berline-Vergne formula was obtained.  \\

\medskip

{\em{Acknowledgements.}} The authors  wish to thank A. Carey, P. Carrillo-Rouse, T. Fack, J. Heitsch, M. Hilsum, P. Hochs, Y. Kordyukov,  V. Mathai, H. Oyono-Oyono, V. Nistor, S. Paycha, M. Puschnigg,  A. Rennie  and G. Skandalis 
for many helpful discussions. 
Part of this work was realized during the postdoctoral position of the first author in the {\em{Institut Elie Cartan de Lorraine}} at Metz, he is indebted to the members of the noncommutative geometry team for the warm hospitality. The second author would like to thank his colleagues in Montpellier, and more specifically  P.-E. Paradan, for several interesting conversations around transversally elliptic operators and their  applications. Both authors thank the French National Research Agency for the financial support via the ANR-14-CE25-0012-01 (SINGSTAR).

\bigskip

\section{Some preliminaries about actions on foliations}

{We gather in this first section many classical results about group actions on foliations that will be used in the sequel. Most of them are well known to experts.}

\subsection{{Holonomy actions and Hilbert $G$-modules}}\label{C(B):module}
This first paragraph  is devoted to a brief review of some standard results.
For most of the classical properties of  Hilbert $C^*$-modules and regular operators between them, we refer the reader to  \cite{lance1995hilbert} and \cite{skandalis:cours:M2}.
The constructions given below extend the standard ones, see for instance  \cite{baldare:KK,baldare:H,fox1994index,hilsum2010bordism,julg1988indice}.
Our hermitian scalar products will always be linear in the second variable and anti-linear in the first.
Let $G$ be a compact group with a fixed bi-invariant Haar measure $dg$. The   convolution $*$-algebra $L^1(G)$ is defined as usual with respect to the rules
$$
(\varphi \psi) (g):=\int_{g_1\in G} \varphi (g_1) \psi (g_1^{-1}g) dg_1\text{ and } \varphi^* (g) := {\overline{\varphi (g^{-1})}}.
$$
We   denote by $C^*G$ the $C^*$-algebra associated with $G$, which is the operator-norm closure of the range of $L^1(G)$ in the  bounded operators on $L^2(G)$. A classical construction shows that any finite-dimensional unitary representation of  $G$ naturally identifies with a finitely generated projective module on  $C^*G$ \cite{julg1981produit:croise}. There is a well defined  action of $G$ by automorphisms of the $C^*$-algebra $C^*G$ which is induced by the adjoint action on $C(G)$ given by
$$
(Ad_g \varphi ) (k) := \varphi (g^{-1} k g),\quad \text{ for } \varphi \in C(G)\text{ and } g, k\in G.
$$

Let now $M$ be a smooth compact manifold  and let $\mathcal{F}$ be a given smooth  foliation of $M$. We assume that $G$ acts smoothly on $M$ by leaf-preserving diffeomorphisms, so any element $g\in G$ preserves  each leaf of $(M, \maF)$. We denote by $F$ the subbundle of $TM$ composed of all the vectors tangent to the leaves of $\mathcal{F}$, this is the tangent bundle of our foliation and its dual bundle is the cotangent bundle of the foliation and is denoted  as usual by $F^*$.  We fix a $G$-invariant riemannian metric on $M$ so that $G$ acts by isometries of $M$, and so that we can  identify  $F^*$ with a  $G$-subbundle of $T^*M$ when needed. We denote by $\mathcal{G}$ the holonomy groupoid that will be confused with the manifold of its arrows. We assume for simplicity that $\maG$ is Hausdorff  so that $M=\maG^{(0)}$ can be identified with a closed subspace (and a submanifold) of  $\maG$.  We denote as usual by $r$ and $s$ respectively the range and source maps of $\maG$ and by $\mathcal{G}_x:=s^{-1}(x)$ and $\mathcal{G}^x:=r^{-1}(x)$.  
The compact group $G$ acts obviously on $\maG$ by groupoid diffeomorphisms, hence $r$ and $s$ are $G$-equivariant submersions. The $G$-invariant riemannian metric induces a $G$-invariant riemannian metric on the leaves, {which in turn induces  a $G$-invariant  leafwise Lebesgue measure. This allows to define our $G$-invariant Haar system $\nu$ on $\maG$.} More precisely, on each holonomy cover $s:\maG^x:=r^{-1} (x)\to L_x$ of the leaf $L_x$ through $x\in M$, we have the well defined ''pull-back'' measure $\nu^x$, see for instance \cite{Connes:integration:non:commutative}. The family $\nu^\bullet:=({\nu^x})_{x\in M}$ is then easily seen to be a (continuous and even smooth) Haar system for $\maG$ in the sense of \cite{renault2006groupoid}. Similarly we may define the  measure $\nu_x$ on the holonomy cover $r:\maG_x\to L_x$ but this latter can also be seen as the image of $\nu^\bullet$ under the diffeomorphism $\gamma\mapsto \gamma^{-1}$. The $G$-invariance of the Haar system means that $g_*\nu^x = \nu^{gx}$ for any  $(g,x)\in G\times M$ or said differently, that for any $f\in C_c(\maG)$ one has
$$
\int_{\maG^{gx}} f(\gamma) d\nu^{gx} (\gamma) = \int_{\maG^{x}} f(g \gamma) d\nu^{x} (\gamma).
$$
The space $C_c(\maG)$ of compactly supported continuous functions on $\maG$ is endowed with the usual structure of an involutive convolution algebra for the rules
$$
(f_1f_2) (\gamma) := \int_{\gamma_1\in \maG^{r(\gamma)}} f_1 (\gamma_1) f_2 (\gamma_1^{-1} \gamma) d\nu^{r(\gamma)} (\gamma_1)\quad \text{ and }\quad f^*(\gamma) := {\overline{f(\gamma^{-1})}}.
$$
Moreover, for any given $x\in M$, we have a $*$-representation $\lambda_x: C_c(\maG) \rightarrow \maL (L^2(\maG_x))$ given by
$$
\lambda_x (f) (\xi) (\gamma)\; := \;\int_{\gamma_2\in \maG_x} f(\gamma\gamma_2^{-1}) \xi (\gamma_2) d\nu_x (\gamma_2).
$$
The completion of $C_c(\maG)$ in the direct sum representation $\oplus_{x\in M} \lambda_x$ is then a well defined $C^*$-algebra called the Connes algebra of the foliation $(M, \maF)$ and denoted $C^*(M, \maF)$, see \cite{Connes:integration:non:commutative} for more details.

Let $\pi : E=E^+\oplus E^- \rightarrow M$ be a $\Z_2$-graded hermitian vector bundle on $M$ which is assumed to be $G$-equivariant with a $G$-invariant hermitian structure. Then, there is a classical $G$-equivariant Hilbert $C^*(M,\mathcal{F})$-module $\maE$ associated with $E$, which is composed of sections of the $\maG$-equivariant continuous field of Hilbert spaces $(L^2(\maG_x, r^*E))_{x\in M}$ and that we now recall for the sake of completeness. 



Setting for $\eta \in C_c(\mathcal{G},r^*E)$ and $f\in C_c(\mathcal{G})$
$$
\eta \cdot f(\gamma)=\int_{\mathcal{G}^{r(\gamma)}}\eta(\gamma_1)f(\gamma_1^{-1}\gamma)d\nu^{r(\gamma)}(\gamma_1),
$$ 
we get  a right $C_c(\mathcal{G})$-module structure on $C_c(\mathcal{G},r^*E)$. The prehilbertian structure of this module is obtained by using the $C_c(\mathcal{G})$-valued scalar product  given by
$$
\langle \eta ',\eta\rangle (\gamma)=\int_{\mathcal{G}_{r(\gamma)}}\langle \eta '(\gamma_1) , \eta(\gamma_1\gamma)\rangle_{E_{r(\gamma_1)}}d\nu_{r(\gamma)}(\gamma_1),\quad \text{for } \eta, \eta '\in C_c(\mathcal{G},r^*E)\ \mathrm{and}\ \gamma \in \mathcal{G}.
$$
That $\langle \eta, \eta\rangle $ is a non-negative element of the $C^*$-algebra is standard. Moreover, all the axioms for a prehilbertian module are satisfied.
The completion of $C_c(\mathcal{G},r^*E)$ for  $\|\bullet\|_{\mathcal{E}}:=\|\langle \bullet,\bullet\rangle\|^{1/2}_{C^*(M,\mathcal{F})}$ is our Hilbert $C^*(M,\mathcal{F})$-module $\mathcal{E}$.\\

Our goal now is to use the $G$-action on $(M, \maF)$ and $E$ in order to define a representation $\pi$ of the $C^*$-algebra $C^*G$ in adjointable operators of the Hilbert module $\maE$. 
An easy inspection of the case  of  simple foliations  shows that an extra compatibility condition between the action of $G$ and the foliation $\maF$  needs to be imposed. Roughly speaking, we need  an action of $G$ which preserves each Hilbert space $L^2(\maG_x, r^*E)$ so that the average representation of $C^*G$ would make sense. We proceed now to explain this action which is taken from \cite{BenameurHeitschlefschetz}. 
Recall the action groupoid $M\rtimes G$, which as a space of arrows is just $M\times G$, with the rules
$$
s(x, g)=x,\; r(x,g) =gx \text{ and } (gx, k) \circ (x, g) = (x, kg).
$$

\begin{definition}\cite{BenameurHeitschlefschetz}\
\begin{itemize}
\item A leafwise diffeomorphism $f$ is called a holonomy diffeomorphism if there exists a smooth
map $\theta^f: M\to \maG$ so that for any $x\in M$,
$$
s(\theta^f(x)) = x\text{ and } r(\theta^f(x)) = f(x),
$$
 and the holonomy along $\theta^f (x)$ coincides with the induced action of $f$ on transversals. 
 \item The action of the compact group $G$ is a holonomy action if any $g\in G$ is a holonomy diffeomorphism and the induced  groupoid morphism 
 $$
 \theta:M\rtimes G \longrightarrow  \maG,
 $$
 given by $\theta (x, g) = \theta^g(x)$ is smooth.
\end{itemize}
\end{definition}

{Notice that when $\theta^f$ exists, it is unique.
From the very definition,}  if $f$ is a holonomy diffeomorphism, then for any $\gamma\in \maG$, one has
$$
\theta^f (r(\gamma)) \; \gamma \; \theta^f(s(\gamma))^{-1} \; = \; f(\gamma).
$$
The  holonomy diffeomorphisms form a subgroup of the group of leaf-preserving diffeomorphisms of $M$. If for instance $f$ is a holonomy diffeomorphism, then so is $f^{-1}$ and we have $\theta^{f^{-1}} (x) = \theta^{f} (f^{-1}(x))^{-1}$. When the $G$-action is a holonomy action, we  have 
$$
g\theta^k(x) = \theta^{gkg^{-1}} (gx)\text{ and } \theta^k(gx) \; \theta^g(x) =\theta^{kg} (x).
$$
The following lemma is proved in  \cite{BenameurHeitschlefschetz}:

\begin{lem}\cite{BenameurHeitschlefschetz}
The leaf-preserving diffeomorphism $f$ is a holonomy diffeomorphism in the following cases:
\begin{enumerate}
\item
When the holonomy is trivial, and the foliation is tame {in the sense of \cite{CantwellConlon}. See also \cite{BenameurHeitschlefschetz}.} 
\item
When the foliation is Riemannian.
\item
When $f$ belongs to a connected (Lie) group which acts on $V$ by leaf-preserving diffeomorphisms.  More generally, if $f$ belongs to the path connected component of a holonomy diffeomorphism $g$  in the group of leaf-preserving diffeomorphisms. 
\item
When restricted to the saturation $sat(V^f)$ of the fixed point submanifold $V^f$, that is the union of the leaves that intersect $V^f$. 
\end{enumerate}
\end{lem}

As an obvious corollary for instance, we see that when the compact Lie group $G$ is  connected, then its leaf-preserving action is automatically a holonomy action. As for the non-foliated case, we are mainly interested, especially for the cohomological index formula, in the case of the action of a compact connected Lie group $G$. However, this assumption is not needed yet and only the holonomy assumption will be necessary.\\
{\bf{From now on, we shall assume that the leafwise $G$-action is a holonomy action.}} \\

Using {the groupoid morphism} $\theta$, we get for any $x\in M$ an action of $G$ on the manifold $\maG_x$ by setting
$$
\Phi : G\times \maG_x \longrightarrow \maG_x \text{ given by } \Phi (g, \gamma) := {\theta^g(r(\gamma))\gamma}.
$$
{Indeed, one has}
$$
{\Phi (g, \Phi (k, \gamma)) \, = \, \theta^g (k r(\gamma)) \Phi (k, \gamma) =  \theta^g (k r(\gamma)) \theta^k (r(\gamma)) \gamma = \theta^{gk} (r(\gamma)) \gamma \, = \, \Phi (gk, \gamma).}
$$
The holonomy covering map $r:\maG_x\rightarrow L_x$ is then $G$-equivariant, so that $\Phi$ can be understood as an $r$-lift of the original $G$ action which fixes the source map $s$. Using the $G$-invariance of the leafwise Lebesgue-class measure, it is then easy to check using the definition of the measure $\nu_x$ that this latter is $\Phi$-invariant, i.e.
$$
\int_{\mathcal{G}_x}f(\Phi(g,\gamma))d\nu_x(\gamma)=\displaystyle\int_{\mathcal{G}_x}f(\gamma ) d\nu_{x}(\gamma).
$$
{Indeed, this follows from the relation $\Phi(g,\gamma)= (g\gamma) \; \theta^g(s(\gamma))$.}
We can now  define our unitary $G$-action $U_x$  on the Hilbert space $L^2(\maG_x, r^*E)$ by setting
$$
(U_{x,g} \eta )(\gamma):=g \, \eta(\Phi (g^{-1}, \gamma)),\quad  \eta \in C_c(\mathcal{G}_x,r^*E),\  g\in G, \text{ and } \gamma \in \mathcal{G}_x.
$$
The  family $U=(U_x)_{x\in M}$ actually represents the group $G$ in the  unitary adjointable operators on the Hilbert module $\maE$. More precisely:

\begin{lem}\label{lem:unitary:U}
For the trivial action of $G$ on $C^*(M,\mathcal{F})$, the Hilbert module $\mathcal{E}$ is a $G$-Hilbert module. {Indeed, 
for any $\eta,\eta'\in C_c(\maG, r^*E)$ and $g\in G$ we have: 
$$
\langle U_g \eta ,\eta' \rangle =\langle\eta , U_{g^{-1}} \eta'\rangle,
$$ 
so in particular, the operator $U_g$ extends to an adjointable (unitary) operator on the Hilbert module $\maE$.} 
\end{lem}

\begin{proof}
{For  $\eta\in \maE$, $f\in C^*(M, \maF)$ and $g\in G$, the relation $
(U_g\eta)\cdot f = U_g (\eta\cdot f)$
can be easily verified by direct computation, however this will be automatically satisfied since the operator $U_g$ is adjointable. More precisely, }we have 
{$$
\langle U_g \eta , \eta'\rangle (\gamma)=\int_{\mathcal{G}_{r(\gamma)}}\left\langle g\eta\left(\theta^{g^{-1}}(r(\gamma'))\gamma'\right), \eta'(\gamma'\gamma)\right\rangle_{E_{r(\gamma')}} d\nu_{r(\gamma)}(\gamma'). 
$$}
{
Setting $\gamma_1=\theta^{g^{-1}}(r(\gamma'))\gamma'$ and using the  $G$-invariance of the metric on $E$ as well as the $\Phi$-invariance, we get:
\begin{eqnarray*}
\langle U_g \eta , \eta' \rangle (\gamma) 
&=& \int_{\mathcal{G}_{r(\gamma)}}\left\langle \eta (\gamma_1), g^{-1} \eta' \left(\theta^{g}(r(\gamma_1))\gamma_1 \gamma\right)\right\rangle_{E_{r(\gamma_1)}} d\nu_{r(\gamma)}(\gamma_1)\\
& = & \langle  \eta , U_{g^{-1}} \eta' \rangle (\gamma).
\end{eqnarray*}
}
\end{proof}

\begin{cor}\label{cor:G:C(M,F):interne}
The $G$-action on  the foliation $C^*$-algebra $C^*(M, \maF)$  is  {inner}, i.e. it is implemented by unitary multipliers and we have $\lambda (g f)=U_g\circ \lambda (f) \circ U_{g^{-1}}$ for any $f\in C_c(\maG)$. 
\end{cor}

\begin{proof}
 When $E$ is the trivial line bundle, the previous lemma  shows that $G$ acts by unitary multipliers $(U_g)_{g\in G}$ of the foliation $C^*$-algebra $C^*(M, \maF)$. We now compute
{
\begin{eqnarray*}
\left(U_g\circ \lambda (f)\circ U_{g^{-1}}\right) (\xi) (\gamma) & = & g\, \left[ \lambda(f) (U_{g^{-1}}\xi)\right] \left( \theta^{g^{-1}}(r(\gamma))\gamma\right)\\
& = & \int_{\maG_{s(\gamma)}} f\left( \theta^{g^{-1}}(r(\gamma))\gamma \gamma_1^{-1}\right) \xi \left( \theta^{g}(r(\gamma_1)\gamma_1\right) d\nu_{s(\gamma)} (\gamma_1)\\
& = & \int_{\maG_{s(\gamma)}} f\left(\theta^{g^{-1}}(r(\gamma))\gamma \gamma_1^{-1} \theta^{{g}}({g^{-1}} r(\gamma_1))\right) \xi (\gamma_1) d\nu_{s(\gamma)} (\gamma_1),
\end{eqnarray*}
where the last equality is obtained by using again the $\Phi$-invariance of the Haar system. 
The result follows since 
$$
\theta^{g^{-1}}(r(\gamma))\gamma\gamma_1^{-1}\theta^{{g}}({g^{-1}} r(\gamma_1))= g^{-1}\gamma\theta^{g^{-1}}(s(\gamma)) \theta^g(g^{-1} s(\gamma)) g^{-1}\gamma_1^{-1} = g^{-1}(\gamma\gamma_1^{-1}).
$$ 
}
\end{proof}

\begin{prop}\label{prop:pi*rep} 

We set for $\eta\in C_c(\maG, r^*E)$  and  $ \varphi \in C(G)$:
\begin{equation}\label{remarque:piGequi}
\pi(\varphi)(\eta)\; :=\; \int_G\varphi(g)\, (U_g\eta)\,  dg.
\end{equation}
Then $\pi$ extends to an involutive representation $\pi $ of $C^*G$ in the Hilbert module $\mathcal{E}$. More precisely, $\pi$ is  a continuous $*$-homomorphism into the $C^*$-algebra of adjointable operators. 
\end{prop}


\begin{proof}
Since $U_g$ is adjointable with $U_g^*=U_{g^{-1}}$, we obtain that $\pi(\varphi)$ is also adjointable with $\pi(\varphi^*)=\pi(\varphi)^*$.
The relation $\pi(\varphi \star \psi)=\pi(\varphi)\circ \pi(\psi)$ is also immediately verified.
It follows that $\pi$ is a $\ast$-homomorphism {which satisfies, by its very definition, the estimate $\|\pi(\varphi)\|\leq \|\varphi\|_{L^1G}$ for $\varphi\in L^1G$}. Hence we get a well defined continuous $*$-representation of the $C^*$-algebra $C^*G$. 
\end{proof}

{
\begin{remarque}
If we endow $C^*G$ with the conjugation action $Ad$ of $G$, then it is easy to check that the representation $\pi$ is  $G$-equivariant, i.e. $
\pi (Ad_g\varphi) (\eta) = U_g\circ \pi (\varphi)\circ U_{g^{-1}}$ for $\varphi\in C^*(G)$, $\eta\in \maE$ and $g\in G$. 
\end{remarque}}

\begin{defi}\cite{Connes:integration:non:commutative}
Let $E=E^+\oplus E^-$ be a $\Z_2$-graded vector bundle over $M$. A (classical) pseudodifferential $\mathcal{G}$-operator $P$ of order $m$ acting from $E^+$ to $E^-$ is a {smooth} family $(P_x)_{x\in M}$, where 
$$
P_x: C_c^{\infty}(\mathcal{G}_x,r^*E^+)\longrightarrow C_c^{\infty}(\mathcal{G}_x,r^*E^-),
$$ 
is a (uniformly supported and classical) pseudodifferential operator of order $m$, with the right  $\mathcal{G}$-invariance property:
$$
P_{r(\gamma)}R_\gamma = R_\gamma P_{s(\gamma)}.
$$ 
\end{defi}

The uniform support is assumed here for simplicity and proper support would suffice  in order to preserve the space of compactly supported sections, see \cite{NWX}. We shall denote by $\Psi^m(M, \maF; E^+,E^-)$ the space of (classical) pseudodifferential $\maG$-operators on $M$ of order $m$. So such pseudodifferential $\maG$ operators correspond to  longitudinal pseudodifferential operators on the graph manifold $\maG$ with respect to the foliation $r^*\maF$, but which are $\maG$-invariant so that they induce operators downstairs acting over the leaves of $(M, \maF)$. We shall also sometimes call the elements of $\Psi^m(M, \maF; E^+,E^-)$ longitudinal or leafwise pseudodifferential operators on $(M, \maF)$ since no confusion can occur. 

The principal symbol of such a longitudinal operator $P$ of order $m$ is defined as usual by the formula:
$$
\sigma_m(P)(x,\xi)= \sigma_{pr}(P_x)(x,\xi),\  \text{ for } (x,\xi)\in T^*_x\maG_x\simeq F^*_x,
$$
where $\sigma_{pr}(P_x)$ is the principal symbol of the $m$-th order classical pseudodifferential operator $P_x$ acting on the manifold $\mathcal{G}_x$. {When $E$ is a $G$-equivariant $\Z_2$-graded hermitian vector bundle,  the longitudinal operator $P$ will be $G$-invariant  if  for any $g\in G$, the family $P$ commutes with the family of unitaries $U_{g, \bullet}$, i.e. 
$$
U_{g, x} \circ P_{x}=P_x\circ  U_{g, x}, \quad \forall g, x.
$$
In this case, the principal symbol of $P$ is $G$-invariant with the action on the leafwise cotangent bundle $F^*$  obtained as usual by codifferentiating the original $G$-action. }

A zero-th order longitudinal pseudodifferential operator $P_0 : C^{\infty}_c(\mathcal{G},r^*E^+) \rightarrow C^{\infty }(\mathcal{G},r^*E^-)$ extends  into an adjointable operator, still denoted $P_0$, between the Hilbert modules $\maE^+$ and $\maE^-$ corresponding to the vector bundles $E^+$ and $E^-$ respectively \cite{ConnesSkandalis, Connes:integration:non:commutative}. The formal adjoint of $P_0$ defined over each $\maG_x$, with respect to the hermitian structures and the Haar system, is then again a zero-th order longitudinal pseudodifferential operator acting from $E^-$ to $E^+$.  Moreover, its extension to an adjointable operator from $\maE^-$ to $\maE^+$ is just the adjoint of $P_0$ with respect to the Hilbert module structures. So,  if we denote by $P$ the operator
$P:=\begin{pmatrix}
0&P_0^*\\P_0&0
\end{pmatrix}$, then $P$ is an adjointable operator on $\maE=\maE^+\oplus \maE^-$ which is by construction odd for the $\Z_2$-grading.

\begin{lem}\label{lem:Ppi=piP} With the previous notations, if we assume in addition that $P_0$ is  $G$-invariant, then for any $\varphi \in C^*G$, we have $[\pi(\varphi),P]=0$.
\end{lem}

\begin{proof}
{As can be checked easily, the operator $P$ is $G$-invariant in the usual sense if and only if $P$ commutes with the unitary $U$ of $\maE$ corresponding to the family of unitaries $(U_{g,x})_{(g, x)\in G\times M}$. Now, let $\varphi\in L^1(G)$, then by definition of $\pi (\varphi)$ we deduce that $\pi (\varphi)\circ P= P\circ \pi (\varphi)$.
Therefore  this commutation relation also holds for any $\varphi\in C^*G$ by continuity. }
%
\end{proof}

\subsection{The moment map and some standard $G$-operators}\label{section:moment:map}

 Assume now that $G$ is a compact {\em Lie} group with Lie algebra $\g$, and that the action of $G$ on $M$ preserves the leaves and is through holonomy diffeomorphisms as explained in the previous section. This is for instance the case for any compact connected Lie group. 
We start by extending some results from \cite[Section 6]{Kasparov:KKindex} to our foliation setting, and for the convenience of the reader we shall use Kasparov's notations from there. For $x\in M$, we hence denote by $f_x : G \rightarrow M$ the map given by  $f_x(g)=g\, x$ and by $f'_x : \g \rightarrow T_xM$ its tangent map at the neutral element of $G$. The dual map of  $f'_x$ is denoted  by $f_x^{'*} : T^*_xM \rightarrow \g^*$.
So, any $X\in \g$ defines as usual the vector field $X^*$ given by $X^*_x:=f'_x (X)$ which,  under our assumptions, is  tangent to the leaves, i.e. $X^*_x\in F_x$ for any $x\in M$. Notice also that 
$g\cdot f_x'(X)=f'_{g\, x}(\text{Ad}(g)X)$, for  $g\in G$, $x\in M$ and $X\in \g$ \cite{Kasparov:KKindex}. 
%

Let $\g_M:=M\times \g$ be the $G$-equivariant trivial bundle of Lie algebras on $M$, associated to $\g$ for the action $g\cdot (x,v)=(g\, x,\text{Ad}(g)v)$. The map $f' :\g_M \rightarrow TM$ defined by $f'(x,v)=f'_x(v)$ is a $G$-equivariant vector bundle morphism.
We  endow $\g_M$ with a $G$-invariant metric and we denote by $\|\cdot \|_x$ the associated family of Euclidean norms.  Up to normalization, we can always assume that $\forall v\in \g$, $\|f'_x(v)\|\leq \|v\|_x$.  Here $\|f'_x(v)\|$ is the norm given by the riemannian metric at $x$. We thus assume from now on that  $\|f'_x\|\leq 1$, $\forall x\in M$. These metrics on $\g_M$ and $TM$ are also used to identify $\g_M$ with $\g_M^*$ and $TM$ with $T^*M$.  Then we can define the map $\phi : T^*M \rightarrow T^*M$ by setting $\phi_x=f'_xf^{'*}_x$. Again according to Kasparov's notations \cite{Kasparov:KKindex}, we introduce the quadratic form $q=(q_x)_{x\in M}$ on the fibers of $T^*M$ by setting: 
\begin{equation}\label{Def-q}
q_x(\xi)=|\langle f'_xf^{'*}_x(\xi),\xi\rangle |=\|f^{'*}_x(\xi)\|_x^2,~\forall (x,\xi)\in T^*M.
\end{equation}
If  $\xi \in T^*_xM$, then it is easy to see that  $ \xi $ is orthogonal to the $G$-orbit   of $m$ if and only if $q_x (\xi) = 0$. Notice also that we 
 have $q_x(\xi)\leq \|\xi\|^2$.

As in the seminal book \cite{atiyah1974elliptic}, we introduce a second order $G$-invariant longitudinal differential operator  $\Delta_G$ whose symbol coincides with $q$. This is achieved for instance by using an orthonormal basis of $\g$ for a bi-invariant metric on the compact Lie group $G$ and by considering the first order differential operators which are the Lie derivatives of the $G$-action, see again  \cite[page 12]{atiyah1974elliptic}. Recall that 
if $X\in \g$ and $\eta\in C^{\infty }(\mathcal{G},r^*E)$, then the Lie derivative  ${\mathscr L} (X)(\eta)$  is defined as
$$
{\mathscr L} (X)(\eta)(\gamma) := \frac{d}{dt}{\vert_{t=0}} (e^{tX}\cdot \eta)(\gamma)= \frac{d}{dt}{\vert_{t=0}} e^{tX}\left(\eta(\Phi(e^{-tX},\gamma)\right).
$$
So, $\mathscr{L} (X)$ preserves each space $C_c^{\infty }(\mathcal{G}_x,r^*E)$ and the corresponding family of first order differential operators is clearly right $\maG$-invariant. 
Note indeed that  $
\Phi(g,\gamma'\gamma)=
\Phi(g,\gamma')\gamma.$ Therefore, 
\begin{eqnarray*}
R_\gamma(\mathscr{L}(X)_{s(\gamma)}\eta)(\gamma')&=& \mathscr{L}(X)_{s(\gamma)}\eta(\gamma'\gamma)\\
&=&\dfrac{d}{dt}{\vert_{t=0}}e^{tX}(\eta(\Phi(e^{-tX},\gamma'\gamma))\\
&=& \dfrac{d}{dt}{\vert_{t=0}}e^{tX}(\eta(\Phi(e^{-tX},\gamma')\gamma)\\
&=&\dfrac{d}{dt}{\vert_{t=0}}e^{tX}\big(R_\gamma(\eta)(\Phi(e^{-tX},\gamma'))\big)\\
&=&\mathscr{L}(X)_{r(\gamma)}R_\gamma(\eta)(\gamma').
\end{eqnarray*}
Then, for any orthonormal basis $\{V_k\}$ of $\g$ with dual basis $\{v_k\}$, we define a longitudinal differential operator $d_G$ by considering the right $\maG$-invariant family 
$$
d_G = \left(d_{G, x}: C_c^{\infty }(\mathcal{G}_x,r^*E) \longrightarrow C_c^{\infty}(\mathcal{G}_x,r^*(E\otimes \g_M^*))\right)_{x\in M}
$$ 
of differential operators between $E$ and $E\otimes \g_M^*$ given by
$$
d_{G, x} (\eta):=\sum \limits_k \mathscr{L}(V_k) \eta \otimes v_k, \quad \forall \eta\in C_c^{\infty }(\mathcal{G}_x,r^*E).
$$
This definition is independent of the choice of the orthonormal basis of $\g$. 
%

\begin{remarque}\label{DeltaG}
We may take for $\Delta_G$ the operator $d^*_Gd_G$. Indeed, it is easy to see that the symbol of $d_G$ at $(x,\xi) \in F$ is given by $\sqrt{-1}~ \ext(f'^{*}_{x}(\xi))$. So the symbol of $d_G^*$ is given by $-\sqrt{-1}~\interieur (f'^{*}_{x}(\xi))$ and hence the principal symbol of $d_G^*d_G$ is given by $\langle f'^{*}_{x}(\xi),f'^{*}_{x}(\xi)\rangle =q_x(\xi)$. 
\end{remarque}

In the same way and working on the manifold $G$ itself with its $G$-action  by left translations, the orthonormal basis $\{V_k\}$ of $\g$ (with dual basis $\{v_k\}$) allows to define the exterior differential of the manifold $G$ as follows.  For any $\varphi\in C^\infty (G)$, let $\dfrac{\partial\varphi}{\partial V}$ be the derivative along the one-parameter subgroup of $G$ corresponding to the vector $V\in \g$, then we get the first-order differential operator $d$ acting on smooth functions on $G$ and valued in  $\g^*$-valued smooth functions on $G$, by setting
$$
d\varphi = \sum_k \dfrac{\partial\varphi}{\partial V_k}\otimes v_k.
$$
We may tensor the representation $\pi: C^*G \to \maL_{C^* (M, \maF)} (\maE)$ with the identity of the vector space $\g^*$ and get the extended map
$$
\pi: C^*G \otimes \g^* \longrightarrow  \maL_{C^* (M, \maF)} (\maE, \maE \otimes \g^*)\simeq \maL_{C^* (M, \maF)} (\maE) \otimes \g^*.
$$
Said differently, we simply set for $\psi\in L^1(G)$ and $v\in \g^*$:
$$
\pi(\psi\otimes v)\eta=\pi (\psi)\eta \otimes v =  \int_G \psi (g) (U_g \eta \otimes v) \; dg, \quad \forall \eta \in \maE.
$$
Here again the map $\pi$ corresponds to  a family $(\pi_x)_{x\in M}$ of maps
$$
\pi_x:  C^*G \otimes \g^* \longrightarrow  \maL (L^2(\maG_x, r^*E))\otimes \g^*. 
$$

\begin{prop}$\cite{Kasparov:KKindex}$\label{prop:dG}
For $\varphi \in C^{\infty}(G)$ and $V\in \g$, we have 
$$
\mathscr{L} (V) \circ \pi (\varphi) = \pi \left(\dfrac{\partial\varphi}{\partial V}\right).
$$
In particular, $d_G(\pi(\varphi ) \eta)=\pi(d\varphi )\eta$ for any $\eta\in C_c^{\infty}(\mathcal{G},r^*E)$, or equivalently $\left(d_{G, x} [\pi_x(\varphi)] =\pi_x (d\varphi)\right)_{x\in M}$.
\end{prop}

\begin{proof}
We  only need to check the first relation with the Lie derivatives. But we have for $V\in \g$, $\varphi\in C^\infty (G)$ and $\eta\in C_c^\infty (\maG, r^*E)$:
\begin{eqnarray*}
\mathscr{L}(V)\pi(\varphi)\eta(\gamma)&=&  \dfrac{d}{dt}{\vert_{t=0}} \; e^{tV}\big(\pi(\varphi)\eta(\Phi(e^{-tV},\gamma))\big)\\
& = & \dfrac{d}{dt}{{\vert_{t=0}}} \;\int_G \varphi(g) (e^{tV}g)\big(\eta(\Phi(g^{-1}, \Phi (e^{-tV},\gamma)))\big) dg\\
&=& \dfrac{d}{dt}{{\vert_{t=0}}} \;\int_G \varphi(g) (e^{tV}g)\big(\eta(\Phi(g^{-1}e^{-tV},\gamma))\big) dg\\
& = & \int_G \dfrac{d}{dt}{\vert_{t=0}} \;\varphi (e^{-tV} h)  h \big(\eta(\Phi(h^{-1},\gamma))\big) dh.
\end{eqnarray*}
In the second to third line we have used the relation $\Phi(g^{-1}, \Phi (e^{-tV},\gamma))= \Phi (g^{-1}e^{-tV}, \gamma)$, and in  the last equality, we have substituted $e^{tV_k}g=h$ and used the $G$-invariance of the Haar measure on $G$. Therefore, we get
$$
\mathscr{L}(V)\pi(\varphi)\eta(\gamma)= \int_G \dfrac{\partial\varphi}{\partial V} (h)   U_h \eta (\gamma) dh = \left[\pi \left(\dfrac{\partial\varphi}{\partial V}\right) \eta\right] (\gamma).
$$
\end{proof}

 Recall that we are given for any $g\in G$ and any $x\in M$ a holonomy class $\theta^g(x)\in \maG_x^{g x}$ with the natural properties recalled in the previous section. So, for any $X\in \g$, we have 
 $$
 \theta^{e^{-tX}} (x) \; \in \; \maG_x^{e^{-tX}x}
 $$
and  $t\mapsto \theta^{e^{-tX}} (x)$ is a smooth path in $\maG_x$ which starts at $x$ viewed in $\maG_x$. Therefore, we defined a vector $\tilde{X}_x\in T_x \maG_x$ and hence in $F_x$ by setting 
$$
\tilde{X}(x)\; :=\; \dfrac{d}{dt}{\vert_{t=0}}\; \theta^{e^{-tX}} (x).
$$
An easy inspection in a local chart allows to see that 
the vector field $\tilde{X}$ coincides with the vector field $X^*_M$. 

%
%

\medskip

\section{The index morphism}\label{Section.Index}

{In this section we define the index class of a $G$-invariant  leafwise $G$-transversally elliptic operator  and we also introduce the $\k$-multiplicity of any unitary irreductible representation in the index class.}

\subsection{Leafwise $G$-transversally elliptic operators}
%

{Recall the bundle map $f':M\times \mathfrak{g}\to F$ and its fiberwise transpose ${f'}^*$.
\begin{defi}
Denote by $F_G^*\subset F^*$ the kernel  of ${f'}^*$. So, $F_G^*$ is the subspace of $F^*$ composed of leafwise covectors which are transverse to the $G$-orbits, or equivalently: 
$$
F_G^* := \{(x,\xi) \in F^*\text{ such that }q_x(\xi)=0\}   .
$$ 
\end{defi}
This definition extends the classical one from \cite{atiyah1974elliptic} where $T_G^*M\subset T^*M$ is defined similarly. }
{We shall use  our Riemannian metric  to identify $TM$ with $T^*M$ and also  $F^*$ with $F$.
With these identifications, $T^*_GM$ can be identified with the subspace $T_GM$ of $T(M)$ which is the orthogonal of the $G$-orbits, and $F^*_G$ can also be identified with
the subspace of $F$ {which is the leafwise orthogonal to the $G$-orbits, so}  $F_G:= F \cap T_GM$. 
}

Recall that any zero-th order longitudinal pseudodifferential operator $P_0$ gives rise to the self-adjoint operator that we have denoted {by $P$ and which is defined following the usual convention, see \cite{Kasparov:KKindex}.  More precisely, in the even case, say when $P_0$ acts from the sections of the hermitian vector bundle $E^+$ to the sections of the hermitian vector bundle $E^-$, we consider   the $\Z_2$-graded Hilbert module $\maE=\maE^+\oplus \maE^-$ associated with the $\Z_2$-grading $E=E^+\oplus E^-$, and the operator $P$ is odd for the grading and given by $P= \begin{pmatrix}
0&P_0^*\\
P_0&0
\end{pmatrix}$. 
In the ungraded case, $E^+=E^-=E$ and $P=P_0$ is assumed to be a selfadjoint operator, say the bounded extension of a  leafwise (formally) selfadjoint operator $P_0 :C^\infty_c(\maG,r^*E)\rightarrow C^\infty_c(\maG,r^*E)$. We shall refer to this convention as convention (K). The notion of $G$-invariant $G$-transversally elliptic operator was introduced and studied in 
 \cite{atiyah1974elliptic}. 
In our case of foliations, we need to assume that the principal symbol of such $G$-invariant longitudinal operator be invertible away from the ``zero section'' of $F_G$. {So by homogeneity, this means that we assume that the restriction of  the principal symbol of our operator to the subspace $S^*_G\maF$ of covectors in $F^*_G$ of length $1$, is pointwise invertible}. 
 {Following classical normalizations  (see \cite{ConnesSkandalis, Kasparov:KKindex}), we  introduce  the following simpler definition:
 \begin{defi}\label{ope:trans:elliptic:naive}
A   zero-th order $G$-invariant  longitudinal pseudodifferential operator $P_0$ acting from the sections of $E^+$ to the sections $E^-$  is a longitudinal $G$-transversally elliptic operator or  a leafwise $G$-transversally elliptic operator, if the symbol of the associated self-adjoint operator $P=\left(\begin{array}{cc} 0 & P_0^*\\ P_0 & 0\end{array}\right)$  on  $E=E^+\oplus E^-$ satisfies  the following condition
\begin{equation}\label{ope:trans:elliptic:naïve}
 \sigma(P)^2=\mathrm{id}\text{ in restriction to } S^*_G\maF.
\end{equation}
\end{defi}
}
{The principal symbol of such  leafwise $G$-transversally elliptic operator $P_0$  {then}  represents a class in the $G$-equivariant Kasparov bivariant group 
$\k\K(\mathbb{C},C_0(F_G))$ and is represented by  the Kasparov even cycle $(C_0(\pi^*E), {\sigma(P)})$, where $C_0(\pi^*E)$ is the space of continuous sections of the continuous bundle $\pi^*E \rightarrow F_G$ which vanish at infinity. Hence,  using the isomorphism $\k\K(\mathbb{C},C_0(F_G))\simeq \K (F_G)$, any $G$-invariant leafwise $G$-transversally elliptic operator $P_0$ has a symbol  class 
$$
[\sigma (P_0)]\; \in \; \K (F_G).
$$
{
\noindent
In the ungraded case, Definition \ref{ope:trans:elliptic:naive} applies to the self-adjoint operator $P=P_0$ and we get an odd Kasparov cycle and hence  a symbol class $[\sigma(P)]\in \k\k_{\mathrm{G}}^1(\C,C_0(F_G))\simeq \k_{\mathrm{G}}^1(F_G).$
}

{
We end this subsection with the following lemma which will be needed in the sequel. Item (2) was used in \cite{Kasparov:KKindex} to define the notion of $G$-transversally elliptic symbols for non classical symbols, and in the more general setting of proper actions. 
\begin{lem}\label{lem:equiv:def} \label{lem:Kasparov:eq:def}
Let $(W, \maF^W)$ be a smooth (not necessary compact) foliated  manifold, and denote as usual by $F^W$ the longitudinal (co)tangent bundle of $(W, \maF^W)$. Let $A$ be an order $0$ longitudinal operator on $W$.  Then the following are equivalent
\begin{itemize}
\item[(1)] {${\sigma(A)}(x, \xi)=0$, $\forall (x, \xi) \in F_G^W\smallsetminus 0=T_GW\cap F^W\smallsetminus 0$.}
\item[(2)] {$\forall \varepsilon >0, \forall \text{compact }K\subset W, \exists c >0,\; \| {\sigma(A)}(x, \xi)\| \leq c \dfrac{q_x(\xi)}{\|\xi\|^2} +\varepsilon$, $\forall x\in K$  and $\forall \xi\in F^{W}_x\smallsetminus 0$. }
\end{itemize} 
\end{lem}

\begin{proof}\
(2) implies (1) because $q_x(\xi)=0\, {\Longleftrightarrow}\, (x, \xi) \in F_G^W$. Let us show that (1) implies (2). {Since ${\sigma(A)}(x, \bullet)$ and $q_x$ are homogeneous, we only need to prove the relation (2) on the sphere bundle $S^*\maF^W$ of $F^W$.  Let $\ep >0$ be given and denote by $A_{\ep, K}$ the subspace of $S^*\maF^W$ of those $(x, \xi)$ such that $\|{\sigma(A)}(x, \xi)\| \geq \ep$ and $x\in K$. Then by (1), the continuous positive function $q$  is bounded below on $A_{\ep, K}$, by compacity of $A_{\ep, K}$. This shows that $\frac{\|{\sigma(A)}(\bullet)\|}{q}$ is bounded on $A_{\ep, K}$. Since $\| {\sigma(A)}(x, \xi)\|<\ep$ on  $S^*\maF^W\smallsetminus A_{\ep, K}$, the proof of (2) is complete.
}
%
\end{proof}
}

\subsection{The index class}
\noindent

We fix a $G$-invariant selfadjoint longitudinal {zero-th order} pseudodifferential operator $Q$ acting on the sections of the vector bundle $E$ and with principal symbol given  for non-zero $\xi$ by  {${\sigma(Q)}(x,\xi)=\frac{q_x(\xi)}{|\xi|^2} \times \id_E$}. This can be achieved by using for instance the usual quantization map, see  \cite{ConnesSkandalis}.

{

\begin{prop}\label{prop:inégalité:preuve:thm} 
Let $A$ be a $G$-invariant selfadjoint {leafwise  pseudodifferential operators of order $0$} acting on the sections of the {hermitian} bundle $E$ over $M$. {Assume that the principal symbol ${\sigma(A)}$ vanishes in restriction to  the subspace $F_G\smallsetminus 0$ of $G$-transverse leafwise  covectors}. Then there exist  two $G$-invariant selfadjoint compact operators $R_1$ and $R_2$ on the Hilbert module $\maE$ such that:

$$
-(c\, Q + \varepsilon +R_1)\leq A\leq c\, Q +\varepsilon +R_2 \text{  as self-adjoint operators on }\maE.
$$
\end{prop}

\begin{proof}\ 
It is a classical result for a single operator even on non compact manifolds but with the proper support that such operators $R_1$ and $R_2$ exist as smoothing properly supported operators, see  \cite{Hormander1971, shubin2001pseudodifferential, Kasparov:KKindex}. Since we shall only need the condition of compactness {of the operators} and since our ambiant manifold is compact here, the proof is immediate. Indeed, by Lemma \ref{lem:equiv:def},
 for any $\ep > 0$, there exists $c>0$ such that the principal symbol of the operators $cQ +\ep \id_{\maE} \pm A$ are non-negative as elements of the $C^*$-algebra $C(S^*\maF, \END (E))$ of continuous sections of the algebra bundle $\END (E)=\pi^*\End (E)$ over the cosphere bundle $S^*\maF$ of the longitudinal bundle $F$.  Now,  a classical result of Connes \cite{Connes:integration:non:commutative, ConnesSkandalis} gives us a $C^*$-algebra short exact sequence obtained out of the closure of the zero-th-order pseudodifferential operators along the leaves of $\maF$:
$$
0 \to \maK_{C^*(M, \maF)} (\maE) \hookrightarrow {\overline{\Psi^0 (M, \maF; E)}} \stackrel{\sigma}{\longrightarrow} C( S^*\maF, \END (E)) \to 0,
$$
where ${\overline{\Psi^0 (M, \maF; E)}}$ is the closure in $\maL_{C^*(M, \maF)} (\maE)$ of the $*$-algebra of zero-th order pseudodifferential operators along the leaves (acting on the sections of $E$) and $\sigma$ is the principal symbol map. Hence, we deduce that the operators $cQ +\ep \id_{\maE} \pm A$ are non-negative up to compact operators and hence the conclusion. 
\end{proof}
}

\begin{remarque}
{Proposition \ref{prop:inégalité:preuve:thm} can be stated for a non compact foliated manifold and for operators with compactly supported symbols in $S^*\maF$. In this case one needs to work with locally compact operators. See Proposition \ref{prop:ineq:non:compact} in Appendix \ref{ShubinLemma} where  the corresponding generalization is stated and proved  using  the exact sequence of locally compact pseudodifferential operators  as obtained in \cite{ConnesSkandalis}.}
\end{remarque}

We are now in  position to state the following important  result. 

{
\begin{thm}\label{index-cycle}\
The triple $\big( \mathcal{E},\pi,P\big)=\left[\maE^+\oplus \maE^-,\pi, \begin{pmatrix}
0&P_0^*\\
P_0&0
\end{pmatrix}\right]$ is an  even Kasparov cycle for the $C^*$-algebras $C^*G$ and $C^*(M,\mathcal{F})$. 
In  the ungraded case, we similarly have an odd Kasparov cycle which is represented by $(\maE,\pi,P_0)$.
\end{thm}
}

\begin{proof}
By Lemma \ref{lem:Ppi=piP}, we know that $[\pi(\varphi), P]=0$ for any $\varphi\in C^*G$. Moreover, $P$ is selfadjoint and odd for the $\Z_2$-grading while $\pi$ obviously respects the $\Z_2$-grading. It thus remains to check that $(\rm{id}-P^2)\circ \pi(\varphi )\in \mathcal{K}_{C^*(M,\mathcal{F})}(\mathcal{E})$, where $\mathcal{K}_{C^*(M,\mathcal{F})}(\mathcal{E})$ stands as usual for the $C^*$-algebra of compact operators in the Hilbert module $\maE$. As $P$ is a leafwise $G$-transversally elliptic operator, the principal symbol of $\rm{id}-P^2$ which coincides with $\rm{id}- {\sigma(P)}^2$, satisfies the assumption of Proposition \ref{prop:inégalité:preuve:thm}. Therefore $\forall \varepsilon>0$, there exist $c_1$, $c_2>0$ and 
compact operators  $R_1$ and $R_2$ on the Hilbert module (in fact leafwise smoothing operators) such that
$$
-(c_1Q+\varepsilon +R_1)\leq 1-P^2 \leq c_2Q+\varepsilon +R_2.
$$ 
Let us take for $\Delta_G$ the operator $d^*_Gd_G$, see Remark \ref{DeltaG}. Denote also by $\Delta$ a  $G$-invariant longitudinal second order differential operator with principal  symbol $\|\xi\|^2\times \id_E$ for $(x,\xi)\in F_x$.  Modulo longitudinal pseudodifferential operators of negative order, the longitudinal pseudodifferential operator $Q$ then coincides with the operator $d_G^*(1+\Delta)^{-1}d_G$.

By Proposition \ref{prop:dG}, we know that $d_G\circ \pi (\varphi ) =\pi(d\varphi )$ and hence this latter is a bounded operator on $\mathcal{E}$. Moreover, $d_G^*(1+\Delta)^{-1}$ has negative order, so by Corollary 3 of \cite{lauter2000pseudodiff}, it is a compact operator of the Hilbert module $\mathcal E$. It follow that $d_G^*(\mathrm{id}+\Delta )^{-1}d_G\pi(\varphi )$ is compact as well. Again since longitudinal pseudodifferential operators of negative order extend to compact operators on the Hilbert module $\maE$, we deduce that $Q\pi(\varphi)$ is compact.
In order to show that the operator $(\mathrm{id}-P^2)\circ \pi(\varphi )$ is compact, we first notice that for $\psi =\varphi^* \varphi$, we have: 
$$-\pi(\varphi )^*\circ (c_1Q+\varepsilon +R_1)\circ \pi(\varphi )\leq \pi(\varphi )^*\circ (\mathrm{id}-P^2)\circ \pi(\varphi ) \leq \pi(\varphi )^*\circ (c_2Q+\varepsilon +R_2)\circ \pi(\varphi ){,}$$
i.e.
$$- (c_1Q+\varepsilon +R_1)\circ \pi(\psi )\leq  (\mathrm{id}-P^2)\circ \pi(\psi ) \leq (c_2Q+\varepsilon +R_2)\circ \pi(\psi ),$$
since all the operators are $G$-invariant. Therefore, projecting in the Calkin algebra and letting $\epsilon\to 0$, we deduce that $(\mathrm{id}-P^2)\circ \pi(\psi )$ is compact for any non-negative $\psi\in C^*G$. Now, using the spectral theorem, we may write any $\varphi \in C^*G$ as a linear combination of non-negative elements and conclude. 
\end{proof}

\begin{defi}
The index class  {$\ind (P_0)$} of the $G$-invariant leafwise $G$-transversally elliptic operator  $P_0$ is defined as:
{
$$
\ind (P_0) \, :=\, [\mathcal{E},\pi,P]\; \in \k\k^i(C^*G,C^*(M,\mathcal{F})),
$$}
with $i\in \Z_2$ {determined} according to Convention (K). 
\end{defi}





{We end this subsection by a short explanation for this choice of the non-equivariant index class $\ind (P_0)$, in opposition to the apparently more interesting equivariant one in $\k\K(C^*G, C^*(M, \maF))$. This latter obviously exists and covers $\ind (P_0)$ under the $G$-forgetful map.  Indeed, the triple $(\maE, \pi, P)$ is $G$-equivariant, 
the representation $\pi$ of $C^*(G)$ in the adjointable operators of $\maE$   is for instance $G$-equivariant when $C^*G$ is endowed with the inner adjoint action. Now, recall that for inner actions the equivariant $\k\k$ groups are isomorphic to the ones corresponding to trivial $G$-actions.
Notice in addition that  $\pi$ allows to recover the action by the unitaries $U_g$, hence  there is no essential lost in concentrating on the non-equivariant class $\ind (P_0)\in \k\k (C^*G, C^*(M,\maF))$. We are grateful to the referee for pointing this out to us.}
%
%

\subsection{The index map}

\begin{prop}\label{G-index}
{The index class  {$\ind(P)$} only depends on the $K$-theory class $[\sigma(P)]$ of the principal symbol $\sigma (P)$, and this induces for $i=0, 1$, a group homomorphism:
{
$$
\ind : \k_{\mathrm{G}}^i(F_G) \longrightarrow \k\k^i(C^*G,C^{*}(M,\mathcal{F})).
$$
}
More precisely, the map  {$[\sigma] \mapsto \ind(P(\sigma))$} is well defined by using any  quantization $P_0$  of $\sigma$.}
\end{prop}

\begin{proof}\
This is classical and we follow \cite{Atiyah-Singer:I} and \cite{atiyah1974elliptic}. We only give the graded case, the ungraded being similar and easier. Let ${\mathcal C}(F_G)$ be the semigroup of $0$-homogeneous homotopy classes of transversally elliptic symbols of order $0$ and let ${\mathcal C}_{\phi}(F_G)\subset {\mathcal C}(F_G)$ be the  classes of such symbols whose restriction to the sphere bundle of $F_G$, is induced by a bundle isomorphism over $M$. By a standard argument, we know that $\K(F_G):={\mathcal C} (F_G)/{\mathcal C}_{\phi}(F_G)$. Let now $\sigma_t$ be a homotopy of leafwise $0$-th order $G$-transversally elliptic symbols, then the quantization of this homotopy gives an operator homotopy and hence  by the very definition of the Kasparov group  {$\k\k(C^*G,C^{*}(M,\mathcal{F}))$}, the index classes of $\sigma_0$ and $\sigma_1$ coincide  in  {$\k\k(C^*G,C^{*}(M,\mathcal{F}))$}. On the other hand, given two $0$-th order $G$-invariant leafwise $G$-transversally elliptic operators $P : C^{\infty ,0}(\mathcal{G},r^*E) \rightarrow C^{\infty ,0}(\mathcal{G},r^*E)$ and $P': C^{\infty ,0}(\mathcal{G},r^*E') \rightarrow C^{\infty ,0}(\mathcal{G},r^*E')$, we obviously have 
$$
\ind (P\oplus P' )=[(\mathcal{E}\oplus \mathcal{E}', \pi_{\mathcal{E}}\oplus \pi_{\mathcal{E}'} ,P\oplus P')]=[(\mathcal{E}, \pi_{\mathcal{E}} ,P)]+[(\mathcal{E}',\pi_{\mathcal{E}'}, P')]=\ind (P)+\ind (P').
$$
Finally, it is clear that  any zero-th order $G$-invariant longitudinal pseudodifferential operator whose symbol  is induced by a bundle isomorphism over $M$ has $\ind(P)=0$, for more details  see for instance \cite{Atiyah-Singer:IV}.
\end{proof}
{We shall also use in the sequel  the classical isomorphism  $\K(F_G)\simeq ^k{\mathcal C} (F_G)/^k{\mathcal C}_{\phi}(F_G)$ for any $k\in \Z$. Here $^k{\mathcal C}(F_G)$ and $^k{\mathcal C}_{\phi}(F_G)$ are the same semi-groups introduced in the previous proof but replacing $0$-homogeneous by $k$-homogeneous. See for instance \cite{Atiyah-Singer:I}.}
%
%

\begin{remarque}
{\begin{itemize}
\item Since the commutator $[P, f]$ is compact for $f \in C(M)$, the triple $(\maE, \pi , P)$ extends to an
element in $KK(C(M) \rtimes  G  , C^*(M,\maF))$. Therefore, the index morphism ${\Ind}^{\mathcal{F}}$ factors through $\Ind^{M, \maF}: \k^i_{\mathrm{G}} (F_{G}) \to \k\k^i (C(M)\rtimes G, C^*(M, \maF))$. 
\item {When the $G$-action is locally free so as to generate a smooth regular subfoliation $\maF'$ of the foliation $\maF$, the index morphism $\Ind^{M, \maF}$ can be recast as valued in $\k\k (C^*(M, \maF'), C^*(M, \maF))$ and} we recover in this case the index construction given for more general double foliations  in \cite{hilsum1987morphismes}.  
\end{itemize}}
\end{remarque}

We may state the similar proposition when an extra compact  group acts on the whole data. More precisely, we have:

\begin{prop}\label{HequivariantIndexClass}
Assume that the compact group $G_1$ acts as before by holonomy diffeomorphisms, and that an extra compact group $G_2$ acts also on $M$ by $F$-preserving isometries  (not necessarily preserving the leaves) such that this $G_2$-action commutes with the action of $G_1$, then the previous construction yields, for $G_1\times G_2$-invariant $G_1$-transversally elliptic operators along the leaves of $(M, \maF)$, to a well defined  $G_2$-equivariant index map
{$$
\mathrm{Ind}^{\mathcal{F},G_2} : \k^i_{\mathrm{G_1\times G_2}} (F_{G_1}) \longrightarrow \k\k^i_{\mathrm{ G_2}}(C^*G_1,C^*(M,\mathcal{F})), \quad i\in \Z_2.
$$}
\end{prop}

The $G_2$-equivariant index class of the $G_1\times G_2$-invariant leafwise pseudodifferential operator $P_0$ on $(M, \maF)$ which is $G_1$-transversally elliptic is represented again by the cycle $(\maE, \pi, P)$ which is now in addition $G_2$-equivariant. Indeed, the  Hilbert module $\maE$ is automatically endowed with the extra $G_2$-action so that $\maE$ is a $G_2$-equivariant Hilbert module over the $G_2$-algebra $C^*(M, \maF)$. Here of course the actions are the usual ones induced from the action on the holonomy groupoid $\maG$ and on the bundle and no need to assume that the action of $G_2$ preserves the leaves. When the group $G_1$ is for instance the trivial group, then we recover the $G_2$-equivariant index class for $G_2$-invariant leafwise elliptic operators as considered for instance in \cite{Benameur:LongLefschetzKtheorie}.

%

{\begin{remarque}\label{Remark:KK_G:equiv}
If we assume in the previous proposition that the extra group $G_2$ also acts by holonomy diffeomorphisms on $(M, \maF)$, then exactly as for the $G_1$-action, we can arrange the $G_2$-action on $\maE$ so that it becomes a $G_2$-equivariant Hilbert module over the {$G_2$-trivial $C^*$-algebra $C^*(M, \maF)$}. Hence in this case,  there are again two ways to define the equivariant index class {but they are isomorphic.}

\end{remarque}}

\begin{remarque}
{We shall see in Subsection \ref{Section:Excision} that the index morphism is also well defined when $M$ is not necessarily compact as a morphism on the (compactly supported) equivariant $K$-theory of the space $F_G$.} 
\end{remarque}

\subsection{The $\k$-theory multiplicity of a representation}
\noindent
For any irreducible unitary representation of $G$, we now proceed to define a class  in $\k_i (C^{*}(M,\mathcal{F})) $ playing the role of its  multiplicity   in the index class $\ind(P)$, and which coincides with  the usual multiplicity as defined by Atiyah in \cite{atiyah1974elliptic} in the case of a single operator, corresponding for us to the maximal  foliation with a single leaf.

So let  $\rho: G\to U(X)$ be an  (irreductible) unitary representation  of $G$ in the finite dimensional space $X$. For simplicity, we shall refer to such representation by $X$ when no confusion can occur. Recall that the space of isomorphism classes of irreducible unitary representations of $G$ is the discrete dual $\widehat{G}$ of $G$, hence we have fixed $X\in \widehat{G}$. 

{
Recall that the representation $X$  corresponds to a projection   which represents   a  $\k$-theory class in $K_0(C^*G)$ (see \cite{julg1981produit:croise}),  equivalently $X$ corresponds to a class   $[X]\in \k\k(\mathbb{C}, C^*G)$ given by $[(X, 0)]$.  More precisely, denote  as usual  by $C_{v, w}$  the coefficient of the representation $X$ corresponding to the vectors $v, w\in X$, i.e. the map $g\mapsto <v, \rho (g) w>$. Then the Schur Lemma gives  the relations 
$$
C_{v_1,w_1} \ast C_{v_2,w_2}=\frac{1}{\dim X} \langle w_1,v_2 \rangle C_{v_1,w_2}\quad \mbox{and}\quad C_{v_1,v_2}^*=C_{v_2,v_1}, \qquad v_1,v_2,w_1,w_2 \in X,
$$ 
which in turn imply that for any $v\in X$ such that $\|v\|=1$, the element $p_v=(\dim X) C_{v,v}$  is a minimal projection in $C^*G$, and the map $\phi_v : X \to p_vC^*G$ given by $\phi_v(w)= \sqrt{\dim X} C_{v,w}$ is an isomorphism of Hilbert $C^*G$-modules.  
Therefore, $[X]=[p_vC^*G, 0]$ independently of the unital vector $v$.

Using our representation $\pi$, we thus have a well defined Kasparov cycle 
$(p_v \mathcal{E}, P_{p_v})$ with $P_{p_v}$ being the restriction of $P$ to $p_v \maE=\pi (p_v)\maE$,  which represents by definition the class of the Kasparov product $[X]\underset{C^*G}{\otimes} \ind(P_0)$. 
} 

{\begin{defi}\label{DefMultiplicite}
The $\k$-multiplicity $m_{P_0}(X)$ of the irreducible unitary representation $\rho: G\to U(X)$  in the index class  $\ind (P_0)$ is  the image of the above class 
$[p_v\mathcal{E}, P_{p_v})]\in \k\k^i(\mathbb{C},C^{*}(M,\mathcal{F}))$ under the  isomorphism $\k\k^i(\mathbb{C},C^{*}(M,\mathcal{F}))\stackrel{\cong} {\longrightarrow}\k_ i (C^{*}(M,\mathcal{F}))$.  Hence  we end up with the well defined $K$-multiplicity map:
$$
m_{P_0}: \widehat{G} \longrightarrow \k_i(C^*(M, \maF))\text{ which assigns to }X\text{ the multiplicity }m_{P_0}(X) \in \k_i(C^*(M, \maF)).
$$
\end{defi}
}

{Let us compare our definition of the $K$-multiplicity with Atiyah's definition. 
Recall the Hilbert module structure of $X\otimes \maE$ over $C^*(M, \maF)$:
$$
(v\otimes \xi) a := v\otimes \xi a \text{ and } \langle v\otimes \xi, v'\otimes \xi'\rangle := <v, v'> \langle \xi, \xi'\rangle, \quad \xi, \xi'\in \maE, v, v'\in X\text{ and } a\in C^*(M, \maF).
$$
Denote by  $\mathcal{E}^G_X$  the Hilbert $C^*(M, \maF)$-submodule of $X\otimes \mathcal{E}$ composed of the $G$-invariant elements for the action of $G$ given by $\rho\otimes U$ where $U$  has been introduced in Section \ref{C(B):module} using that the action is by holonomy diffeomorphims, i.e.  with the previous notations, 
$$
\mathcal{E}^G_X=(X\otimes \mathcal{E})^G := \{\xi \in X\otimes \maE \text{ such that } (\rho(g)\otimes U_g) \xi = \xi, \forall g\in G\}.
$$
}

{
\begin{lem}\label{TechnicIsom}
The following standard formula holds:
$$
\left\langle \sum v_i\otimes \eta_i, \sum w_j\otimes \xi_j \right\rangle = \sum \langle \eta_i,\pi(C_{v_i,w_j})\xi_j\rangle,\quad \text{for } \sum v_i\otimes\eta_i\in X\otimes \maE\text{ and } \sum w_j\otimes\xi_j\in \maE_X^{G}.
$$
\end{lem}} 

\begin{proof}\ 
{By linearity, we forget the sums. We} have for any $\gamma\in \maG$, $\eta,\xi \in C_c^\infty(\maG,r^*E)$ and since $\int_G \rho(g) (w)\otimes U_g(\xi) dg=  w\otimes \xi$: 
\begin{eqnarray*}
\left\langle v\otimes\eta, w\otimes\xi\right\rangle (\gamma) & = & \int_{\maG_{r(\gamma)}}  \left\langle  (v\otimes \eta)(\gamma_1),  (w\otimes \xi)(\gamma_1\gamma)\right\rangle d\nu_{r(\gamma)}(\gamma_1)\\
& = & \int_{\maG_{r(\gamma)}} \left\langle (v\otimes \eta)(\gamma_1), \int_G (\rho(g) (w)\otimes U_g\xi) (\gamma_1\gamma) dg\right\rangle d\nu_{r(\gamma)}(\gamma_1)\\
& = & \int_{\maG_{r(\gamma)}} \int_G   \langle v, \rho (g)w\rangle \left\langle  \eta(\gamma_1),  (U_g\xi)(\gamma_1\gamma) \right\rangle  dg d\nu_{r(\gamma)}(\gamma_1)\\
& = & \int_{\maG_{r(\gamma)}}    \left\langle  \eta(\gamma_1),   \int_G C_{v, w} (g) (U_g\xi)(\gamma_1\gamma) dg \right\rangle  d\nu_{r(\gamma)}(\gamma_1)\\
& = & \langle \eta,\pi(C_{v,w})\xi\rangle(\gamma).
\end{eqnarray*}
\end{proof}
}

%

By restricting the operator $\text{id}_X \otimes P \in \maL_{C^*(M, \maF)} (X\otimes \maE)$ to the $G$-invariant elements, we get the adjointable  operator $P^G_X \in \mathcal{L}_{C^*(M,\mathcal{F})}(\mathcal{E}_X^G)$, i.e. $
P^G_X:=(\text{id}_X \otimes P)\vert_{\mathcal{E}^G_X}.$

{
\begin{prop}
For any (irreducible)  unitary representation $\rho:G\to U(X)$, $\big( \mathcal{E}^G_X,P^G_X\big)$ is a Kasparov cycle whose class in $\k\k^i (\mathbb{C},C^*(M,\mathcal{F}))$ coincides with the class of the cycle $(p_v \mathcal{E}, P_{p_v})$ for any unital vector $v\in X$. In particular, the image of $\left[ \mathcal{E}^G_X,P^G_X\right]$ under the isomorphism $\k\k^i(\mathbb{C},C^{*}(M,\mathcal{F}))\stackrel{\cong} {\longrightarrow}\k_ i (C^{*}(M,\mathcal{F}))$ coincides with the $K$-multiplicity $m_{P_0} (X)$.
\end{prop}}

{\begin{remarque}
Our definition of the $K$-multiplicity is therefore the exact generalization of Atiyah's definition, and the multiplicity map $m_{P_0}$ defined in \ref{DefMultiplicite} is just the corresponding inteprepretation of the index class $\ind (P_0)$,  in the spirit of the series decomposition used for the distributional index in \cite{atiyah1974elliptic}.
\end{remarque}}

\begin{proof}\
{We have for any unital vector $v\in X$, a unitary equivalence between $(p_v C^*G \underset{C^*G}{\otimes} \maE, \Id \underset{C^*G}{\otimes} P)$ and $(X \underset{C^*G}{\otimes} \maE, \Id \underset{C^*G}{\otimes} P)$ given by $\phi_v \otimes \Id : X \underset{C^*G}{\otimes} \maE \rightarrow p_vC^*G \underset{C^*G}{\otimes} \maE $.
%
On the other hand, we have an isomorphism
$$
\Av \, : \,X\underset{C^*G}{\otimes}\mathcal{E} \longrightarrow \maE_X^G=(X\otimes \mathcal{E})^G
\text{ given by }  \int_{G} (\rho(g) \otimes U_g) (\bullet) \; dg.
$$
We compute for  $w\in X$ and $\eta\in \maE$:
\begin{multline*}
\int_G (\rho(g) \otimes U_g) (w\cdot \varphi \otimes \eta) dg = \int_G (\rho(g) \otimes U_g)(w\otimes \pi(\varphi)\eta) \; dg \\ \text{ and } \int_G (\rho(g) \otimes U_g) (w\otimes \eta)dg=\sum \limits_k e_k  \otimes \pi(C_{e_k,w})\eta
\end{multline*}
where $(e_k)_k$ is  a given orthonormal basis of $X$. The first relation shows that the map $\Av$ is well defined and we observe that the range of $\Av$ is exactly equal to $\maE_X^G$.
Using the second relation, Lemma \ref{TechnicIsom}, and the standard fact that for any $(w_1, w_2)\in X^2$ we have $
\sum_k C_{w_1, e_k} C_{e_k, w_2} = C_{w_1, w_2}$,  in $C^*G$,
we see that $\Av$ is indeed a unitary between the Hilbert $C^*(M,\mathcal{F})$-modules  $X\underset{C^*G}{\otimes}\mathcal{E}$ and $\maE_X^G$. That $\big( \mathcal{E}^G_X,P^G_X\big)$ is a Kasparov cycle is now a consequence of the above unitary equivalence.}
\end{proof}

{\begin{remarque}
If $(M, \maF)$ admits a holonomy invariant Borel transverse measure $\Lambda$, then applying the associated additive map $K_0(C^*(M, \maF))\to \R$, we get a well  defined $\Lambda$-multiplicity morphism, in the graded case, for the $G$-transversely elliptic operator $P_0$:
$$
m_{P_0}^\Lambda:  \widehat{G} \longrightarrow \R,
$$ 
in the spirit of the Murray-von Neumann dimension theory. 
\end{remarque}}

\section{The Atiyah axioms for our index morphism}\label{axioms}

\noindent As before, we denote by $F_G$ the closed subspace of $F$ defined by $F\cap T_GM$. 

\subsection{The index for free actions}\label{Section:FreeActions}
In this subsection, we let $G$ and $H$ be both compact {\underline{Lie}} groups. Let $M$ be a smooth compact manifold and let $\mathcal{F}$ be a given smooth foliation of $M$.  We suppose that the compact group $G\times H$ acts on $M$ by leaf-preserving diffeomorphisms that we may assume to be isometries of the ambiant manifold $M$, by averaging the metric. We further assume that  $H$ acts {\underline{freely}} on $M$ so that the projection $q: M \rightarrow M/H$ corresponds to a $G$-equivariant principal $H$-fibration which sends leaves to leaves. So, we insist that we assume here that $H$ {\underline{preserves the leaves upstairs}} and induces the corresponding leaves downstairs, this is automatic when $H$ is connected. Notice  that the leaf of $(M, \maF)$ through a given point $m\in M$ coincides here with the inverse image of the leaf through $q(m)$ in the quotient manifold $M/H$. The induced foliation downstairs in $M/H$ will be denoted $\maF/H$ in the sequel. We denote again by $\pi : F \rightarrow M$ the vector bundle projection and by $\overline{\pi} : F/H\rightarrow M/H$ the induced vector bundle projection downstairs. The foliations $(M, \maF)$ and $(M/H, \maF/H)$ then have the same codimension and under our assumptions do actually have the same space of leaves as we explain below. 
The action of $H$ on $F$ then preserves the subspace $F_G$ and  we have an isomorphism
$$
q^* : \k_{\mathrm{G}}^i((F/H)_G) \rightarrow \k_{\mathrm{G\times H}}^i(F_{G\times H}).
$$ 
To be specific, this isomorphism identifies the classes of the $G$-invariant $G$-transversally elliptic  $\maF/H$-leafwise symbols over $M/H$, with those of the symbols of $G\times H$-invariant $G\times H$-transversally elliptic $\maF$-leafwise symbols over $M$. At the level of cycles, $q^*$ associates with  $(E, a)$ the cycle $(q^*E, q^*a)$ with  $q^*a (m, \xi) = a (q(m), q_*\xi)$, identifying again covectors with vectors.
{Said differently,  the space $F_{G\times H}$ is just the fibered product
$(F/H)_G \times_{ M/H} M$ and using the proper $G$-equivariant map
$\tilde{q} : F_{G\times H} \to (F/H)_G$, we can see that the map $q^*$ is the functoriality map $\tilde{q}^*$.
}

Let $\hat{H}$ be the set of isomorphism classes of irreducible unitary representations of the compact group $H$. We shall sometimes refer to an element $\alpha: H\to \End (W_\alpha)$ of $\hat{H}$ simply as $\alpha$, and the corresponding character on $H$ will be denoted $\chi_\alpha$. Associated with such representation we have the homogeneous bundle $\underline{W_\alpha}\to M/H$ associated with the principal $H$-bundle $q:M\to M/H$.  We thus have the classical  map:
$$
\begin{array}{clc}
R(H\times G) &\longrightarrow &\K(M/H)\\
V &\longmapsto &\underline{V}^* 
\end{array}
$$
where $V^*$ is the dual representation.

Using a distinguished open cover for the foliated manifold  $(M/H, \maF/H)$ which trivializes the principal fibration $q:M\to M/H$ as well, it is easy to see that the foliations $(M, \maF)$ and $(M/H, \maF/H)$ have Morita equivalent $C^*$-algebras. 
If we  denote  by $\maG (M/H, \maF/H)$ the holonomy groupoid of the foliation $(M/H, \maF/H)$, then this Morita equivalence is implemented by the Hilbert module associated with the graph space 
$$
G_q := \{ (m, \alpha)\in M\times \maG (M/H, \maF/H)\, \vert\, q(m) = r(\alpha)\}.
$$
This is the graph of the morphism of groupoids induced by the projection $q : M \rightarrow M/H$. The action of $\mathcal{G}$ on $G_q$ is given by $\gamma \cdot (m,\alpha)=(\gamma m, q(\gamma )\alpha )$ and we leave it as an exercise for the interested reader to show that we get in this way a principal $\maG$-bundle in the sense of \cite{renault2006groupoid} and that this bundle indeed embodies the Morita equivalence. 
As a consequence, we can define the imprimitivity Hilbert bimodule which realizes the Morita equivalence between the corresponding $C^*$-algebras as the completion of the pre-Hilbert module $C_c(G_q)$. 

There is a left prehilbert $C_c(\maG)$-module structure  on $C_c(G_q)$ given by 
$$
f\cdot \varphi (m, \alpha)=\int_{\maG^{m}} f(\gamma ) \varphi (s(\gamma), q(\gamma)^{-1}\alpha) \, d\nu^{m}(\gamma)
$$
and 
$$
_\mathcal{\maG}\langle \varphi,\psi \rangle (\gamma ) =\int_{\maG(M/H, \maF/H)^{qr(\gamma)}} \varphi (s(\gamma), q(\gamma^{-1}) \beta )\overline{\psi(r(\gamma),  \beta)}\, d\nu^{qr(\gamma)}(\beta).
$$
There is similarly a right prehilbert $C_c(\maG(M/H, \maF/H))$-module structure on  $C_c(G_q)$ given by
$$
\varphi\cdot \xi (m, \alpha)=\int_{\beta\in \maG(M/H, \maF/H)^{s(\alpha)}}  \varphi (m, \alpha \beta) \xi(\beta^{-1}) \, d\nu^{s(\alpha)}(\beta)
$$
and
$$
\langle \varphi,\psi \rangle_{\maG(M/H, \maF/H)} (\beta ) = \int_{\gamma\in \maG^{m_0}} \overline{\varphi (s(\gamma), q(\gamma)^{-1})}\; \psi(s(\gamma), q(\gamma)^{-1}  \beta)\, d\lambda^{m_0}(\gamma),
$$
 for a chosen $m_0\in q^{-1}\{r(\beta)\}$. Notice that the last integral does not depend on the choice of $m_0$ due to the $H$-invariance of our Haar system. We can now state our theorem.

\begin{thm}\label{thm:action:libre}
Denote by $\chi_1$ the class of the trivial representation in $\k\k(\C, C^*H)$. Then for $i\in \Z_2$, the following diagram commutes:
 $$\xymatrix{\k_\mathrm{G}^i((F/H)_G)\ar[r]^{q^*} \ar[d]_{\mathrm{Ind}^{\mathcal{F}/H}}~&~\k_{\mathrm{G\times H}}^i(F_{G\times H})\ar[d]^{\chi_1 \underset{C^*H}{\otimes}\mathrm{Ind}^\maF (\bullet)}\\
\k\k^i(C^*G,C^*(M/H,\mathcal{F}/H))\;\;&~\;\;\k\k^i(C^*G,C^*(M,\mathcal{F})).\ar[l]^{\;\;\;\underset{C^*(M,\mathcal{F})}{\otimes}\mathcal{E}_q}
 }$$
So if $a\in \k_\mathrm{G}^i((F/H)_G)$ then ignoring  the quasi-trivial Morita isomorphism, we may write:
$$
\mathrm{Ind}^{\mathcal{F}/H}(a)\cong \chi_{1} \underset{C^*H}{\otimes} \ind (q^*a).
$$ 
\end{thm}


\begin{proof}
Recall that $H$ acts freely on $M$ and preserves the leaves of $\maF$.  The holonomy groupoid upstairs is a principal  $H$-fibration over $G_q$, this latter is an $H$-fibration over the holonomy groupoid downstairs. More precisely, $\maG$ can be  identified with the smooth pull-back groupoid $\hat{\maG}$ given by
$$
\hat{\maG}:= \{(m, \alpha, m')\in M\times \maG(M/H, \maF/H)\times M\,\vert \, q(m)=r(\alpha) \text{ and } q(m')= s(\alpha)\}.
$$
The Haar system on $\maG$ is supposed to be   $H$-invariant and normalized. 
More precisely, we assume that the Haar system $(\nu^m)_{m\in M}$ on $\hat{\maG}$ combines the normalized Haar measure on $H$  with a Haar system downstairs $(\bar\nu^{\overline{m}})_{\overline{m}\in M/H}$ on $\mathcal{G} (M/H, \maF/H)$. 
Let $A$ be a $G$-invariant leafwise $G$-transversally elliptic pseudodifferential operator representing $a$, of order $1$ and supported as close as we please to the units $M/H$.  So the operator $A$  can be seen as a $\maG(M/H, \maF/H)$-invariant operator along the leaves of the groupoid $\maG(M/H, \maF/H)$, i.e. $A=(A_x)_{x\in M/H}$  with  
$$
A_x: C_c^\infty (\maG(M/H, \maF/H)_x, r^*E) \longrightarrow C_c^\infty (\maG(M/H, \maF/H)_x, r^*E)\text{ and with the usual equivariance}. 
$$
Using a partition of unity argument, we may lift $A$ to a $G\times H$-invariant leafwise $G\times H$-transversally elliptic pseudodifferential operator $\hat{A}$ on  $(M, \maF)$, which represents $q^*a$. Roughly speaking, the operator $\hat{A}$ corresponds in the identification $\maG\simeq \hat{\maG}$ to tensoring locally by the identity on both sides and can be denoted abusively by $\id\otimes A\otimes \id$. The index class of $q^*a$ can then be represented in $\k\k^i (C^*(G\times H), C^*(M, \maF))$ by the unbounded Kasparov cycle with the closure of $\hat{A}$ as an operator acting on the Hilbert module $\maE$ corresponding to the pull-back bundle $q^*E$ over $M$, and with the usual representation $\pi_{G\times H}$ of $C^*(G\times H)$.


Let us  first compute  the  Kasparov product $\ind(q^*a) \underset{C^*(M, \maF)}{\otimes} \mathcal{E}_q$. This is by definition the class of the triple 
$$ 
\left( \mathcal{E} \underset{C^{*}(M,\mathcal{F})}{ \otimes} C^*(G_q), \; \pi_{G\times H} \underset{C^*(M,\mathcal{F})}{\otimes} \id, \; \hat{A} \underset{C^*(M,\mathcal{F})}{\otimes} \id\right).
$$ 
If we denote as well by $r:G_q\to M/H$ the map $r(m, \alpha):= r(\alpha)=q(m)$ then it is easy to check that 
the Hilbert module $\mathcal{E} \underset{C^{*}(M,\mathcal{F})}{\otimes} C^*(G_q)$  is isomorphic to $C^*(G_q, r^*E)$, i.e. the completion of $C_c(G_q, r^*E)$ with respect to the prehilbertian structure given for $e_1, e_2\in C_c(G_q, r^*E)$ by 
\begin{equation}\label{equ:ps:C*MxG}
\langle e_1,e_2 \rangle (\beta) =\int_{\mathcal{G}^{m}}\langle e_1( s(\gamma), q(\gamma)^{-1}), e_2(s(\gamma), q(\gamma)^{-1}  \beta) \rangle d\nu^{m}(\gamma ).
\end{equation}
To be specific, this identification  can be described by a unitary $\maV$ which is given for $u\in C_c(\maG, r^*E)$ and $f\in C_c (G_q)$ by the formula
$$
\maV(u \otimes f )(m,\beta)=\int_{\mathcal{G}^m} u(\gamma ) f( s(\gamma), q(\gamma)^{-1} \beta ) d\nu^m(\gamma ).
$$
One then checks immediately that for any $\varphi\in C(G\times H)$, we have $\maV\circ (\pi_{G\times H}(\varphi) \otimes 1)=\tilde{\pi}_{G\times H}(\varphi) \circ \maV$, where for any $u \in C_c(G_q, r^*E)$, 
$$
\left[\tilde{\pi}_{G\times H}(\varphi ) (u) \right] (m,\eta)=\int_{H\times G} \varphi(g,h)~(g,h)\Big(u\big((g,h)^{-1} (m,\eta)\theta_{G\times H}^{(g,h)^{-1}}(s(\eta))\big)\Big) dg\, dh. 
$$

%
%
%

Similarly, we have that $\maV\circ (\hat{A}\otimes_{C^*(M, \maF)} \id)=(\id \otimes A) \circ \maV$, where $\id \otimes A$ stands for a first order  lift of the operator $A$ to $G_q$ using the $H$-fibration $G_q\to \maG(M/H, \maF/H)$. To sum up, we see  that
$$
\ind(q^*a) \underset{C^*(M, \maF)}{\otimes} \mathcal{E}_q =\big[ C^*(G_q ,r^*E), \tilde{\pi}_{G\times H}, \id\otimes A \big]. 
$$
It now remains to compute the Kasparov product of this latter class with the trivial representation of $H$.   
We shall use the  identification 
$$
C^*(G_q,r^*E)^{\theta_H}\simeq C^*(\mathcal{G}(M/H, \maF/H) , r^*E).
$$ 
Notice indeed  that for  $\theta_H$-invariant sections $e_1$ and $e_2$, we have:
$$
\begin{array}{lll}
\langle e_1,e_2 \rangle (\beta) &=  \int_{\mathcal{G}^{m_0}}\langle e_1(s(\gamma ),q(\gamma^{-1}) ), e_2(s(\gamma), q(\gamma^{-1}) \beta) \rangle \; d\nu^{m_0}(\gamma )\\ 
&=\int_{\mathcal{G}(M/H, \maF/H)^{r(\beta)}}\;  \langle \overline{e}_1( \beta^{-1}_1), \overline{e}_2 (\beta_1^{-1} \beta) \rangle \; d\overline{\nu}^{r(\beta)}(\beta_1 ).
\end{array}
$$ 
In the first expression $m_0\in M$ is any chosen element of the fiber over $r(\beta)$, and the last equality is a consequence of the $\theta_H$-invariance together with our choice of Haar system upstairs which uses the {\underline{normalized}} Haar measure on $H$. Now $\id \underset{C^*H}{\otimes} \tilde{\pi}_{G\times H}$ and $\id \underset{C^*H}{\otimes} (\id \otimes A)$ both make sense and by using the previous isomorphism  we can see that the first coincides with the representation $\pi_G$ of $C^*G$ on $C^*(\mathcal{G}(M/H, \maF/H),r^*E)$ while the second is just the  operator $A$. The verification is an exact rephrasing of the same proof for a single operator and is therefore omitted here. 
%
%
Whence we eventually  get the allowed equality
$$
\mathrm{Ind}^{\mathcal{F}/H}(a)=\chi_{1} \underset{C^*(H)}{\otimes} \, \left[ \ind (q^*a)\, \underset{C^*(M,\mathcal{F})}{\otimes}\,  \mathcal{E}_q\right].
$$ 
Associativity of the Kasprov product allows to conclude.
\end{proof}

If we replace $a$ by the symbol corresponding to the twist of $a$ by a given unitary representation $(\alpha, W_\alpha)$, then the same proof yields to the following result:

\begin{thm}$\cite{atiyah1974elliptic}$\label{thm:action:libre-twistee}
Let $(W_\alpha, \alpha)$ be a given finite dimensional unitary  representation of $H$ and denote as before by $\chi_\alpha$ the corresponding class in $\k\k(\C, C^*H)$. Then the following diagram commutes:
 $$\xymatrix{\k_\mathrm{G}^i((F/H)_G)\ar[r]^{q^*} \ar[d]_{\underline{W}_{\alpha}^*\otimes \; \bullet}~&~\k_{\mathrm{G\times H}}^i(F_{G\times H})\ar[r]^{\hspace*{-0.5cm}\mathrm{Ind}_{}^{\mathcal{F}}}~&~\k\k^i(C^*(G\times H),C^{*}(M,\mathcal{F}))\ar[d]^{\chi_\alpha \underset{C^*H}{\otimes} \bullet}\\
\k_\mathrm{G}^i((F/H)_G)\ar[r]_{\hspace*{-0.9cm}\mathrm{Ind}^{\mathcal{F}/H}} ~&\k\k^i(C^*G,C^*(M/H,\mathcal{F}/H))&~\k\k^i (C^*G,C^*(M,\mathcal{F}))\ar[l]^{\underset{C^*(M,\mathcal{F})}{\otimes}\mathcal{E}_q}. 
 }$$
In other words, if $a\in \k_G^i((F/H)_G)$ and ignoring the Morita isomorphism $\underset{C^*(M,\mathcal{F})}{\otimes}\mathcal{E}_q$, we have 
$$
\mathrm{Ind}^{\mathcal{F}/H}( \underline{W}_{\alpha}^*\otimes a)\cong \chi_{\alpha } \underset{C^*H}{\otimes} \ind (q^*a).
$$ 
\end{thm}

\begin{remarque}
In particular, as an element of $\mathrm{Hom}\big(R(H),\k\k^i(C^*G,C^{*}(M,\mathcal{F}))\big)$ we have
 $$\mathrm{Ind}^{{\mathcal{F}}} (q^*a)=\sum \limits_{\alpha \in \hat{H}}\hat{\chi}_\alpha \otimes \mathrm{Ind}^{{\mathcal{F}/H}}( \underline{W}_{\alpha}^*\otimes a), 
 $$
where  $\hat{\chi}_{\alpha }$ is the element of $\Hom (R(H), \Z)$ given by the usual multiplicity.

\end{remarque}

When the group $G$ is the trivial group, we obtain the following expected relation between the Connes-Skandalis index of leafwise elliptic operators downstairs and the index of leafwise $H$-transversally elliptic operators upstairs.

\begin{cor}
Let $q:(M,\mathcal{F})\to (M/H, \maF/H)$ be as above a principal $H$-fibration of smooth foliations, recall that $H$ preserves the leaves of $(M, \maF)$. Then for any leafwise elliptic pseudodifferential symbol $\sigma$ on $(M/H, \maF/H)$ so that  $a=[\sigma]\in \k^i(F/H)$, we have the following equality:
$$
\chi_\alpha \underset{C^*H}{\otimes}\mathrm{Ind}^{\mathcal{F}/H} (q^*a) \underset{C^*(M,\mathcal{F})}{\otimes} \mathcal{E}_q=\ind(\underline{W}_{\alpha}^*\otimes a ),
$$
where $ \ind(\underline{W}_{\alpha}^*\otimes a )\in \k\k^i(\mathbb{C}, C^*(M,\mathcal{F}))\simeq \k_i (C^*(M,\mathcal{F}))$ is the Connes-Skandalis index, as defined in \cite{ConnesSkandalis}, for the leafwise elliptic symbol $\underline{W}_{\alpha}^*\otimes \sigma$ on the compact foliated manifold $(M, \maF)$.
\end{cor}

\noindent
\begin{remarque} The previous corollary can as well be stated with the extra action of the compact Lie group $G$ and gives a relation between the corresponding $G$-indices \cite{benameur1993theoreme}.
\end{remarque}

\subsection{The excision property}\label{Section:Excision}

In this subsection, we show an excision property for the index class of $G$-invariant leafwise $G$-transversally elliptic operators. More precisely, we shall first extend our definition of the index  morphism to the case of any smooth foliated (open) manifold $(U, \maF^U)$  which is again endowed with a leaf-preserving  action of  $G$ by holonomy diffeomorphisms, and obtain an index  morphism
$$
\Ind^{\maF^U}\, : \, \k_\mathrm{G}^i (F^U_G) \longrightarrow \k\k^i  (C^*G, C^*(U, \maF^U)).
$$
Then we shall show the compatibility of this morphism with foliated open embeddings, in particular in closed foliated manifolds, this is the expected excision result. Again $C^*(U, \maF^U)$ is the Connes algebra of the foliation $(U, \maF^U)$, i.e. the $C^*$-completion of the convolution algebra of compactly supported continuous functions on the holonomy groupid $\maG (U, \maF^U)$ of $(U, \maF^U)$.  As usual, we have fixed a $G$-invariant metric on $U$ and used it to identify for instance the colongitudinal bundle $(F^U)^*$ with the longitudinal bundle $F^U$.
We shall use the following classical lemma which is shown for instance in \cite[lemma 3.6]{atiyah1974elliptic} in the non-foliated case, see also \cite{Atiyah-Singer:I} for the original proof in the elliptic case and  \cite{ConnesSkandalis} for the leafwise elliptic case. The proof for the foliated $G$-transversely elliptic case is similar with the same standard techniques and hence is omitted. Let  $\pi_U : F^{U} \rightarrow U$  be the projection of the tangent space to the foliation $\maF^U$. We denote as before by $F^U_G = F^U\cap T_GU$.

\begin{lem}\label{lem:symbole:T^VU}\
Each element $a\in \K(F^U_G)$ can be represented by a $G$-equivariant  morphism $ \pi_U^*E^+ \stackrel{\sigma}{\longrightarrow} \pi_U^*E^-$ {over the whole of $F^U$, which is zero-homogeneous for large $\xi$}
 with $E^\pm$ being $G$-equivariant vector bundles over $U$, and such that:
\begin{itemize}
\item Outside some compact $G$-subspace $L$ in $U$, the bundles $E^\pm$  are trivialized and the restriction of $\sigma$ to $\pi_U^{-1} (U\smallsetminus L)$ is the identity morphism modulo the trivializations of $E^\pm$.
\item The morphism $\sigma(x, \xi): E^+_x\to E^-_x$ is an isomorphism for $(x, \xi)\in F^U_G\smallsetminus U$.
\end{itemize}
\end{lem}
\noindent

So, the first item means that there exist bundle $G$-equivariant isomorphisms over $U\smallsetminus L$ (or rather over each of its connected components)
$$
\psi^{\pm} \,:\,  E^\pm\vert_{U\smallsetminus L} \longrightarrow (U\smallsetminus L)\times \C^{\dim (E^\pm)} \text{ such that } \forall (x, \xi)\in \pi_U^{-1} (U\smallsetminus L): \sigma (x, \xi) = (\psi_x^-)^{-1} \circ \psi^+_x: E^+_x\rightarrow E^-_x.
$$
We endow the vector bundles $E^\pm$ with $G$-invariant hermitian structures and consider the Hilbert modules $\maE^\pm$ over $C^*(U, \maF^U)$ which, as in the previous sections, are the completions of the prehilbertian $C_c (\maG (U, \maF^U))$-modules $C_c (\maG (U, \maF^U), r^* E^\pm)$. Moreover, using the equivalence relation of stable homotopies with compact support as in \cite{Atiyah-Singer:I},  the bundle trivialization $\psi^\pm$ can be assumed to be {\underline{bounded}} and in fact even fiberwise unitaries for the hermitian structures.  We thus assume  as well that $\sigma^*\sigma$ and $\sigma\sigma^*$ are the identity bundle isomorphisms of $E^+$ and $E^-$ respectively, over $U\smallsetminus L$. By using the holonomy action as in Section \ref{Section.Index}, we can endow the Hilbert modules  $\maE^\pm$ with the structure of Hilbert $G$-modules when $C^*(U, \maF^U)$ is trivially acted on by the compact Lie group $G$. We can now quantize any such zero-degree homogeneous symbol $\sigma$ and choose a uniformly supported zero-th order $G$-invariant pseudodifferential $\maG (U, \maF^U)$-operator $P_0: C^\infty_c(\maG (U, \maF^U), r^*E^+) \rightarrow C^\infty_c(\maG (U, \maF^U), r^*E^-)$ with the principal symbol equal to $\sigma$. More precisely, uniform support is taken in the sense of \cite{Shubin92}, see also \cite[Proposition 4.6]{ConnesSkandalis}. Here, we can in fact ensure that   $P_0$ is the identity operator outside some compact $G$-subspace  $L'$ {whose interior contains $L$}, i.e. that we have
$$
P_0(\eta)(\gamma)=(\psi^-_{x})^{-1}\circ (\psi^+_{x})(\eta(\gamma)), \text{for any }\eta\in C_c^\infty (\maG(U, \maF^U)_x^{U\setminus L'}, r^* E^+).
$$
Hence $P_0^*P_0$ and $P_0P_0^*$ reduce to the identity operators on the sections which are supported above $U\smallsetminus L'$, say in $r^{-1} (U\smallsetminus L')$.


The operator  $P_0$ hence reduces to multiplication by the unitary bundle morphism $(\psi^-)^{-1}\circ (\psi^+)$ over $U\smallsetminus L'$, and it is easy to deduce that it extends to an adjointable $G$-equivariant operator from $\maE^+$ to $\maE^-$ \cite{NWX} that we still denote by $P_0$, see also \cite{vassout2006unbounded, BenameurHeitschTransverseNCG}. We  denote  as usual by $P$ the self-adjoint $G$-invariant operator 
\begin{equation}\label{equ:P:non:compact}
P:=\begin{pmatrix}
0&P_0^*\\
P_0& 0
\end{pmatrix}
\end{equation} acting on the Hilbert $G$-module $\maE=\maE^+\oplus \maE^-$. Recall that we are considering the trivial $G$-action on $C^*(U, \maF^U)$ and that  $\maE$ is endowed with the structure of a $\Z_2$-graded Hilbert $G$-module.

\begin{prop}\label{Index-intrinseque}\
The triple $(\maE, \pi, P)$ is a $\Z_2$-graded Kasparov cycle over the pair of $C^*$-algebras \linebreak $(C^*G, C^*(U, \maF^U))$ and defines an index class in $\k\k (C^*G, C^*(U, \maF^U))$ which only depends on the \linebreak $G$-equivariant stable homotopy class $a=[\sigma]\in \K(F^U_G)$ and is denoted $\Ind^{\maF^U} (a)$. Moreover, we get in this way a well defined index morphism for the open foliation $(U, \maF^U)$:
$$
\Ind^{\maF^U}\, : \, \K(F^U_G) \longrightarrow \k\k (C^*G, C^*(U, \maF^U)).
$$
{We have similarly a well defined (ungraded) index map
$$
\Ind^{\maF^U}\, : \, \k^1_\mathrm{G}(F^U_G) \longrightarrow \k\k^1 (C^*G, C^*(U, \maF^U)).
$$}
\end{prop}

\begin{proof}
{We freely use notations from Section \ref{Section.Index} and only treat the even case. Since $\sigma$ is bounded on $L$ and assumed to be unitary outside $L$, we get that $\sigma$ is bounded. 
Now notice that $\|\sigma(P)(x,\xi)^2-\id\|\neq 0$ only for $x\in  L$ which implies, using Lemma \ref{lem:Kasparov:eq:def}, that 
\begin{equation}\label{equ:ineg:non:compact:symbol}
 \forall \varepsilon>0, \exists c>0\ \mathrm{such\ that}\  \|\sigma(P)(x,\xi)^2-\id\|\leq c\sigma(Q)(x,\xi)+ \varepsilon,
\end{equation} 
where $\sigma(Q) (x,\xi)= q_x(\xi)$ on $S^*\maF^U$. 
Using Proposition \ref{prop:ineq:non:compact}, we get
\begin{equation}\label{equ:ineg:non:compact:op}
-(c Q+\varepsilon +R_1)\leq P^2-\id\leq c Q+\varepsilon +R_2 \text{  as self-adjoint operators on }\maE
\end{equation}
where $R_i \in \maK_U(\maE)$. Denote  by $\theta\in C^\infty_c(U,[0,1])$ a bump function which equals $1$ on $L{'}$. If $\zeta \in C_c(U,[0,1])$ is any function which equals  $1$ on $L$ and $0$ outside $L^{'}$, we have $\sigma(P)^2(r(\gamma),\xi)-\id = \zeta(r(\gamma))(\sigma(P)^2(r(\gamma),\xi)-\id)$ and $\zeta \theta =\zeta$. 
We use  as usual an oscillatory integral to define the quantization map, see for instance  \cite{ConnesSkandalis}.  More precisely,  the $\maG$-invariant operator $P^2-\id$ is given through its Schwartz kernel, a distribution $k_{P^2-\id}$ on $\maG$  given by an expression of the following type
$$
k_{P^2-\id}(\gamma) = \int_{F^{U^*}_{r(\gamma)}} \chi(\gamma^{-1}) e^{-i\langle \Phi(\gamma^{-1}) ,\xi \rangle} \big( \sigma(P)^{2}(r(\gamma),\xi )-\id \big) d\xi =k_{ (P^2-\id)r^*\zeta}(\gamma)=k_{r^*\zeta  (P^2-\id)}(\gamma), 
$$
where $\Phi$ is a diffeomorphism from a {uniform} neighbourhood $W$ of $U$ in $\maG(U,\maF^U)$ to a neighbourhood of the zero section in ${F^U}$ with $d\Phi=\id$ and $\chi$ is a cut off function with support inside $W$ {which is equal to $1$ in a smaller neighborhood of $U$ whose closure is contained in $W$, see  \cite{ConnesSkandalis} as well as \cite{NWX} or \cite{vassout2006unbounded}. With the appropriate choice of these small neighborhood of the unit manifold $U$ in $\maG(U,\maF^U)$ in coherence with $L'$, we obtain  that}
 $$
 r^*\theta (P^2-\id)=P^2-\id {=(P^2-\id) r^*\theta}.
 $$
Multiplying {on both sides}  the inequality \eqref{equ:ineg:non:compact:op}  by $r^*\theta$, we get
$$
-(c r^*\theta Qr^*\theta+\varepsilon r^*\theta^2 +r^*\theta R_1r^*\theta)\leq P^2-\id\leq c r^*\theta Qr^*\theta +\varepsilon r^*\theta^2 + r^*\theta R_2r^*\theta.
$$ 
Furthermore, modulo  $\maK_{U}(\maE)$, $Q$ can be represented by $d_G^*(1+\Delta)^{-1}d_G$. 
Notice that $
\sigma(d_G^*(1+\Delta)^{-1})(x,\xi)$ 
is bounded, therefore $d_G^*(1+\Delta)^{-1} \in \overline{\Psi^0(U,\maF^U,E)}_U$ and its zero-th order symbol vanishes, so that  $d_G^*(1+\Delta)^{-1} \in \maK_U(\maE)$. {In particular, the operator $r^*\theta d_G^*(1+\Delta)^{-1}$ is compact.} Now we can conclude exactly as in {the proof of} Theorem \ref{index-cycle}. Indeed, recall that we have
$$
d_G(r^*\theta \pi(\varphi)\eta)(\gamma)=\sum \frac{\partial \theta}{\partial V_k}(r(\gamma)) \pi(\varphi)\eta \otimes v_k + \theta (r(\gamma)) \pi(d\varphi) \eta(\gamma),
$$
{and since $d_G r^*\theta \pi(\varphi)$ is bounded  we deduce that $r^*\theta d_G^*(1+\Delta)^{-1}d_G r^*\theta \pi(\varphi)$ is compact.} Moreover, the operators $r^*\theta R_ir^*\theta$ are also compact since each $R_i \in \maK_U(\maE)$. {Now for any non-negative  $\psi=\varphi^* \varphi\in C^*G$,}
$$
-\pi(\varphi)^*(c r^*\theta Qr^*\theta+\varepsilon r^*\theta^2 +r^*\theta R_1 r^*\theta)\pi(\varphi)\leq \pi(\varphi)^*(P^2-\id) \pi(\varphi)\leq \pi(\varphi)^*(c r^*\theta Qr^*\theta+\varepsilon r^*\theta^2 + r^*\theta R_2r^*\theta)\pi(\varphi).
$$
{Since $P$ is $G$-invariant, this can be rewritten as}
$$
-\pi(\varphi)^*(c r^*\theta Qr^*\theta+\varepsilon r^*\theta^2 +r^*\theta R_1r^*\theta)\pi(\varphi)\leq (P^2-\id) \pi(\psi)\leq \pi(\varphi)^*(c \theta Qr^*\theta+\varepsilon r^*\theta^2 + r^*\theta R_2r^*\theta)\pi(\varphi).
$$
Therefore, projecting in the Calkin algebra and letting $\varepsilon \to 0$, we
{deduce that $(P^2-\id)\circ\pi(\psi)$ is compact. This implies, using the spectral theorem, that $(P^2-\id)\circ\pi(\varphi)$ is compact for any $\varphi\in C^*G$.}
}
\end{proof}

\medskip

Assume now that there exists an open foliated $G$-embedding $j: (U, \maF^U) \hookrightarrow (M, \maF)$ of smooth foliated manifolds. This means that $j$ is a $G$-equivariant embedding of $U$ as an open submanifold of the foliated $G$-manifold $M$ which transports the foliation $\maF^U$ into the restriction of the foliation $\maF$ to the open submanifold $j(U)$. We assume again that $G$ acts on all foliations by holonomy diffeomorphisms. We shall mainly be interested in the present paper in the case $M$ compact, but this is not needed so far. The embedding $j$ then induces an open embedding at the level of holonomy groupoids that we still denote by $j$ for simplicity, i.e. $
j\, : \, \maG (U, \maF^U) \hookrightarrow \maG (M, \maF)=\maG$. The $C^*$-algebra $C^* (U, \maF^U)$ is hence isomorphic to the $C^*$-algebra of the foliation of $j(U)$ induced by $\maF$, but this latter can be seen as a $C^*$-subalgebra of $C^*(M, \maF)$ in an obvious way. We therefore end-up with a well defined class $j_!\in \k\k (C^*(U, \maF^U), C^*(M, \maF))$. In the notations of \cite{ConnesSkandalis}, the map  $j: U\to M$ induces in particular a  submersion  from $U$ to $M/\maF$, and we therefore have a well defined  Morita extension Kasparov class which is exactly the class $j_!$, but the construction is simpler in our open embedding case.  Finally, the differential $dj:F^U\to F^M=F$ of the $G$-embedding $j$ restricts to an open $G$-map sending the space $F^U_G$ in the space $F^M_G= F_G$, therefore and by functoriality, we get an $R(G)$-morphism
$$
j_* \, : \, \k_\mathrm{G}^i(F^U_G) \longrightarrow \k_\mathrm{G}^i (F_G).
$$
We are now in position to state the main theorem of this subsection.

\medskip

{
\begin{thm}\label{thm:excision}
Under the above assumptions {and for any $i\in \Z_2$, the following diagram commutes: 
$$\xymatrix{\k_\mathrm{G}^i(F_G^U) \ar[r]^{j_*} \ar[d]^{\mathrm{Ind}^{\maF^U}} & \k_\mathrm{G}^i(F_G)\ar[d]^\ind \\
\k\k^i(C^*G,C^*(U,\maF^U)) \ar[r]_{j_!} &\k\k^i(C^*G,C^*(M,\maF)).
} $$}
\end{thm}}

\medskip

\begin{remarque}
If the action of $G$ on $(U, \maF^U)$ is not necessarily a holonomy action while it is a holonomy action on $(M, \maF)$, then the index morphism $\mathrm{Ind}^{\maF^U}$ is not well defined anymore, and one can use Theorem \ref{thm:excision} precisely to define it for any given such embedding, as a class in $\k\k^i(C^*G,C^*(M,\maF))$  with the usual compatibility with embeddings.
\end{remarque}

{
Recall first the notion of support of a Kasparov $(A, B)$-cycle $(\mathcal{E},\pi ,F)$, for given separable $C^*$-algebras $A$ and $B$, as introduced in \cite[Appendix A]{ConnesSkandalis}. This is the Hilbert submodule of $\maE$ which is generated by $\mathcal{K}_1\mathcal{E}$, with $\mathcal{K}_1$ the $C^{*}$-subalgebra of the $C^*$-algebra $\mathcal{K}(\mathcal{E})$ of compact operators on the Hilbert $B$-module $\maE$, which is generated by the operators $[\pi(a),F]$ , $\pi(a)(F^2-1)$ and $\pi(a)(F-F^*)$ and their multiples   by $A$, $F$ and $F^*$. Here $a$ runs over $A$ of course. Then obviously $\maE_1$ is a Hilbert $(A, B)$-bimodule and $F$ restricts automatically to $\maE_1$ to yield the operator $F_1$ so that $(\mathcal{E}_1, \pi, F_1)$  is again a Kasparov $(A, B)$-cycle. We quote the following interesting observation from \cite{ConnesSkandalis} which will be used in the sequel.
\begin{lem}\cite{ConnesSkandalis}\label{lem:support}
{Let $(\mathcal{E},\pi ,F)$ be a Kasparov $(A,B)$-cycle where $A$ and $B$ are given separable $C^*$-algebras. Let $(\mathcal{E}_1, \pi, F_1)$ be the  Kasparov $(A, B)$-cycle obtained by restricting to the support $\maE_1$. Then $(\mathcal{E}_1,\pi,F_1)$ defines the same $\k\k$-class, i.e. 
$$
\left[(\mathcal{E}_1,\pi,F_1)\right]  =  \left[(\mathcal{E},\pi ,F)\right]  \in \k\k (A, B).
$$}
\end{lem}
}

\medskip

\begin{proof}[Proof of Theorem \ref{thm:excision}]\
{{We concentrate again on the even case $i=0$.} Let $a \in \k_G (F_G^U)$ be fixed. We denote as before by $\pi_U : F^U \rightarrow U$ and by $\pi_M:F\to M$ the respective bundle projections. We start by representing $a$ by a symbol of order $0$ on $F^U$ according to Lemma \ref{lem:symbole:T^VU}:
$$
\pi_U^{*}E^+ \stackrel{\sigma}{\longrightarrow}  \pi_U^{*}E^-, 
$$ 
which is thus trivial outside a compact set $L$ of $U$. By using   the trivializations $\psi^\pm$, a standard argument allows to extend the hermitian bundles $E^\pm$ viewed over $j(U)$ to hermitian $G$-equivariant vector bundles $j_*E^\pm$ over $M$ with the obvious extension $j_*\sigma$ {so that $(j_*E^+, j_*E^-, j_*\sigma)$  represents the push-forward class $j_*a$}, see for instance \cite{Atiyah-Singer:I,atiyah1974elliptic}. The bundle trivilizations $\psi^\pm$ then give the extended bundle isomorphisms, still denoted $\psi^\pm$, over $M\smallsetminus j(L)$. Associated with the hermitian $G$-bundles $j_*E^\pm$, we then obtain  the corresponding Hilbert $G$-modules  over the $C^*$-algebra $C^*(M, \maF)$ that we denote by $j_*\mathcal{E}^\pm$. Recall that $j$ induces as well a $*$-homomorphism 
$$
j_*\,:\, C^*(U, \maF^U) \longrightarrow C^*(M, \maF),
$$
which allows to represent $C^*(U, \maF^U)$ as adjointable operators on $C^*(M, \maF)$ when this latter is viewed as a Hilbert module over itself. We can therefore consider the $\Z_2$-graded Hilbert $G$-module over $C^*(M, \maF)$, obtained by composition, and denoted as usual $\maE\underset{C^*(U, \maF^U)}{\otimes} C^*(M, \maF)$. This latter Hilbert $G$-module can be identified with a Hilbert $G$-submodule of $j_*\maE$, i.e. there is a Hilbert module {isometry}
$$
V\, : \, \maE\underset{C^*(U, \maF^U)}{\otimes} C^*(M, \maF)\longrightarrow j_*\maE,
$$
{which is given for {$\eta \in C_c(\maG(U,\maF^U), r^*E)$ and $f\in C_c(\maG)$ by $V(\eta \otimes f) = \tilde{\eta}\cdot f$, that is, the convolution of $\tilde{\eta}$, the extension by $0$ of $\eta$ outside $\maG(U,\maF^U)$, with $f$, i.e.
$$
V(\eta \otimes f) (\gamma) := \int_{\maG^{r(\gamma)}} \tilde\eta (\gamma_1) f(\gamma_1^{-1}\gamma) d\nu^{r(\gamma)} (\gamma_1).
$$}
We identify for simplicity $U$ with $j(U)$ for the rest of this proof.}
{{The Hilbert submodule 
$$
V\left(\maE\underset{C^*(U, \maF^U)}{\otimes} C^*(M, \maF)\right)
$$ 
can be identified with  the completion $C^*(\maG^{U},r^*E)$   of $C_c(\maG^{U} , r^*E)$ in $j_*\maE$, where $\maG^{U}$ is the space of elements of $\maG$ which end inside $U$. See \cite[Proposition 4.3]{ConnesSkandalis}. To finish the proof, we only need to compare the supports of the two Kasparov cycles, and to apply Lemma \ref{lem:support}.} 
}
}

{We  choose  a {uniformly} supported $G$-invariant leafwise pseudodifferential operator $P_0$ on $(U, \maF^U)$ with symbol $\sigma$ as in the above construction of the index class on $(U, \maF^U)$. So,  $P_0$ can be seen as a $\maG(U, \maF^U)$-operator in the sense of \cite{Connes:integration:non:commutative} that we denote again by 
$$
P_0\, : \, C^\infty_c (\maG (U, \maF^U), r^*E) \longrightarrow C^\infty_c(\maG(U, \maF^U),r^*E),
$$ 
which acts along the fibers of the groupoid and is an invariant family $(P_{0,x})_{x\in U}$. Here of course we assume, as we did in the construction of the index class,  that $P_0$ is the identity outside some compact subspace  $L'$ of $U$, modulo the trivializations $\psi^\pm$. For simplicity of notations, this operator is also the one over $j(U)$ with its foliation induced from $\maF$. In order to quantize the pushforward class $j_*a$, we can then consider  the uniformly supported $G$-invariant leafwise operator on $M$ defined as follows.}

{Let $\theta \in C_c^{\infty }(M, [0,1])$ be some $G$-invariant bump function which is equal to $1$ on $L'$, and whose support is a compact subspace of $j(U)$ outside of which the operator $P_0$ is trivial.  Denote by $\psi_r^\pm$ the isomorphisms $\psi^\pm$ viewed between the bundles $r^*E^\pm$ and which are only well defined over $r^{-1} (M\smallsetminus j(L))$.  Then $j_*P_0$ can be taken as the  $\maG$-operator on $(M, \maF)$ defined by
$$
j_*P_0\, := \, \tilde{P}_0 \; r^*\theta + (\psi^-_{r})^{-1}\circ \psi^+_{r} \; (1-r^*\theta).
$$ 
We use here the same cut-off function used to extend $\sigma$ to $F$. Hence $j_*P_0$ is obviously a zero-th order  leafwise $\maG$-operator which is $G$-invariant and has the principal symbol equal to {$j_*\sigma=\sigma \; \theta + (\psi^-)^{-1}\circ \psi^+ \; (1-\theta)$} and so represents $j_*a$. Recall that the index class $\Ind^{\maF^U} (a)$ is represented by the adjointable extension of the operator 
$P=\begin{pmatrix}
0&P_0^*\\
P_0& 0
\end{pmatrix}$
acting on the Hilbert module $\maE$, while the class $\Ind (j_*a)$ can obviously be represented by the adjointable extension of the operator 
$j_*P=\begin{pmatrix}
0&j_*P_0^*\\
j_*P_0& 0
\end{pmatrix}$
acting on the Hilbert module $j_*\maE$.} 

{{Notice  that $\mathrm{Ind}^{\maF^U}(a) \otimes j_! =\left[(\maE \underset{C^*(U,\maF^U)}{\otimes} C^*(M,\maF) , \pi \otimes 1 , P\otimes 1)\right]$ and using the  isometry $V$ defined above we deduce that the Kasparov cycle $(\maE \underset{C^*(U,\maF^U)}{\otimes} C^*(M,\maF) , \pi \otimes 1 , P\otimes 1)$  is unitarily equivalent to the cycle $[C^*(\maG^U ,r^*E) ,\pi , j_*P_{|_{C^*(\maG^U ,r^*E)}}]$. Indeed,  the representations of the $C^*$-algebra $C^*G$ are clearly compatible, and we have}}
$$
V(P\eta \otimes f) = \widetilde{P\eta}\cdot f= \tilde{P} \tilde{\eta} \otimes f =\tilde{P}(\tilde{\eta}\cdot f) = j_*P_{|_{C^*(\maG^U ,r^*E)}}V(\eta \otimes f),
$$
 with $\tilde{P}$ being as before the $\maG$-operator obtain from $P$ by extending trivially  its distributional kernel.
{To complete the proof, thanks to the {Connes-Skandalis Lemma \ref{lem:support}}, we only need to show that the supports of $\ind(j_*a)$ and $[C^*(\maG^U ,r^*E) ,\pi , j_*P_{|_{C^*(\maG^U ,r^*E)}}]$ are the same. But using the cut off function $\theta$ which is supported in $U$, we can write 
$$
((j_*P)^2-\id)r^*\theta = (j_*P)^2-\id
$$ 
and the same equality is true when replacing $j_*P$ by $j_*P_{|_{C^*(\maG^U ,r^*E)}}$ and $\theta $ by $\theta_{|_U}$. Therefore the supports do coincide as allowed. 
}

\end{proof}

\subsection{Multiplicativity of the index morphism}\label{Multiplicative}

Recall that $G$ is a compact Lie group. Let $M$ and $M'$ be two smooth closed manifolds endowed with smooth foliations that we denote respectively by $\maF$ and $\maF'$. We assume that $G$ acts by holonomy diffeomorphisms on $(M, \maF)$ and on $(M', \maF')$. We assume in addition that another compact Lie group $H$ acts on the first manifold $M$ {also by holonomy diffeomorphisms}, and that the actions of $G$ and $H$ commute. So said differently,  the compact Lie group $G\times H$ acts by holonomy diffeomorphisms which are isometries (for the ambiant manifold metric)  on $(M, \maF)$ and $(M', \maF')$ and we assume that the action of $H$ on the second manifold $M'$ is trivial. Recall that in this situation the compact Lie groups $G$ and $H$ act by inner automorphisms on the Connes' $C^*$-algebras of the foliations $(M, \maF)$ and $(M', \maF')$. We thus get for instance the following $C^*$-algebra isomorphism which will be used later on (see Corollary \ref{cor:G:C(M,F):interne})
$$
\Psi : C^*(M,\mathcal{F})\rtimes G \longrightarrow C^*(M,\mathcal{F}) \otimes C^*G,
$$
and which is induced by the map $C(G, C^*(M, \maF))\to C(G, C^*(M, \maF))$ given for $f\in C(G, C^*(M, \maF))$ and  $g\in G$ by $\Psi (f) (g) := f(g)U_g$.
This isomorphism allows indeed to replace, the crossed product $C^*$-algebra $C^*(M,\mathcal{F})\rtimes G$ by the tensor product $C^*$-algebra $C^*(M,\mathcal{F}) \otimes C^*G$. We denote by $[\Psi]\in \k\k(C^*(M,\mathcal{F})\rtimes G),C^*(M,\mathcal{F})\otimes C^*G)$ the induced $\k\k$-equivalence. {This element is {just the Kasparov descent of the $KK_G$-equivalence between $C^*(M,\maF)$ endowed with the inner action and $C^*(M,\maF)$ endowed with the trivial action.}}


Recall  the Kasparov descent map \cite{Kasparov1988} for given $G$-$C^*$-algebras $A$ and $B$. Let $\mathcal{E}$ be a Hilbert $G$-module on $B$. Define a right prehilbertian $C(G, B)$-module structure on the space $C(G, \maE)$ of continuous $\maE$-valued functions on $G$, by setting
$$
e\cdot d(s)=\int_G e(t)\, t d(t^{-1}s)\, dt \quad \mathrm{~and~} \quad \langle e_1,e_2\rangle  (s)=\int_G t^{-1}(\langle e_1(t),e_2(ts)\rangle  _{\mathcal{E}})dt,
$$
for  $e$, $e_1$, $e_2 \in C(G,\mathcal{E})$ and $d\in C(G,B)$.  Then the completion of $C(G,\mathcal{E})$ with respect to this Hilbert structure defines by classical arguments, a $B\rtimes G$-Hilbert module that we denote by $\maE\rtimes G$. If $\pi : A \rightarrow \mathcal{L}_B(\mathcal{E})$ is a $G$-equivariant $\ast$-morphism from $A$ to the $C^*$-algebra of adjointable operators on $\maE$ then the map 
$$
\pi\rtimes G  :A\rtimes G \rightarrow \mathcal{L}_{B\rtimes G}(\mathcal{E}\rtimes G)\text{ given by }(\pi \rtimes G(a)e)(t)=\int_G\pi(a(s)) s(e(s^{-1}t))ds
$$ 
is a $\ast$-morphism. Finally, if $T\in \mathcal{L}_B(\mathcal{E})$ then $T$ induces an operator $T\rtimes G\in \mathcal{L}_{B\rtimes G}(\mathcal{E}\rtimes G)$ defined by $(T\rtimes G)e(s):=T(e(s))$. It was then proved in  \cite{Kasparov1988} that if $(\mathcal{E},\pi ,T)$ is an $(A, B)$ Kasparov cycle, then  $(\mathcal{E}\rtimes G,\pi\rtimes G,T\rtimes G)$ is an $(A\rtimes G, B\rtimes G)$ Kasparov cycle and that this induces a well defined group morphism at the level of $\k\k$-theory. More precisely, \\

\begin{defi}\cite{Kasparov1988}\
For $i\in \Z_2$, the Kasparov descent map for the given $G$-algebras $A$ and $B$ is  the well defined induced map 
$$
j^G : \k\k_\mathrm{G}^i(A, B) \longrightarrow \k\k^i (A\rtimes G,B\rtimes G)\text{  is given by }[(\mathcal{E},\pi ,T)]\longmapsto [(\mathcal{E}\rtimes G,\pi \rtimes G,T\rtimes G)].
$$
\end{defi}

\medskip

Back to our foliations, recall from {Proposition \ref{HequivariantIndexClass}} the well defined $G$-equivariant index map for $G\times H$-invariant leafwise symbols on $(M, \maF)$ which are $H$-transversally elliptic along the leaves, i.e.
\begin{equation}\label{Gindex}\
\mathrm{Ind}^{\maF,G}: \mathrm{K^i_{G\times H}} (F_H) \longrightarrow \k\k_\mathrm{G}^i (C^*H, C^*(M, \maF)).
\end{equation}
If we compose this index map with the Kasparov descent map for the $G$-algebras $C^*H$ (with trivial $G$-action) and $C^*(M, \maF)$ {for the standard action induced from the $G$-action along the leaves, and further use} the isomorphism $\Psi$, then we end up with an index map
$$
\mathrm{\widehat{Ind}}^{\maF,G}: \mathrm{K^i_{G\times H}} (F_H) \longrightarrow  \k\k^i (C^*(H\times G), C^*(M, \maF)\otimes C^*G).
$$

\begin{remarque}
Since $G$ acts by holonomy diffeomorphisms here, we can recast the representative of the index class given by Equation \eqref{Gindex} so that it rather represents a $G$-equivariant class for the trivial $G$-action on $C^*(M, \maF)$. If we denote by $\k\k^i_{G^{\rm{trivial}}}(C^*H, C^*(M, \maF))$ the equivariant Kasparov group for the trivial $G$-action, then this yields an index morphism
\begin{equation}\label{TrivialGindex}
\mathrm{Ind}^{\maF,G^{\rm{trivial}}}: \mathrm{K^i_{G\times H}} (F_H) \longrightarrow \k\k^i_{\mathrm{G}^{\rm{trivial}}}(C^*H, C^*(M, \maF)).
\end{equation}
\end{remarque}

\begin{lem}\label{Compatibility}
Denote by $j^{G^{\rm{trivial}}}:\k\k^i_{G^{\rm{trivial}}}(C^*H, C^*(M, \maF)) \to \k\k^i (C^*(H\times G), C^*(M, \maF)\otimes C^*G)$ the Kasparov descent morphism for the trivial $G$-action, then the following relation holds:
$$
\mathrm{\widehat{Ind}}^{\maF,G} \, = \, j^{G^{\rm{trivial}}}\circ \mathrm{Ind}^{\maF,G^{\rm{trivial}}}.
$$
\end{lem}

{
\begin{proof} We denote by $C^*(M,\maF)_t$ the $C^*$-algebra of the foliation $(M,\maF)$ endowed with the trivial action. {Denote by  $[\psi]\in \k\K(C^*(M,\maF),C^*(M,\maF)_t)$ the $\k\K$-equivalence class defined using that the action on $C^*(M, \maF)$ is inner.
We then have} $$  \mathrm{Ind}^{\maF,G} \underset{C^*(M,\maF)}{\otimes} [\psi] =\mathrm{Ind}^{\maF,G^{\rm{trivial}}}.$$
Since 
$$
j^G( [\psi] )=[\Psi]\text{ and }j^G(\mathrm{Ind}^{\maF,G} \underset{C^*(M,\maF)}{\otimes} [\psi] )=j^G(\mathrm{Ind}^{\maF,G}) \underset{C^*(M,\maF)\rtimes G}{\otimes} j^G( [\psi] ),$$ 
we get the result by  applying the Kasparov descent on both side.
\end{proof}

}

Using the action of $G$ on the second foliation $(M', \maF')$ we also have the index map for leafwise $G$-transversally elliptic symbols
$$
\mathrm{Ind}^{\maF'}: \mathrm{K_{G}} (F'_G) \longrightarrow \k\k (C^*G, C^*(M', \maF')).
$$
A classical construction then allows to build up from a $G\times H$-invariant leafwise $H$-transversally elliptic symbol $a$ on $(M, \maF)$ and a $G$-invariant leafwise $G$-transversally elliptic symbol $b$ on $(M', \maF')$ a new symbol which  is a leafwise symbol on the cartesian product $(M\times M', \maF\times \maF')$ of the two foliated manifolds, is $G\times H$-invariant and $G\times H$-transversally elliptic. 

More precisely, there is a well defined product for all $i, j\in \Z_2$,
\begin{equation}\label{product}
\k^i_{\mathrm{G\times H}} (F_H) \otimes \k^j_\mathrm{G} (F'_G) \longrightarrow \k^{i+j}_{\mathrm{G\times H}} ((F\times F')_{G\times H}),
\end{equation}
which assigns to $[\sigma]\otimes [\sigma']$ the class of the sharp product $\sigma\sharp \sigma'$ that we proceed to recall now. The cartesian product $F\times F'$ fibers over $M\times M'$ and generates the foliation of $M\times M'$ whose leaf through any given $(m, m')$ is just the cartesian product  $L_m\times L'_{m'}$ of the leaf of $(M, \maF)$ through $m$ by the leaf of $(M', \maF')$ through $m'$. The compact group $G\times H$ acts obviously by leaf-preserving diffeomorphisms of this product foliation and the subspace $(F\times F')_{G\times H}$ of vectors transverse to this action, is well defined. Notice as well that  this product action of $G\times H$ is also a holonomy  action. For the convenience of the reader, let us describe the above product in the case $i=j=0$ for simplicity.  {The other cases can be
deduced by replacing $M$ by $M \times S^1$ with the foliation $\maF\times 0$.} Recall that any class $b$ in $\k_G(F'_G)$ can be represented  by a classical $G$-invariant pseudodifferential symbol $\sigma'$ along the leaves of the foliation $\maF'$, that is defined over $F'$, and whose restriction to $ F'_G\setminus M' $ is pointwise invertible. In the same way, any class $a$ in $\mathrm{K_{G\times H}} (F_H)$ can be represented by a classical $G\times H$-invariant pseudodifferential symbol $\sigma$ along the leaves of the foliation $\maF$, that is defined over $F$, and whose restriction to $F_H \setminus M$ is pointwise invertible. {We may assume that $\sigma$ and $\sigma'$ both have positive order, see \cite{Atiyah-Singer:I} and also  \cite{atiyah1974elliptic}}. The product $a\sharp b$ is then the class in $\mathrm{K_{G\times H}} ((F\times F')_{G\times H})$ which is represented by the leafwise $G\times H$-invariant symbol on the foliation $\maF\times \maF'$ over $M\times M'$ defined by:
$$
\sigma\sharp\sigma' := \begin{pmatrix}\sigma \otimes 1&-1\otimes {\sigma'}^*\\1\otimes \sigma' &\sigma^*\otimes 1 \end{pmatrix}.
$$
This is the standard cup-product formula, used in   \cite{atiyah1974elliptic} where it  adapted the original Atiyah-Singer construction  from the seminal paper \cite{Atiyah-Singer:I} to the transversally elliptic context, and whose extension to the foliation setting is a routine exercise. In particular, the restriction of $\sigma\sharp\sigma' $ to $(F\times F')_{G\times H}\smallsetminus (M\times M')$ is pointwise invertible as allowed and hence represents our announced sharp product.

\medskip


We are now in position to prove the multiplicativity axiom which computes the index of the sharp product $a\sharp b$ in terms of the indices of $a$ and $b$. Notice that  $
C^*(M\times M', \maF\times \maF') \simeq C^*(M, \maF) \otimes C^*(M', \maF').$

\medskip

\begin{thm}\label{thm:multiplicativité:indice}
{For any $i, j\in \Z_2$, the following diagram commutes: 
{\small{$$\xymatrix{\k^i_{\mathrm{G\times H}} (F_H)\otimes \k_\mathrm{G}^j (F'_G)\ar[r]^{\bullet \sharp\bullet} \ar[d]_{\widehat{\Ind}^{\maF, G} \otimes\; \Ind^{\maF'}}~&~\k^{i+j}_{\mathrm{G\times H}} ((F\times F')_{G\times H})\ar[d]^{\Ind^{\maF\times\maF'}}\\
\k\k^i (C^*(G\times H),C^*(M,\mathcal{F})\otimes C^*G) \otimes \k\k^j (C^*G,C^*(M',\mathcal{F}'))\hspace{0cm}\ar[r]^{\hspace{1,5cm}\bullet \underset{C^*G}{\otimes} \bullet} ~&~\hspace{0cm}\k\k^{i+j}(C^*(G\times H), C^*(M\times M',\mathcal{F}\times \maF')).}
$$}}}

In other words, if $a\in \k^i_{\mathrm{G\times H}} (F_H)$ and $b\in \k_\mathrm{G}^j (F'_G)$ and if $a\sharp b\in \k^{i+j}_{\mathrm{G\times H}} ((F\times F')_{G\times H})$ is their sharp product, then  we have
$$
\mathrm{Ind}^{\mathcal{F}\times \mathcal{F}'} (a\sharp b)= \mathrm{\widehat{Ind}}^{\maF,G} (a) \; \underset{C^*G}{\otimes} \; \mathrm{Ind}^{\maF'} (b).
$$
%
%

\end{thm}

\medskip



\begin{proof}
{We treat the case $i=0=j$, the other cases are similar.} If $P_0$ is a longitudinal pseudodifferential operator of positive order then we denote again by $P$  the closure of the formally self-adjoint longitudinal operator $\begin{pmatrix}
0&P_0^*\\P_0&0
\end{pmatrix}$ in the corresponding Hilbert module. We also recall that the Woronowicz transform of $P$ is the adjointable operator $Q(P)=P(1+P^2)^{-1/2}$.

{Let $A_0 : C^{\infty}_c(\mathcal{G},r^*E^+) \rightarrow C^{\infty}_c(\mathcal{G},r^*E^-)$ be a $G\times H$-invariant, leafwise $H$-transversally elliptic operator of order $1$ whose principal symbol represents the class  $a$.  Let similarly $B_0 : C^{\infty}_c(\mathcal{G}',r^*E^{'+}) \rightarrow C^{\infty}_c(\mathcal{G}',r^*E^{'-})$ be a $G$-invariant, leafwise $G$-transversally elliptic operator of order $1$ whose principal symbol lies in the class  $b$. 
The index classes associated respectively  are then by definition 
$$
\mathrm{Ind}^{\mathcal{F},G}(A)=[(\mathcal{E}, \pi_H,Q(A))]\text{ and }\mathrm{Ind}^{\maF'}(b)=[(\mathcal{E}', \pi_G,Q(B))],
$$
where  the first class is $G$-equivariant for the $G^{\rm{trivial}}$-action on $\maE$, i.e. viewed as a Hilbert $G$-module for the trivial $G$-action on $C^*(M, \maF)$ by using the holonomy hypothesis. 
}

{Hence,  the image of  $[(\mathcal{E}, \pi_H,Q(A))]$ under the Kasparov descent is represented, with our previous notations and using Lemma \ref{Compatibility}, by the Kasparov $(C^*H\rtimes G, C^*(M, \maF)\otimes  C^*G)$ cycle 
$$
(\maE\rtimes G^{\rm{trivial}}, \pi_H\rtimes G^{\rm{trivial}}, Q(A)\rtimes G^{\rm{trivial}}).
$$
}

Recall that the action of $G$ on the $C^*$-algebra $C^*(M', \maF')$ is also inner through  unitary multipliers that we denote by $(U'_g)_{g\in G}$. Let $\maU : C (G, \maE)\otimes \maE' \rightarrow \maE\otimes \maE'$ be the map defined by 
$$
\maU (\rho\otimes \eta'):=\int_G \rho({g)\otimes U'_g} \eta' dg, \quad \text{ for }\rho\in C(G, \maE) \text{ and } \eta'\in \maE' .
$$ 
Here the integral makes sense  in the norm topology of the Hilbert module  closure, denoted as usual  $\maE\otimes \maE'$, over the $C^*$-algebra $C^*(M, \maF)\otimes C^*(M', \maF')$. 
From the very definition of the representation $\pi_G$, we easily deduce that for $\varphi \in C (G)$, one has 
$$
\maU(\rho \cdot \varphi \otimes \eta' )=\maU(\rho\otimes \pi_G(\varphi )\eta'),
$$
so that $\maU$ is well defined. Moreover, we can check now that $\maU$ extends to a unitary isomorphism which identifies $(\mathcal{E}\rtimes G^{\rm{trivial}})\underset{\pi_G}{\otimes}\mathcal{E}'$ with the spatial tensor product Hilbert module $\mathcal{E}\otimes \mathcal{E}'$. 
{Indeed, a direct computation shows that $\maU$ is isometric,}
and it is also straightforward to check that  $\maU(C (G, \maE){\otimes} \maE')$ is dense in $\mathcal{E}\otimes \mathcal{E}'$. {Indeed, given $\eta\in \maE$ and $\eta'\in \maE'$, we may use an approximate unit $(e_\alpha)_\alpha$ of the $C^*$-algebra $C^*G$, composed of continuous functions on $G$ which are supported as close as we please to the neutral element of $G$, to see that $\pi_G(e_\alpha)(\eta')$ converges {\underline{in $\maE'$}} to $\eta'$. Hence, the net $
\maU \left( (\eta\otimes e_\alpha) \otimes \eta'\right) = \eta \otimes \pi_G(e_\alpha)(\eta')$ 
converges to $\eta\otimes \eta'$  in the spatial tensor product $\maE\otimes \maE'$. }
 It thus remains to check that $\maU$ intertwines representations and operators, {but this is as well a straightforward verification}. 

{The Kasparov product of $\mathrm{\widehat{Ind}}^{\maF,G}(A)$ and $ \mathrm{Ind}^{\mathcal{F}'}(B)$ can be represented by  the unbounded cycle 
$$
\left((\mathcal{E}\rtimes G^{\rm{trivial}}) \underset{\pi_G}{\otimes}\mathcal{E}' , (\pi_H \rtimes G^{\rm{trivial}}) \underset{\pi_G}{\otimes}  \id, (A  \rtimes G^{\rm{trivial}})\underset{\pi_G}{\otimes}\id +\id \underset{\pi_G}{\otimes} B\right)
$$ 
where} the operator $\id\underset{\pi_G}{\otimes} B$ is well defined here since $B$ commutes strictly  with the representation $\pi_G$.  
{Indeed, although this is not an external Kasparov product, this strict commutation allows to apply the argument given in \cite{baaj1983theorie}[Lemma 3.1 $\&$ Theorem 3.2] which adapts mutatis mutandis to our situation. We thank the referee for pointing out this observation to us.}

\end{proof}

\section{Reduction to tori actions}\label{NaturalityInduction}

We now use the previous axioms  to investigate the induction property of our index morphism with respect to closed subgroups, and then more specifically to a maximal torus.

{We recall first some standard constructions from \cite{atiyah1974elliptic}. Let $G$ be a compact \underline{connected} Lie group and let $H$ be a closed subgroup of $G$. Denote by $i : H \hookrightarrow G$ the inclusion. 
Then the functoriality class  $[i] \in \k\k(C^*G, C^*H)$ is defined as follows, see \cite{julg1982induction}. We fix Haar measures on $H$ and $G$ and consider the right $L^1(H)$-module structure on the space $C(G)$, which is induced by the right action of  $H$ on $G$. More precisely,  we set for $f \in C(G)$ and  $\psi \in L^{1}(H)$:
$$
f \cdot \psi (g)=\int_H f(gh^{-1}) \psi(h) \; dh,
$$
and define the $L^1(H)$-valued hermitian structure by setting for $f_1, f_2\in C(G)$:
$$
\langle f_1,f_2 \rangle (h)=\int_G \overline{f_1}(g)f_2(gh)\; dg.
$$
The completion  of this prehilbertian $L^1(H)$-module  is then a Hilbert $C^*H$-module that we shall denote by  $J(G, H)$. 
The left action of $G$ on itself by translation allows to define, after completing, the representation $\pi_G : C^*G \rightarrow \mathcal{L}_{C^*H} (J(G, H))$. The triple $(J(G, H),\pi_G ,0)$ is then a Kasparov cycle over the pair of $C^*$-algebras $(C^*H, C^*G)$, see again  \cite{julg1982induction}.}

{\begin{defi}\cite{julg1982induction}\label{Def-[i]}
The functoriality class $[i]$ is the  class of the Kasparov cycle $(J(G, H),\pi_G ,0)$, i.e.
\begin{equation}\label{LaClasse-i*}
[i] := [(J(G, H),\pi_G ,0)] \; \in \; \k\k(C^*G,C^*H).
\end{equation}
\end{defi}}

{\begin{remarque}
Since the crossed product $C^*$-algebra $C(G/H)\rtimes G$ for the induced left action of $G$ on the homogeneous manifold $G/H$, is Morita equivalent to $C^*H$, it is easy to reinterpret the class $[i]$ as  the class induced {through the descent map for the $G$-action, by the trivial representation of $H$ viewed  as a trivial $G$-equivariant vector bundle over $G/H$}. 
\end{remarque}}
%
%

{Since the underlying closed manifold $G$ is endowed with the $G\times H$-action given by $(g,h)\cdot g'=gg'h^{-1}$  for $g,g'\in G$ and $h\in H$, we may use the product defined in Equation \eqref{product}  for any given smooth foliation $\maF$ on a  closed manifold $M$ as soon as this latter is endowed  with a smooth leaf-preserving $H$-action, {which is a holonomy action.} Indeed, we are considering here the trivial top-dimensional foliation on $G$ and we thus get  the following product for $j\in \Z_2$
\begin{equation}\label{productG}
\k_{\mathrm{G\times H}}\big(T_GG \big)\otimes \k^j_\mathrm{H}\big(F_H\big)  \longrightarrow \k^{j}_{\mathrm{G\times H}}\big((TG\times F)_{G\times H}\big).
\end{equation}
The space $T_GG$ is just $G\times \{0\}\simeq G$, and hence since $H$ acts freely on $G$:
$$
\k_{\mathrm{G\times H}} (T_GG) \simeq \k_\mathrm{G} (G/H) \simeq  R(H).
$$
Moreover, $H$ also acts freely on the cartesian product $G\times M$ preserving the product foliation $TG\times F$ and the quotient manifold $Y:=G\times_H M$ inherits  a foliation that we denote by  $\mathcal{F}^Y$ and which is automatically endowed with the action of $G$ by holonomy diffeomorphisms, as can be checked easily. The receptacle group $\k^{j}_{\mathrm{G\times H}}\big((TG\times F)_{G\times H}\big)$ in \eqref{productG} is then given by 
\begin{equation}\label{Iso-q*}
\k^j_{\mathrm{G\times H}}\big((TG\times F)_{G\times H}\big) \simeq \k^j_\mathrm{G} (F^Y_G).
\end{equation}
{Notice that the space $F^Y_G=G\times_H F_H$ is $G$-equivariantly Morita equivalent as a groupoid to $(G\times F_H)\rtimes H$ and  we deduce the following list of Morita equivalences
$$
F^Y_G\rtimes G \sim [(G\times F_H)\rtimes H]\rtimes G \simeq [(G\times F_H)\rtimes G]\rtimes H\sim [(G\times F_H)/ G]\rtimes H \simeq F_H\rtimes H.
$$
In particular, the  group $\k^j_\mathrm{G} (F^Y_G)$ is isomorphic to the group $K^j_\mathrm{H}(F_H)$, the isomorphism $i_* : \k^j_\mathrm{H}(F_H) \longrightarrow \k^j_\mathrm{G} (F^Y_G)$ is given explicitely as follows. There is a privileged element in the group $K_{\mathrm{G\times H}} (T_GG)$ which corresponds to the class, in $R(H)$, of the trivial representation of $H$. This class is in fact the class of the $G\times H$-equivariant $G$-transversally elliptic symbol on $G$, associated  with the zero operator $0 : C^{\infty}(G) \rightarrow 0$. The product in \eqref{productG} by this trivial class yields the allowed isomorphism
\begin{equation}\label{Iso-i*}
i_* : \k^j_\mathrm{H}(F_H) \longrightarrow \k^j_\mathrm{G} (F^Y_G).
\end{equation}
As we prove below, this isomorphism allows to reduce the index problem for leafwise $H$-transversally elliptic operators on foliated $H$-manifolds to the index problem for leafwise $G$-transversally elliptic operators on foliated $G$-manifolds. Notice that $Y$ is the base of the principal $H$-fibration $G\times M \to Y$ and we are exactly in position to apply the properties of the index morphism with respect to free actions, see Subsection \ref{Section:FreeActions}. Furthermore, since the compact Lie group $G$ is assumed to be connected here, the $C^*$-algebra upstairs, that is $C^*(G\times M, G\times \maF)$ is Morita equivalent, and in fact isomorphic when $\maF$ is not the zero foliation \cite{BenameurPacific, HilsumSkandalis1983Stability}, to $C^*(M, \maF)$. Hence we end up with a $\k\k$-equivalence that we denote by $\epsilon\in \k\k (C^*(M, \maF), C^*(Y, \maF^Y))$.  }

\begin{thm} \cite{atiyah1974elliptic} \label{thm:induction:1}
{For $j\in \Z_2$, the following diagram commutes}

{$$
\xymatrix{\k^j_\mathrm{H}\big(F_H \big) \ar[rrr]^{i_*}\ar[d]_{\ind } &&& \k_\mathrm{G}^j (  F^Y_{G} ) \ar[d]^{\mathrm{Ind}^{\mathcal{F}^Y}}\\
\k\k^j \big(C^*H,C^*(M,\mathcal{F})\big) \ar[rrr]_{[i]\underset{C^*H}{\otimes} \hspace*{0.2cm}\bullet \underset{C^*(M,\mathcal{F})}{\otimes} \epsilon }&&&\k\k^j \big(C^*G ,C^*(Y,\mathcal{F}^Y)\big).}
$$}
\end{thm} 

\begin{proof}
Recall  the Kasparov class  
$$
\mathcal{E}_q \; \in \; \k\k\left(C^*(G\times M, G\times \maF), C^*(Y, \maF^Y)\right) \simeq \k\k\left( \mathcal{K}(L^2(G)) \otimes C^*(M, \maF), C^*(Y, \maF^Y)\right)
$$ 
introduced in Section \ref{Section:FreeActions} and associated here to the principal $H$-fibration $q : G\times M \rightarrow Y = G\underset{H}{\times} M$. If we denote by  $\mu (G) \in \k\k(\mathbb{C}, \mathcal{K}(L^2(G)))$ the standard  $\k\k$-equivalence then we have by definition  
$$
\epsilon = \mu (G)\underset{\mathcal{K}(L^2(G))}{\otimes} \maE_q.
$$
Let now $a\in \k^j_\mathrm{H}(F_H)$ be fixed. By Theorem \ref{thm:action:libre}, we know that
$$
\Ind ^{\maF^Y} (i_*a) = \chi_1^H \underset{C^*H}{\otimes}\mathrm{Ind}^{G\times \mathcal{F}}(q^* (i_*a)) \underset{C^*(G\times M,G\times\mathcal{F})}{\otimes} \mathcal{E}_q,
$$
where $\chi_1^H\in \k\k (\C, C^*H)$ is the class of the trivial representation of $H$ and where in the present case $q^* (i_*a)$ is just the isomorphic class to $i_*a$ through the identification \eqref{Iso-q*} and thus coincides by definition of $i_*$ with $[\sigma (0)]\cdot a$ in the product \eqref{productG}. Thus
$$
\mathrm{Ind}^{G\times \mathcal{F}}(q^* (i_*a)) =  \mathrm{Ind}^{G\times \mathcal{F}}([\sigma (0)]\cdot a)  \; \in \; \k\k^j (C^*H\otimes C^*G, \mathcal{K}(L^2(G)) \otimes C^*(M, \maF)).
$$ 
We can now apply the multiplicative property of the index from Theorem \ref{thm:multiplicativité:indice} to compute
$$
\mathrm{Ind}^{G\times\mathcal{ F}}([\sigma (0)] \cdot a ) =\widehat{\mathrm{Ind}}^{\mathrm{G},H }([\sigma (0)])\underset{C^*H}{\otimes} \ind (a).
$$
For simplicity the  $\k\k$-equivalence  class $\mu (G)$ is often removed from the formulae, it is only used to naturally  identify, in $K$-theory, $\maK (L^2(G))$ with $\C$. 
The index class $\widehat{\mathrm{Ind}}^{\mathrm{G} , H }([\sigma (0)]) \in \k\k (C^*G\otimes C^*H, C^*H)$
 reduces here to the image under the Kasparov descent map, for the trivial $H$-action,  of the $H$-equivariant index class in $\k\k_\mathrm{H}(C^*G,\mathbb{C})$, of the $G$-transversally elliptic operator $0 : C^{\infty}(G) \rightarrow 0$.
  {By gathering the previous equalities, we finally get
 \begin{eqnarray*}
 \mathrm{Ind}^{\mathcal{F}^Y}(i_*(a)) & = & \chi_1^H\underset{C^*H}{\otimes}\mathrm{Ind}^{G\times \mathcal{F}}(q^* (i_*a)) \underset{C^*(G\times M,G\times\mathcal{F})}{\otimes} \mathcal{E}_q\\
 & = & \chi_1^H\underset{C^*H}{\otimes} \left[\widehat{\mathrm{Ind}}^{\mathrm{G},H }([\sigma (0)])\underset{C^*H}{\otimes} \ind (a)\right] \underset{C^*(G\times M,G\times\mathcal{F})}{\otimes} \mathcal{E}_q\\
  & = &\chi_1^H\underset{C^*H}{\otimes} \left[\widehat{\mathrm{Ind}}^{\mathrm{G},H }([\sigma (0)])\underset{C^*H}{\otimes} \ind (a)\underset{\C}{\otimes} \mu (G) \right] \underset{C^*(G\times M,G\times\mathcal{F})}{\otimes} \mathcal{E}_q\\
 & = &\left(\chi_1^H\underset{C^*H}{\otimes} \widehat{\mathrm{Ind}}^{\mathrm{G},H }([\sigma (0)])\right) \underset{C^*H}{\otimes} \ind (a)\underset{C^*(Y, \mathcal{F}^Y)}{\otimes} \left( \mu (G)\otimes_{\mathcal{K}(L^2(G))} \maE_q\right)
 \end{eqnarray*}
 where we have used associativity of the Kasparov product. The proof is now complete since we have
 $$
 \chi_1^H\underset{C^*H}{\otimes} \widehat{\mathrm{Ind}}^{\mathrm{G},H }([\sigma (0)]) = [i] \text{ and } \mu (G) \otimes_{\mathcal{K}(L^2(G))} \maE_q = \epsilon.
 $$
}
\end{proof}

%


\begin{remarque}
{Theorem \ref{thm:induction:1} allows to extract information on the index morphism for the action of the compact Lie group $H$ using all such compact connected Lie groups $G$ and their induced actions on the Morita equivalent foliation $(Y, \maF^Y)$. Such $G$ always exists as any compact Lie group is isomorphic to a closed subgroup of a unitary group. }
\end{remarque}

\bigskip

We  fix for the rest of this section  a compact connected Lie group $G$ and a smooth closed foliated manifold which is endowed with an action of $G$ by leaf-preserving diffeomorphisms. For simplicity, we shall denote this new $G$-foliation again by $(M, \maF)$ since we shall again need to build up the new foliation $(Y, \maF^Y)$ by using a particular closed subgroup of $G$,  so no confusion should occur.  Since $G$ is connected this action is a holonomy action and we may apply all the results of the previous sections. In order to compute the index morphism for leafwise $G$-transversally elliptic operators, we shall use a maximal torus $\T$ in $G$ and we use the induced action of $\T$ to define the Morita equivalent $G$-foliation  $(Y, \maF^Y)$ as explained above. However, since the action of $\T$ on $(M, \maF)$ is now the restriction of an action of the whole group $G$, this foliation is easier to describe. More precisely, the map $(g,m) \rightarrow (gH,g\cdot m)$ is a $G$-equivariant diffeomorphism which allows to identify  the foliation $(Y, \maF^Y)$ with the foliation  $(G/\T\times M, G/\T\times \maF)$. We quote for later use that  $C^*(Y, \maF^Y)$ coincides here with $C^*(M, \maF)\otimes \maK (L^2(G/\T))$ which in turn, when $\maF$ is not the zero foliation, is even isomorphic to $C^*(M, \maF)$. Notice also that there is hence a  well defined product
\begin{equation}\label{Product}
{\k_\mathrm{G}^j (F_G) \otimes \k_\mathrm{G} (T(G/\T)) \longrightarrow \k_\mathrm{G}^j (F^Y_G).}
\end{equation}

Recall that $G/\T$ carries a {\underline{$G$-invariant}} complex structure and we may  use the Dolbeault operator $\overline{\partial}$. This is an elliptic $G$-invariant operator on {the rational variety $G/\T$ whose $G$-index equals  $1\in R(G)$ since only the zero-degree Dolbeault cohomology space is non trivial}, see \cite{atiyah1974elliptic},  i.e. 
$$
\mathrm{Ind}(\overline{\partial})=1 \; \in \; R(G).
$$
{The product by  the symbol class $[\sigma (\overline{\partial})]\in \K(T(G/H))$ in \eqref{Product} allows to  define the morphism} 
$$
\beta : \K(F_G) \longrightarrow \K (F^Y_G).
$$

{Recall the isomorphism $i_*$ defined in Equation \eqref{Iso-i*} as well as the $\k\k$-class $[i]$ introduced in Definition \ref{Def-[i]}.  We use these notations for the torus closed subgroup $H=\T$ to state the following}

\medskip

\begin{thm}\label{thm-induction}
Let $\T$ be a maximal torus of the compact connected Lie group $G$. Denote by 
$r^G_\T: \k_\mathrm{G}^j(F_G)\rightarrow \k^j_{\T}(F_\T)$ the composite map $r^G_\T:= (i_*)^{-1}\circ \beta$. Then for $j\in \Z_2$ the following diagram commutes: 
$$\xymatrix{\k_\mathrm{G}^j(F_G) \ar[r]^{r^G_\T} \ar[d]_{\mathrm{Ind}^{\mathcal{F}}}& \k^j_\T (F_\T)\ar[d]^{\mathrm{Ind}^{\mathcal{F}}}\\
\;\;\;\; \k\k^j(C^*G,C^*(M,\mathcal{F})) \;\;&\;\;\; \k\k^j(C^*\T,C^*(M,\mathcal{F})).\;\; \;\;  \ar[l]^{[i] \underset{C^*\T}{\otimes} \bullet }
}$$
\end{thm}

\begin{proof}\ {We apply  the multiplicative property of our index morphism from Theorem \ref{thm:multiplicativité:indice}. In the notations of Theorem \ref{thm:multiplicativité:indice},  we take for $H$ the trivial group, for $(M, \maF)$ the $G$-manifold $G/\T$ with one leaf, and for $(M', \maF')$ our $G$-foliation here, that is the foliation $(M, \maF)$ used in the statement of Theorem \ref{thm-induction}. Then we obtain the commutativity of the following diagram (recall that $C^*(Y, \maF^Y)=\maK(L^2(G/\T))\otimes C^*(M, \maF)$ and hence can be replaced by $C^*(M, \maF)$):
\small{$$\xymatrix{\k_\mathrm{G} (T(G/\T))\otimes \k_\mathrm{G}^j (F_G)\ar[r]^{\hspace{1,2cm}\bullet \sharp\bullet} \ar[d]_{\widehat{\Ind}^{G} \otimes\; \Ind^{\maF}}~&~\k^j_\mathrm{G} (F^Y_G)\ar[d]^{\Ind^{\maF^Y}}\\
\k\k(C^*G, C^*G) \otimes \k\k^j(C^*G,C^*(M, \mathcal{F}))\ar[r]^{\hspace{1,2cm}\bullet \underset{C^*G}{\otimes} \bullet} ~&~\hspace{0cm}\k\k^j(C^*G,  C^*(M, \mathcal{F})).}
$$}

We recall that $\widehat{\Ind}^{G} = j^G\circ {\Ind}^{G}$ where  ${\Ind}^{G}: \k_\mathrm{G} (T(G/\T)) \rightarrow R(G)\simeq \k\k_\mathrm{G}(\C, \C)$ is the usual Atiyah-Singer $G$-index that we view as valued in the Kasparov group  $\k\k_\mathrm{G}(\C, \C)$ and  $j^G: \k\k_\mathrm{G}(\C, \C) \to \k\k (C^*G, C^*G)$ is the Kasparov descent map for the trivial $G$ action on $\C$. In particular, $\widehat{\Ind}^{G} (\overline{\partial})$ coincides with the unit of the ring $\k\k (C^*G, C^*G)$. 
If we thus apply this multiplicativity result to a given $a\in \k_\mathrm{G}^j(F_G)$ and to the Dolbeault symbol, then we get
$$
\mathrm{Ind}^{\mathcal{F}^Y}(\beta (a)) = \widehat{\Ind}^{G} (\overline{\partial})  \underset{C^*G}{\otimes} \ind(a) = \ind (a) \text{ and so }\mathrm{Ind}^{\mathcal{F}^Y}\circ \beta = \ind.
$$
The proof is now complete since we already proved in Theorem \ref{thm:induction:1}  the compatibility of the index morphism with the map $i_*$.}
\end{proof}

\section{Naturality of the index morphism}\label{Naturality}

We now apply the previous results  to give the allowed topological construction of an index map which will be compared with  our analytical index map from Proposition \ref{G-index}.

\subsection{Compatibility with Gysin maps}

{Let $\iota : (M, \mathcal{F}) \hookrightarrow (M',\mathcal{F}')$ be a foliated embedding  of $G$-foliations. So we assume  that the compact Lie group acts on $M$ and on $M'$ by leaf-preserving holonomy diffeomorphisms and that  $\iota: M\hookrightarrow M'$ is a $G$-equivariant embedding which sends leaves inside leaves.  We assume for simplicity that $M$ is compact, since this is the only needed situation for the proof of our index theorem. We denote by $N:=\iota^*TM'/TM$ the normal bundle to $\iota$. In view of the construction of the topological index in Subsection \ref{TopIndex}, we shall only need the case where the transverse bundles $\tau:=TM/F$ and $\tau'=TM'/F'$ do fit under $\iota$, i.e.  that $\iota^*\tau' \simeq \tau$. 
{In other words, the manifold $M$ embeds transversally in $M'$ to the foliation $\maF'$, i.e. $\forall x\in M$, $F_{\iota(x)} + d\iota(T_x M)=T_{\iota(x)}M'$ and the foliation $\maF=\iota^*\maF'$ is the pull back foliation, see \cite{Connes:surveyFoliation,ConnesSkandalis} for more details.}
 As a consequence, the $G$-equivariant  embedding $d\iota: F\to F'$, obtained  by differentiating $\iota$ and restricting to $F$, is $K$-oriented by a $G$-equivariant complex structure. Indeed, under this assumption, the normal bundle $N$ is identified with the normal bundle to the leaves of $\maF$ inside the leaves of $\maF'$, and it is easy then to see that the normal bundle $N'$ to $d\iota$ is isomorphic to the bundle $\pi_F^* (N\otimes \C)$ with $\pi_F: F\to M$ being the bundle projection. Following \cite{atiyah1974elliptic},  we deduce for any $j\in \Z_2$, a well defined Thom $R(G)$-morphism  
$$
\iota_! : \k^j_\mathrm{G}(F_G) \longrightarrow \k^j_\mathrm{G}(F'_G).
$$}
More precisely, denote by $\pi:N'\to F$ the bundle projection of the normal bundle $N'$ to $F$ in $F'$, and let $(\pi_F\circ \pi)^*(\Lambda^\bullet (N\otimes \mathbb{C}))$ be the associated exterior algebra over $N'$. Together with exterior multiplication by the underlying vector, this defines a complex over $N'$  which is exact off the zero section $F\subset N'$ and which is denoted $\lambda (N\otimes \C)$.  The usual Thom isomorphism $\K(F) \to \K(N')$ is defined by assigning to a given compactly supported $G$-complex $(E, \sigma)$ over $F$ the compactly supported $G$-complex over $N'$ given by $\pi^*(E,\sigma)\cdot \lambda (N\otimes \C)$. See \cite{Atiyah-Singer:I} for more details. On the other hand, the total space of the fibration $\pi:N'\to F$ is $G$-equivariantly diffeomorphic to a $G$-stable open tubular neighborhood $p:U'\to F$ of $d\iota (F)$ in $F'$ and this allows to define classically the Gysin map $\iota_!: \K (F)\to \K(F')$. As explained in \cite{atiyah1974elliptic}, if we only assume that $(E, \sigma)$ represents a class in $\K(F_G)$, then  the complex $\pi^*(E,\sigma)\cdot \lambda (N\otimes \C)$ over $N'$ extends to an element of $\K(F'_G)$. 
More precisely,  if we assume that $(E, \sigma)$ is only compactly supported when restricted to $F_G$, that is $\Supp (E, \sigma)\cap F_G$ is compact, then the $G$-complex $\pi^*(E,\sigma)\cdot \lambda (N\otimes \C)$ yields a compactly supported $G$-complex over an open subspace $U'_G$ of $F'_G$ defined as follows. If we identify similarly the total space $N$ with a $G$-stable open tubular neighborhood $U$ of $\iota (M)$ in $M'$, then the foliation $\maF'$ induces by restriction to the open submanifold $U$ a foliation $\maF^U$. Then $U'$ can be naturally identified with the total space  $F^U$ of the leafwise tangent bundle of the foliation $\maF^U$. The subspace $U'_G$ is then simply $F^U_G= F^U \cap T_G U$. To sum up, we deduce 
in this way a well defined Thom homomorphism of $R(G)$-modules $\K (F_G) \longrightarrow \K (U'_G)$ (see again \cite{atiyah1974elliptic}).
Since $U'_G$ is an open subspace of the locally compact space $F'_G$, the $C^*$-algebra $C_0(U)$ is a $G$-stable ideal in the  $G$-algebra $C_0(F_G)$ and we have the extension $R(G)$-morphism $\K (U'_G) \rightarrow \K (F'_G)$. Composing the Thom homomorphism with this extension map, we end up with our Gysin $R(G)$-morphism
$$
\iota_! : \K (F_G) \longrightarrow \K (F'_G).
$$
Starting with a class in $\k_\mathrm{G}^1(F_G)$ we get in the same way a class in $\k_\mathrm{G}^1(F'_G)$ and we finally get the morphism
$$
\iota_! : \k^j_\mathrm{G} (F_G) \longrightarrow \k^j_\mathrm{G} (F'_G)\text{ for }j\in \Z_2.
$$

%
%

The $G$-embedding $\iota$ gives a submersion $M\to M'/\maF'$ in the sense of \cite{ConnesSkandalis}, we hence deduce from  \cite[Section 4]{ConnesSkandalis} the well defined Connes-Skandalis Morita extension element $\epsilon_\iota \in \k\k (C^*(M, \maF), C^*(M', \maF'))$. Indeed, the submanifold $\iota (M)$ is  a transverse  $G$-submanifold in $(M', \maF')$ which inherits a foliation $\maF^{\iota (M)}$ which is diffeomorphic to $(M, \maF)$, hence identifying $C^*(M, \maF)$ with $C^*(\iota (M), \maF^{\iota(M)})$ and using the $G$-equivariant Morita equivalence of $(\iota (M), \maF^{\iota(M)})$ with a foliation $(U, \maF^U)$ obtained as an open tubular neighborhood of $\iota(M)$ in $M'$, we get the easy definition of the Connes-Skandalis map in our case.

\medskip

\begin{thm}\label{thm:Gysin:!}
{Let $(M', \maF')$ be a smooth $G$-foliation. Let $\iota: M \hookrightarrow M'$ be a $G$-equivariant embedding of a closed $G$-manifold $M$ which is transverse to the foliation $\maF'$ and denote by $\maF=\iota^*\maF'$  the pull back foliation. We assume that    $G$ acts  by leaf-preserving holonomy diffeomorphisms on the foliations $(M, \maF)$ and $(M', \maF')$.} Then for any $j\in \Z_2$, the following diagram commutes:
$$
\xymatrix{\k^j_\mathrm{G}(F_G) \ar[r]^{\iota_!} \ar[d]_{\Ind^\maF}& \k^j_\mathrm{G}(F'_G) \ar[d]^{\mathrm{Ind}^{\mathcal{F}'}} \\
\k\k^j(C^*G ,C^*(M,\mathcal{F}))\; \; \ar[r]_{\underset{C^*(M, \maF)}{\otimes}\epsilon_\iota}&\; \;    \k\k^j(C^*G ,C^*(M',\mathcal{F}')).}$$
Here  the index morphism $\mathrm{Ind}^{\mathcal{F}'}$ is  defined  according to Proposition \ref{Index-intrinseque}.
\end{thm}

\begin{proof}\
{For simplicity, we shall identify $(\iota (M), \maF^{\iota(M)})$ with $(M, \maF)$ and assume that $M$ is a smooth transverse submanifold to the foliation $(M', \maF')$ whose foliation $\maF$ coincides with the restricted foliation generated by $T\maF'\vert_M\cap TM$. By definition, $\iota_!$  is the composite map of a Thom morphism from $\K (F_G)$ to $\K(F^U_G)$ with the  extension corresponding to the inclusion of the open submanifold $U$. Notice that since the $G$-action on $(M, \maF)$ is a holonomy action, it is also a holonomy action on $(U, \maF^U)$. By the excision theorem, the index morphism does automatically respect the latter  extension map and the following diagram commutes: 
$$
\xymatrix{\k^j_\mathrm{G}(F^U_G) \ar[r]^{} \ar[d]_{\Ind^{\maF^U}}& \k^j_\mathrm{G}(F'_G) \ar[d]^{\mathrm{Ind}^{\mathcal{F}'}} \\
\k\k^j(C^*G ,C^*(U,\mathcal{F}^U))\; \; \ar[r]^{}&\; \;    \k\k^j(C^*G ,C^*(M',\mathcal{F}')).}
$$
Now the open submanifold $U$ can be identified with a vector bundle $N\to M$ which is the normal bundle to $M$ in $M'$. Moreover, the foliation $\maF^U$ is then identified with the foliation of $N$ whose leaves are given by the total spaces of the restrictions of the bundle $\pi_N: N\to M$ to the leaves of $(M, \maF)$. It is thus sufficient to show the theorem in the case of a real $G$-vector bundle $N$ over $M$ foliated by $F^N:= \ker(TN \rightarrow TM/F)\simeq  \pi_N^*(F\oplus N)$ and with $\iota_! : \k_\mathrm{G} (F_G) \rightarrow \K (F^N_G)$ being the Thom homomorphism associated to the $0$-section $\zeta : M \rightarrow N$.
Following \cite{atiyah1974elliptic}, we can write $N = P\times_{O(n)} \mathbb{R}^n$, where $q_1:P\to M$ is a $G$-equivariant $O(n)$-principal bundle over $M$, foliated by $F^P:=\ker(TP \rightarrow TM/F)$. Denote by $q_2 : P \times \mathbb{R}^n \rightarrow N$ the  $G$-equivariant projection corresponding to moding out by the action of $O(n)$. 
Let  $\mathcal{F}^{P\times \mathbb{R}^n}$ be the foliation given by $\maF^P \times \mathbb{R}^n$ on $P \times \mathbb{R}^n$ and by $F^{P\times \mathbb{R}^n}$ the tangent bundle to this foliation. }
%
%
%

By using the product defined in \eqref{product} with $G$ replaced by $G\times O(n)$ and with trivial $H$, we obtain  the well defined product: 
\begin{equation}\label{productRnO(n)}
\k^j_{\mathrm{G\times O(n)}}(F^P_{G\times O(n)}) \otimes \k_\mathrm{G\times O(n)}(T\mathbb{R}^n)\longrightarrow \k^j_\mathrm{G\times O(n)}(F^{P\times \mathbb{R}^n}_{G\times O(n)}).
\end{equation}
But since $O(n)$ acts freely on $P$, we also have  the following identifications:
$$
q_1^* : \k^j_\mathrm{G}(F_G) \stackrel{\simeq} {\longrightarrow}\k^j_\mathrm{G\times O(n)}(F^P_{G\times O(n)})\text{ and }
q_2^* : \k^j_\mathrm{G} (F^N_G) \stackrel{\simeq}{\longrightarrow} \k^j_\mathrm{G\times O(n)}(F^{P\times \mathbb{R}^n}_{G\times O(n)}).
$$
Therefore, we end up with the product: 
\begin{equation}\label{productRn}
\k^j_\mathrm{G}(F_G) \otimes \k_\mathrm{G\times O(n)} (T\mathbb{R}^n)\longrightarrow \k^j_\mathrm{G} (F^N_G).
\end{equation}
Since $G$ acts trivially on $\R^n$, the inclusion $i:\{0\}\hookrightarrow \R^n$ induces the Bott morphism $i_! : R(G\times O(n))\rightarrow \k_\mathrm{G\times O(n)}(T\mathbb{R}^n)$ and we have $\mathrm{Ind}(i_!(1))=1 \in \k\k_\mathrm{G\times O(n)}(\mathbb{C},\mathbb{C})$, see \cite{Atiyah-Singer:I}. Now, multiplication by $i_!(1)$ in \eqref{productRn} is exactly the Thom morphism that we denote $\zeta_!$ since $\zeta$ is the zero section here, i.e.
$$
\zeta_!: \k^j_\mathrm{G} (F_G) \longrightarrow \k^j_\mathrm{G}(F^N_G).
$$
Moreover, we may as well consider the multiplication by $i_!(1)$ in the product \eqref{productRnO(n)}, and we then obviously have the following commutative diagram:
$$\xymatrix{ \k^j_\mathrm{G}(F_G) \ar[r]^{\zeta_!} \ar[d]_{q_1^*}& \k^j_\mathrm{G}(F^N_G)\ar[d]^{q_2^*} \\
\k^j_\mathrm{G\times O(n)}(F_{G\times O(n)}^P)\ar[r]_{\cdot i_!(1)} & \k^j_\mathrm{G\times O(n)}(F_{G\times O(n)}^{P\times \mathbb{R}^n}) 
}$$
We deduce that for any $a\in \k^j_\mathrm{G}(F_G)$:
$$
\mathrm{Ind}^{\maF^{P\times \mathbb{R}^n}}(q_2^*(\zeta_!(a))=\mathrm{Ind}^{\maF^{P\times \mathbb{R}^n}} (q_1^*(a)\cdot i_!(1)) = j^{G\times O(n)}\big(1\big) \underset{C^*(G\times O(n))}{\otimes}\mathrm{Ind}^{\mathcal{F}^\mathrm{P}}(q_1^*(a)) \otimes \mu(\mathbb{R}^n),
$$ 
where  the last equality is a consequence of the multiplicativity axiom satisfied by our index morphism, as stated in Theorem \ref{thm:multiplicativité:indice}, and where $\mu(\mathbb{R}^n) \in \k\k(\C,C^*(\R^n\times \R^n))$ is the Morita equivalence.
On the other hand,  by the axiom for free actions stated in Theorem \ref{thm:action:libre}, and denoting by $\chi_1^{O(n)}$  the trivial representation of $O(n)$, we  have:
\begin{multline*}\ind (a )=\chi_1^{O(n)} \underset{C^*O(n)}{\otimes} \mathrm{Ind}^{\mathcal{F}^\mathrm{P}}(q_1^*(a)) \underset{C^*(P,\maF^P)}{\otimes}\mathcal{E}_{q_1}, \text{ while }\\
$$\mathrm{Ind}^{\mathcal{F}^\mathrm{N}}(\zeta_!(a))=\chi_1^{O(n)}\underset{C^*O(n)}{\otimes} \mathrm{Ind}^{\mathcal{F}^\mathrm{P\times \mathbb{R}^n}}(q_2^*(\zeta_!(a))\underset{C^*(P \times \R^n,\maF^{P \times \mathbb{R}^n})}{\otimes} \maE_{q_2}.
\end{multline*}
We finally conclude by gathering the previous relations as follows:
\begin{eqnarray*}
\mathrm{Ind}^{\mathcal{F}^\mathrm{N}}(\zeta_!(a))&= & \chi_1^{O(n)} \underset{C^*O(n)}{\otimes}\bigg( \mathrm{Ind}^{\mathcal{F}^\mathrm{P}}(q_1^*(a)) \otimes \mu(\mathbb{R}^n)\bigg)\underset{C^*(P \times \R^n,\maF^{P \times \mathbb{R}^n})}{\otimes} \maE_{q_2},\\
&=& \chi_1^{O(n)} \underset{C^*O(n)}{\otimes} \mathrm{Ind}^{\mathcal{F}^\mathrm{P}}(q_1^*(a)) \underset{C^*(P,\maF^P)}{\otimes} \big(\mu(\mathbb{R}^n)\underset{C^*(\R^n,\mathbb{R}^n)}{\otimes} \maE_{q_2}\big)\\
& = &  \chi_1^{O(n)} \underset{C^*O(n)}{\otimes} \mathrm{Ind}^{\mathcal{F}^\mathrm{P}}(q_1^*(a)) \underset{C^*(P,\maF^P)}{\otimes} \big(\mathcal{E}_{q_1}\underset{C^*(M,\mathcal{F})}{\otimes}\mathcal{E}_{\zeta} \big)\\
& = & \left(\chi_1^{O(n)} \underset{C^*O(n)}{\otimes} \mathrm{Ind}^{\mathcal{F}^\mathrm{P}}(q_1^*(a)) \underset{C^*(P,\maF^P)}{\otimes} \mathcal{E}_{q_1}\right) \underset{C^*(M,\mathcal{F})}{\otimes}\mathcal{E}_{\zeta}\\
& = & \ind (a )\underset{C^*(M,\mathcal{F})}{\otimes}\mathcal{E}_{\zeta}.
\end{eqnarray*}
We have used that  $\mathcal{E}_{q_1}\underset{C^*(M,\mathcal{F})}{\otimes}\mathcal{E}_{\zeta} = \mu(\R^n) \underset{C^*(\R^n,\mathbb{R}^n)}{\otimes}  \mathcal{E}_{q_2}$ which is a consequence of the equality $
\zeta\circ q_1 = q_2\circ s_0$ where $s_0: P \hookrightarrow P \times \R^n$ is the zero section of this trivial bundle.

\end{proof}

{
\subsection{A  topological index morphism}}\label{TopIndex}

We prove the following important proposition.

\medskip

\begin{prop}\label{Factorisation}\
{Let  $(M, \maF)$ be a smooth foliated  riemannian $G$-manifold such that $G$ acts by leaf-preserving holonomy diffeomorphisms. Assume that we are given an isometric $G$-embedding  $ i : M\hookrightarrow E$ of $M$ in a finite dimensional euclidean $G$-representation $E$. 
Let $A:=\{(x,\xi,\eta) \in M\times TE, \eta \in di_x(F_x)^\perp\}$ and $\iota : F \hookrightarrow A$ be the $G$-embedding given by $\iota(x,\xi)=(x,\xi,0)$.
Then
\begin{enumerate}
\item The map  $\iota$ is $\k$-oriented by a complex $G$-structure;
\item $A$ is diffeomorphic to a smooth $G$-submanifold $\maA$ of the cartesian product $M\times T(E)$, which  is an open   transversal  to the smooth foliation $\maF\times 0$;
\item Setting $\maA_G:=\maA\cap \left(M\times T_G(E)\right)$,  the usual Thom construction yields, for $j\in \Z_2$, a well defined $R(G)$-homomorphism $\iota_! \; : \; \k^j_\mathrm{G}(F_G) \longrightarrow  \k^j_\mathrm{G}(\maA_G)$;
%
\item The space $\maA_G$ is a topological ($G$-stable) transversal to the {foliated space} $(M\times T_G(E), \maF\times 0)$.
 \end{enumerate}
}
\end{prop}
}

\begin{proof}
{A direct inspection shows that the normal bundle to the $G$-embedding $\iota$ is isomorphic to $di_x(F_x)^\perp \oplus di_x(F_x)^\perp$.
This gives the first item. For $\varepsilon >0$, denote by $A_\varepsilon$ the set of points $(x,v,w) \in A$ with $\|w\|<\varepsilon$. The map $A \to M\times TE$ given by $(x,v,w)\mapsto (x,v,i(x)+w)$ then clearly identifies, for $\varepsilon >0$ small enough,  $A_\varepsilon$ with  an open transversal to $\tilde{\maF} = \maF \times \{0\}$ in $M\times TE$. The transversal condition is given in $(x,\xi,0)$ by $F_x + T_{(x,\xi,0)}A=F_x + (T_xM) \oplus T_xE \oplus di_x(F_x)^\perp = T_{(x,\xi,0)}(M\times TE)$.} 
The third and fourth items  are eventually  easily deduced by standard arguments that we already explained in the previous section, see  \cite{atiyah1974elliptic} and \cite{ConnesSkandalis}.  
\end{proof}

{We point out that, exactly as in the case of smooth foliations, the topological transversal $\maA_G$  to the foliated space $(M\times T_G(E), \maF\times 0)$, obtained in the fourth item of  Proposition \ref{Factorisation},  gives rise to a well defined quasi-trivial Morita extension class $\epsilon \in \k\k (C_0(\maA_G), C^*(M\times T_G(E), \maF\times 0))$ and hence the $R(G)$-morphism
$$
\epsilon\;  : \; \k^j_\mathrm{G}(\maA_G) \longrightarrow  \k^\mathrm{G}_j \left(C_0(T_G(E), C^*(M, \maF))\right).
$$
The class $\epsilon$ can be described as follows. By using that the normal bundle to $\maA$ in $M\times T(E)$ is isomorphic to the vector bundle $F\times 0$ that we restrict to $\maA$, we may  consider an open tubular neighborhood $\maN$ of $\maA$ in $M\times T(E)$ which is a disc-bundle over $\maA$ whose fibers are small disc-placques which correspond to the restricted foliation $\maF\times 0$ to $\maN$. It is then clear by construction that $\maN_G:=\maN\cap (M\times T_G(E))$ is also a disc-fibration by the same placques but now over the space $\maA_G$, so the base is no more a smooth manifold. The $C^*$-algebra $C^*(\maN_G, \maF^{\maN_G})$ of the lamination $\maF^{\maN_G}$ of the open subspace $\maN_G$ which is the restriction of the foliation $\maF\times 0$ of $M\times T_G(E)$, is then Morita equivalent to $C_0(\maA_G)$. Hence using the trivial extension map 
$$
\k_j^\mathrm{G} (C^*(\maN_G, \maF^{\maN_G})) \longrightarrow \k_j^\mathrm{G} (C^*(M\times T_G(E), \maF\times 0)) \simeq \k^\mathrm{G}_j \left(C_0(T_G(E), C^*(M, \maF))\right),
$$
corresponding to the open subspace $\maN_G$ in the space $M\times T_G(E)$, 
we finally obtain the allowed quasi-trivial $G$-equivariant extension map  $\epsilon$.}

Following Kasparov, we define  a Dirac element $[D_E]\in \k\k (C_0(T_G(E))\rtimes G, \C)$ which, according to the main result of \cite{Kasparov:KKindex}, computes the index of $G$-invariant $G$-transversally elliptic operators on the orthogonal $G$-representation $E$, through the descent morphism $j^G$. There are though some technical details which are passed over here and which would need to be expanded elsewhere. One especially needs to  replace $C_0(T_GE)$ by a better (although non-separable) symbol $C^*$-algebra denoted by $\mathfrak{S}_G(E)$  in \cite{Kasparov:KKindex}, and therefore one needs as well  to use the extended version of Kasparov's $\k\k$-theory, adapted to non-separable algebras. All these details with their generalizations to foliations will be dealt with in a forthcoming paper. 

We only mention here  that since $C^*(M, \maF)$ is endowed with the trivial $G$-action,  we have a well defined morphism
$$
\k_j^\mathrm{G} (C_0(T_G(E))\otimes C^*(M, \maF)) \stackrel{j^G}{\longrightarrow} \k\k^j (C^*G, [C_0(T_G(E))\rtimes G] \otimes C^*(M, \maF))  \stackrel{\otimes [D_E]}{\longrightarrow} \k\k^j (C^*G, C^*(M, \maF)),
$$
that we denote by $\partial_E\otimes C^*(M, \maF)$. Roughly speaking and using the main result of \cite{Kasparov:KKindex}, the map $\partial_E\otimes C^*(M, \maF)$ is the expected index map for $G$-invariant $G$-transversally elliptic operators on $E$ with coefficients in the $G$-trivial $C^*$-algebra $C^*(M, \maF)$. 
%

\begin{remarque}
The  composite morphism $\Ind^{\maF, top}$ :
$$
\Ind^{\maF, top}\; : \; \k^j_\mathrm{G}(F_G)  \stackrel{\iota_!} {\longrightarrow}\k^j_\mathrm{G}(\maA_G) \stackrel{\epsilon} {\longrightarrow}\k_j^\mathrm{G} (C_0(T_G(E))\otimes C^*(M, \maF)) \stackrel{\partial_E\otimes C^*(M, \maF)}{\longrightarrow} \k\k^j (C^*G, C^*(M, \maF))
$$
is independent of the choice  of euclidean $G$-representation $E$ with the isometric $G$-embedding $i$.
\end{remarque}

{\begin{definition}
The morphism  
$$
\Ind^{\maF, top}\; : \; \k^j_\mathrm{G}(F_G) \longrightarrow \k\k^j (C^*G, C^*(M, \maF)),
$$
will be called the topological index morphism for $G$-invariant leafwise $G$-transversally elliptic operators. 
\end{definition}}

%
%
If $G$ is the trivial group then the topological index morphism  reduces to the topological index morphism for leafwise elliptic operators as defined in \cite{ConnesSkandalis}, and it then coincides with the analytic index morphism $\Ind^\maF$, this is precisely  the  Connes-Skandalis index theorem. For general $G$ and when the foliation is top dimensional,  the naturality of the index distribution proved in \cite{atiyah1974elliptic} together with the Kasparov index theorem proved in \cite{Kasparov:KKindex} implies again the equality of the topological index morphism with the analytical one. 
 
\begin{remarque}
When $G$ and $M$ are no more compact, but the $G$-action is supposed to be proper and cocompact as in \cite{Kasparov:KKindex}, then the proofs given here allow to still define the index morphism 
$$
\Ind^{M, \maF} : \k^j_\mathrm{G} (F_G) \longrightarrow \k\k^j (C_0(M)\rtimes G, C^*(M, \maF)).
$$
\end{remarque}

We finally point out that  most of the constructions given in the present paper apply, with minor changes, to the more general category of foliated spaces ({e.g.} laminations) as studied in \cite{MooreSchochet} using sections and operators which are leafwise smooth and transversally continuous. However, the construction of the topological index for instance is not clear in general   since the $G$-embedding in $E$ is not ensured a priori.

\appendix

\section{Unbounded version of the index class}
We  define in this appendix the index class for operators of  order $1$, using the unbounded version of Kasparov's theory  \cite{baaj1983theorie}. The unbounded version simplifies the computation of some Kasparov products and was used in the present paper. 
In order, to achieve this construction, we shall need the following independent theorem which generalizes results from \cite{vassout2001feuilletages}, see also \cite{vassout2006unbounded}. {We assume as in this whole paper that $G$ acts on $(M, \maF)$ by holonomy diffeomorphisms. This is true for instance when $G$ is connected.} Most of this appendix will be devoted to the proof of the following

\begin{thm}\label{thm:regular}
Let $P$ be a $G$-invariant leafwise $G$-transversally elliptic operator. Then the closure $\overline{P}$ of $P$ is a regular operator. 
{Moreover, the adjoint operator $P^*$ of $P$ coincides with the closure of the leafwise formal adjoint $P^\natural$ of $P$, i.e. $P^*=\overline{P^\natural}$.}
\end{thm}

If $P$  is a leafwise  pseudodifferential operator of order $m>0$ on $M$ acting between the hermitian bundles $E$ and $E'$, then we have 
$$
\langle P\eta,\eta'\rangle = \langle \eta ,P^\natural \eta'\rangle,\qquad \eta \in C^\infty_c(\mathcal{G}, r^*E)\text{ and } \eta' \in C^\infty_c(\mathcal{G}, r^*E'). 
$$
The Hilbert module completions of $C^\infty_c(\mathcal{G}, r^*E)$ and $C^\infty_c(\mathcal{G}, r^*E')$ are respectively denoted $\maE$ and $\maE'$. 
The operator $P$ is densely defined with domain $C^\infty_c(\mathcal{G}, r^*E)$ and has a well defined closure $\overline{P}$. The same observation holds for the leafwise pseudodifferential operator $P^\natural$.
Then 
$\overline{P^\natural}\subset P^*$. 

We shall use notations and discussions from \cite{vassout2001feuilletages,vassout2006unbounded} and especially the following
\begin{lem}\cite{vassout2006unbounded}\label{lem:20:vassout}
Let $A$, $B$ be compactly supported {leafwise pseudodifferential} operators such that $\mathrm{ord}\ A + \mathrm{ord}\ B\leq 0$ and $\mathrm{ord}\ B \leq  0$.
Then we have $\overline{AB} = \overline{A} \; \overline{B}$, an equality of adjointable operators.
\end{lem} 
{The following proposition from \cite{vassout2006unbounded} will also be needed:
\begin{prop}\cite{vassout2006unbounded}\label{prop:21:vassout} Let $P$ be an elliptic, compactly supported {leafwise pseudodifferential operator}. Then the operator $\overline{P}$ is a regular operator.
\end{prop}
}

{Recall that if $\maE$ is a Hilbert $G$-module over a $G$-trivial $C^*$-algebra $A$, and   $ \pi : C^*G \rightarrow \maL_A(\maE)$ is the induced representation, then for any $\alpha \in \hat{G}$ the $\alpha$-isotypical component $\maE_\alpha$ of $\maE$ is the Hilbert $A$-submodule which is the image of the projection $p_\alpha=\pi ((\dim \alpha)\,  \chi_\alpha)$. Here $\chi_\alpha$ is  the character of $\alpha$.  
}

\begin{prop}\label{prop:composante:isotypique}
{Let $A$ be a $G$-trivial $C^*$-algebra and let $\maE, \maE'$ be  Hilbert $G$-modules on $A$. Denote as above by $\maE_\alpha$ and $\maE'_\alpha$ the $\alpha$-isotypical component of $\maE$ and $\maE'$ respectively, and by $p_\alpha$ the corresponding projections. We also denote by $i_\alpha$ the inclusion of $\maE_\alpha$ in $\maE$ as well as the inclusion of $\maE'_\alpha$ in $\maE'$. Then 
\begin{enumerate}
\item $\maE = \bigoplus \limits_{\alpha \in \hat{G}} \maE_\alpha$ as a countable Hilbert $A$-module decomposition and the same holds for $\maE'$. In particular,  any $G$-invariant adjointable operator from $\maE$  to $\maE'$ is diagonal with respect to  this decomposition, i.e. $\maL_A(\maE, \maE')^G=\bigoplus \limits_{\alpha \in \hat{G}} \maL_A(\maE_\alpha, \maE'_\alpha)^G$.
\item {Let $P=(P, \operatorname{dom} (P))$ be a closable $G$-invariant operator from $\maE$ to $\maE'$ such that $(i_\alpha p_\alpha)\operatorname{dom} (P) \subset \operatorname{dom} (P)$, then $P_\alpha=(P\vert_{\maE_\alpha}, p_\alpha \operatorname{dom} (P))$ is closable and $\overline{P_\alpha}=\left(\overline{P}\right)_\alpha$, for any $\alpha \in \hat{G}$.}
\end{enumerate}}
\end{prop}

\begin{remarque}
Hence, the $\alpha$-component of any closed $G$-invariant operator is closed. Indeed, in this case the inclusion $(i_\alpha p_\alpha)\operatorname{dom} (P) \subset \operatorname{dom} (P)$ is automatically fulfilled.
\end{remarque}

\begin{proof}
1. Notice first that $G$ is a compact Lie group therefore $\hat{G}$ is countable.
Our hypothesis that the $G$-action on $A$ is trivial implies that $G$ acts by {adjointable unitaries $(U_g)_{g\in G}$} on $\maE$. This action extends to an adjointable representation $\pi$ of $C^*G$ given for any $\varphi \in C(G)$ {by $ \pi(\varphi)\eta=\int_G \varphi(g) U_g\eta\, dg.$ }
Now recall that $[(\dim \alpha)\, \chi_\alpha]_{\alpha\in \hat{G}}$  is a family of projections  in $C^*G$ such that $((\dim \alpha)\, \chi_\alpha) ((\dim \beta)\, \chi_\beta) = 0$ whenever $\alpha\neq \beta$ and $\Id_{C^*G} = \sum \limits_{\alpha \in \hat{G}} (\dim \alpha)\, \chi_\alpha$ as a multiplier of $C^*G$. Therefore, the family of (adjointable) projections $(p_\alpha)_{\alpha\in \hat{G}}$ defined by $p_\alpha =  \pi((\dim \alpha)\, \chi_\alpha)$  satisfies the same properties, in particular one has $
\Id_\maE = \sum \limits_{\alpha \in \hat{G}} p_\alpha \text{ and } p_\alpha p_\beta = 0 \text{ if }\alpha\neq \beta.$  Notice that under our assumptions, $\pi (C^*G)\maE$ is automatically dense in $\maE$ as can be seen by using an approximate identity in $C^*G$ composed of continuous non-negative functions and the precise expression of $\pi$.
%
Since a $G$-invariant operator $T\in \maL_A({\maE, \maE'})^G$ commutes with the representations both denoted  $\pi$ in $\maE$ and $\maE'$, it commutes with each $p_\alpha$ and hence the first item is proved. 

2. We have {$P_\alpha = {p_\alpha P i_\alpha}$ where again we have denoted by the same letters $p_\alpha$ and $i_\alpha$  the operators for $\maE$ and $\maE'$. In particular $P_\alpha p_\alpha=p_\alpha P$  on $\operatorname{dom} (P)$ and $i_\alpha P_\alpha = P i_\alpha$  on $\operatorname{dom} (P_\alpha)=p_\alpha \operatorname{dom} (P)$. 
The graph of $P_\alpha$ is the image under $p_\alpha\times p_\alpha$ of the graph of $P$. Since $p_\alpha$ and $i_\alpha$ are continuous, the conclusion follows by using that $\operatorname{dom} (P)$ is preseved by $i_\alpha p_\alpha$.} \\
\end{proof}

It is worthpointing out that under the assumptions of Proposition \ref{prop:composante:isotypique}, if the closure $\overline{P}$ of $P$ is regular then for any $\alpha\in \widehat{G}$, $\overline{P}_\alpha=\overline{P_\alpha}$ is regular.  The converse is also true and is used below to deduce our theorem \ref{thm:regular}. Indeed,  notice the following general observation whose proof is a direct inspection of the graphs:

\begin{lem}\label{StandardLemma}
Let $(\maE_k)_{k\in \N}$ and $(\maE'_k)_{k\in \N}$ be two sequences of Hilbert $A$-modules and let $\maE=\oplus_{k\geq 0} \maE_k$ and $\maE'=\oplus_{k\geq 0} \maE'_k$ be the Hilbert direct sums. Let $(T_k: \operatorname{dom} (T_k)\subset \maE_k\to \maE'_k)_{k\geq 0}$ be a sequence of   regular operators. Let $T: \operatorname{dom} (T)\subset \maE  \to \maE'$ be the direct sum operator given by
$$
\operatorname{dom} (T) :=\{x=(x_k)_{k\geq 0} \in \oplus_k \operatorname{dom} (T_k)\; \vert\;   (T_k (x_k))_{k\geq 0}\in \maE'\}
\text{ and }
 T((x_k)_{k\geq 0}) = (T_k (x_k))_{k\geq 0}.
$$
Then $(T, \operatorname{dom} (T))$ is a  regular  operator. 
\end{lem}
\noindent

\medskip 

We are now in position to prove our Theorem  \ref{thm:regular}.

\begin{proof}[Proof of Theorem \ref{thm:regular}]
Let $\Delta_G$ be the Laplace Casimir operator along the orbits introduced in Subsection \ref{section:moment:map}, i.e. $\Delta_G = \sum \mathscr{L}(V_k)^2$ for an orthonormal basis  $(V_k)_k$ of $\mathfrak{g}$, see again Subsection \ref{section:moment:map}. We use again the same notation for the Casimir operators on both $E$ and $E'$. Set $B:=P^\natural P+\Delta_G^m$ and $ C=PP^\natural +\Delta_G^m$, then $B$ and $C$ are  $G$-invariant leafwise elliptic operators of order $2m$. Indeed, {the principal symbol of $B$ is a pointwise non-negative linear map and we have for any $(x, \xi)\in F$:
$$
 \langle \sigma (B) (x, \xi)u, u\rangle =  \vert \sigma (P) (x, \xi) u \vert ^2 + q_x(\xi)^m \vert u\vert^2.
$$
Since $q_x(\xi) =0$ only happens if $\xi\in (F_G)_x$, and since $\sigma (P) (x, \xi)$ is invertible for any $\xi\in  (F_G)_x\smallsetminus 0$ by the $G$-transverse ellipticity of $P$, we deduce that $\sigma (B) (x, \xi)$ is invertible for any  $\xi\in F_x\smallsetminus 0$. The argument for $C$ is completely similar.}
Denote then by $Q$ a $G$-invariant leafwise parametrix for $B$ and similarly 
 by $\tilde{Q}$ a $G$-invariant leafwise parametrix for $C$. Set  
$$
R= \id - BQ, \; S=\id - QB, \; \tilde{R}= \id - C\tilde{Q} \;\text{ and }\;   \tilde{S}=\id - \tilde{Q}C.
$$
 {Since $S$ and $Q$ are negative order operators they extend to adjointable operators. In particular, for any $\lambda\in \R$, $\overline{S} +\lambda \overline{Q} = \overline{S} +\lambda \overline{Q} =\overline{S +\lambda Q}$. We have $\operatorname{ord}(P) + \operatorname{ord}(QP^\natural)=0$ and $\operatorname{ord}(P) + \operatorname{ord}(S+\lambda Q) \leq 0$ therefore by Lemma \ref{lem:20:vassout} we obtain 
$$
\overline{P} \; \overline{QP^\natural}= \overline{P Q P^\natural}\text{ and }\overline{P} (\overline{S}+\lambda \overline{Q})=\overline{P(S+\lambda Q)}.
$$ 
Therefore, $\operatorname{im}(\overline{Q P^\natural}) + \operatorname{im}(\overline{S} +\lambda \overline{Q}) \subset \operatorname{dom}(\overline{P}).$}
{Denote by $\widehat{\Delta}_G$ the Laplacian on $G$, we have that $\Delta_G^m \pi(\varphi)= \pi(\widehat{\Delta}_G^m \varphi)$ for any $\varphi \in C^\infty(G)$. But  $\widehat{\Delta}_G\chi_\alpha = \lambda_\alpha \chi_\alpha$ with $\lambda_\alpha \geq 0$ only depending on $\alpha$ and positive if $\alpha$ is different from the trivial representation. Therefore}
$$
{(\Delta_G^m)_\alpha=\Delta_G^m p_\alpha i_\alpha= \pi\left((\dim\alpha) \widehat{\Delta}_G^m \chi_\alpha\right) i_\alpha=\lambda_\alpha^m \Id_{\maE_\alpha}.}
$$
On the other hand composition with $p_\alpha$ on the left and with $i_\alpha$ on the right in the first parametrix relation yields using Proposition \ref{prop:composante:isotypique} to the following relation on $p_\alpha C_c^\infty (\maG, r^*E)$:
$$
(Q  P^\natural)_\alpha P_\alpha = p_\alpha - \lambda_\alpha^m Q_\alpha - S_\alpha.
$$
This allows to  prove the inclusion  $\operatorname{dom}(\overline{P_\alpha})\subset \operatorname{im}(\overline{Q P^\natural}_\alpha) + \operatorname{im}(\overline{S_\alpha} +\lambda_\alpha^m \overline{Q_\alpha})$. Recall that we already proved the opposite inclusion (before reducing to the $\alpha$-isotypical component), in particular, we already proved that
$$
\operatorname{im}(\overline{Q P^\natural}_\alpha) + \operatorname{im}(\overline{S_\alpha} +\lambda_\alpha^m \overline{Q_\alpha}) \subset \operatorname{dom}(\overline{P_\alpha}).
$$
Let  then $x\in \operatorname{dom}(\overline{P}_\alpha)$ and $x_n \in \operatorname{dom}(P_\alpha)=p_\alpha C^\infty_c(\maG, r^*E)$ such that $x_n \to x$ and $P_\alpha x_n \to \overline{P}_\alpha x$. We can then write
$$x_n=(S_\alpha +\lambda_\alpha^m Q_\alpha)x_n + (Q P^\natural)_\alpha P_\alpha x_n.$$
Since $(\overline{S} +\lambda_\alpha^m \overline{Q})$ and $\overline{Q P^\natural}$ are adjointable, we obtain that  $(\overline{S} +\lambda_\alpha^m \overline{Q})_\alpha$ and $\overline{Q P^\natural}_\alpha$ are adjointable. It follows by continuity that
$$
x=(\overline{S} +\lambda_\alpha^m \overline{Q})_\alpha x + \overline{(Q P^\natural)}_\alpha\overline{P}_\alpha x.
$$
Therefore  $\operatorname{dom}(\overline{P}_\alpha) \subset \operatorname{im}(\overline{Q P^\natural}_\alpha) + \operatorname{im}(\overline{S}_\alpha +\lambda_\alpha^m \overline{Q}_\alpha) $. We thus obtained the equality 
\begin{equation}\label{DomainEquality}
\operatorname{dom}(\overline{P}_\alpha)=\operatorname{im}(\overline{Q P^\natural}_\alpha) + \operatorname{im}(\overline{S}_\alpha +\lambda_\alpha^m \overline{Q}_\alpha).
\end{equation}
We similarly have by the same method the inclusion $
\operatorname{im}(\overline{\tilde{Q}^\natural P}) + \operatorname{im}(\overline{\tilde{R}^\natural} +\lambda_\alpha^m \overline{\tilde{Q}^\natural}) \subset \operatorname{dom}(\overline{P^\natural}).$
From the previous assertions, we can now deduce that $P_\alpha^* \subset \overline{P^\natural}_\alpha$. 
Recall that $(P^\natural\tilde{Q})_\alpha^*  P_\alpha^* \subset (P_\alpha (P^\natural\tilde{Q})_\alpha)^*$ and let $x \in \operatorname{dom}(P_\alpha^*)$ then 
$$
x= (P_\alpha (P^\natural\tilde{Q})_\alpha)^* x + (\lambda_\alpha^m \tilde{Q} + \tilde{R})_\alpha^* x = (P^\natural\tilde{Q})_\alpha^*  P_\alpha^* x + (\lambda_\alpha^m \tilde{Q}^* + \tilde{R}^*)_\alpha x.
$$
It follows that $x \in \operatorname{im}( (P^\natural\tilde{Q})_\alpha^*)  + \operatorname{im}( (\lambda_\alpha^m \tilde{Q}^* + \tilde{R}^*)_\alpha) $. Since $P^\natural \tilde{Q}$ and $\lambda_\alpha^m \tilde{Q} + \tilde{R}$ are negative order operators, we have $\overline{(P^\natural \tilde{Q})^\natural} = (P^\natural \tilde{Q})^*$ and $\overline{(\lambda_\alpha^m \tilde{Q} + \tilde{R})^\natural} =\lambda_\alpha^m \overline{\tilde{Q}^\natural} + \overline{\tilde{R}^\natural} =(\lambda_\alpha^m \tilde{Q} + \tilde{R})^*=\lambda_\alpha^m \tilde{Q}^* + \tilde{R}^*$. This gives that
$$x \in \operatorname{im}(\overline{(P^\natural\tilde{Q})^\natural}_\alpha)  + \operatorname{im}( \overline{(\lambda_\alpha^m \tilde{Q}^\natural + \tilde{R}^\natural)}_\alpha)= \operatorname{im}(\overline{\tilde{Q}^\natural P}_\alpha)  + \operatorname{im}( \overline{(\lambda_\alpha^m \tilde{Q}^\natural + \tilde{R}^\natural)}_\alpha) .$$
But we proved that $
\operatorname{im}(\overline{\tilde{Q}^\natural P}) + \operatorname{im}(\overline{\tilde{R}^\natural} +\lambda_\alpha^m \overline{\tilde{Q}^\natural}) \subset \operatorname{dom}(\overline{P^\natural}),$ hence $x\in \operatorname{dom}(\overline{P^\natural}_\alpha)$ as allowed.

It remains to prove that  $\overline{P_\alpha}$ is regular. The equality \eqref{DomainEquality} shows that the  graph of  $ \overline{P_\alpha}$ is  given by 
$$
G(\overline{P_\alpha})=\{(\overline{Q_\alpha P^\natural}_\alpha x+(\overline{S}_\alpha +\lambda_\alpha^m \overline{Q}_\alpha) y,\overline{ P_\alpha  Q_\alpha P^\natural_\alpha}x+\overline{P_\alpha (S_\alpha +\lambda_\alpha^m Q_\alpha) }y),\ (x,y)\in \mathcal{E}'_\alpha\times \mathcal{E}_\alpha\}.
$$
Hence this graph coincides with the range of the adjointable operator from $\maE'\oplus \maE$ to $\maE\oplus \maE'$ given by 
$$
U_\alpha = \begin{pmatrix}  \overline{Q P^\natural}_\alpha &  \overline{S}_\alpha +\lambda_\alpha^m \overline{Q}_\alpha\\
&\\
\overline{P_\alpha Q_\alpha P^\natural_\alpha} & \overline{P_\alpha(S_\alpha +\lambda_\alpha^m Q_\alpha)} 
\end{pmatrix}.
$$
{Hence $G(\overline{P_\alpha})=\mathrm{im}(U_\alpha)$ is orthocomplemented  in $\maE\oplus \maE'$ as the range of an adjointable operator (with closed range).}
It follows that $\overline{P}$ is regular by Lemma  \ref{StandardLemma}.
 Since $P^\natural$ is also $G$-transversally elliptic, $\overline{P^\natural}$ is also regular. 
{Now we have seen above  that $\overline{P^\natural}_\alpha= P_\alpha^*$ for any $\alpha\in \hat{G}$, and this implies that $\overline{P^\natural}=P^*$.
}

\end{proof}

In the sequel we will denote simply by $P$ the regular operator obtained from a formally selfadjoint $G$-invariant leafwise  $G$-transversally elliptic operator.

\begin{defi}[Unbounded Kasparov module $\cite{baaj1983theorie}$]
Let $A$ and $B$ be $C^*$-algebras. An $(A,B)$-unbounded Kasparov cycle $(E,\phi,D)$ is a triple where $E$ is a Hilbert $B$-module, $\phi : A \rightarrow \mathcal{L}(E)$ is a graded $\star$-homomorphism and $(D, \rm{dom} (D))$ is an unbounded regular seladjoint operator such that:
\begin{enumerate}
\item $(1+D^2)^{-1}\phi(a)\in k(E)$, $\forall a \in A$,
\item  The subspace of $A$ composed of the elements $a\in A$ such that $\phi (a) (\rm{dom}(D))\subset \rm{dom}(D)$ and $[D,\phi(a)]=D\phi (a) - \phi(a) D$ is densely defined and extends to an adjointable operator on $E$,  is dense in $A$.
\end{enumerate}
When $E$ is $\Z_2$-graded with $D$ odd and $\pi (a)$ even for any $a$, we say that the Kasparov cycle is even. Otherwise, it is odd. 
\end{defi}

In \cite{baaj1983theorie}, appropriate equivalence relations are introduced on such (even/odd) unbounded Kasparov cycles, which allowed to recover the groups $\k\k^* (A, B)$. When the compact group $G$ acts on all the above  data, one recovers similarly $\k\k^*_{\mathrm{G}}(A, B)$ by using the equivariant version of the  Baaj-Julg unbounded cycles of the previous definition. 
We are now in position to state the second main  result of this appendix. 
\medskip

\begin{thm}
Let $P_0:C^{\infty}_c(\mathcal{G},r^*E^+)\rightarrow C^{\infty}_c(\mathcal{G},r^*E^-)$ be a $G$-invariant leafwise $G$-transversally elliptic pseudodifferential operators of order $1$, and let $P$ be the associated regular self-adjoint operator defined by $\begin{pmatrix}
0&P_0^*\\P_0&0
\end{pmatrix}$. Then the triple $(\mathcal{E}, \pi ,P)$ is an even $(C^*G,C^*(M,\mathcal{F}))$-unbounded Kasparov cycle, which defines a class in  {$\k\k(C^*G,C^*(M,\mathcal{F}))$}. The similar statement holds in the ungraded case giving a class in  {$\k\k^1(C^*G,C^*(M,\mathcal{F}))$}.
\end{thm}

\medskip

\begin{proof}
For any $\varphi\in C (G)$, it is easy to see that $\pi(\varphi)$ preserves the domain of $P$ and by Remark \ref{lem:Ppi=piP}, we have $[\pi(\varphi),P]=0$. 
It remains to check that $(1+P^2)^{-1}\circ\pi(\varphi) \in \mathcal{K}(\mathcal{E})$. We may take for our operator $\Delta_G$  the Casimir operator, which is a leafwise differential operator of order $2$. We already noticed that the operator $P^2+\Delta_G$ is elliptic. We hence deduce that the resolvent $(1+P^2+\Delta_G)^{-1}$ is a  compact operator in $\mathcal{E}$. 

We now show that for any $\varphi\in C^\infty (G)$, the operator 
$$
(1+P^2+\Delta_G)^{-1}\circ \pi(\varphi)-(1+P^2)^{-1}\circ \pi (\varphi ), 
$$ 
is compact, which will ensure that $(1+P^2)\circ \pi(\varphi )$ is also compact. 
Denote by $\widehat{\Delta}_G$ the Laplacian on $G$ viewed as a riemannian $G$-manifold. Using that $\Delta_G\pi(\varphi)=\pi(\widehat{\Delta}_G\phi)$ and that $[\pi(\varphi) ,P]=0$, we have: 
\begin{multline*}
(1+P^2+\Delta_G)^{-1}\circ \pi(\varphi)-(1+P^2)^{-1}\circ \pi (\varphi )=-(1+P^2+\Delta_G)^{-1}\Delta_G(1+P^2)^{-1} \circ\pi(\varphi )\\
~~~~~~=-(1+P^2+\Delta_G)^{-1}\pi(\widehat{\Delta}_G\varphi)(1+P^2)^{-1}.
\end{multline*}
Now since $\pi(\widehat{\Delta}_G\varphi)$ and $(1+P^2)^{-1}$ are adjointable operators and since $(1+P^2+\Delta_G)^{-1}$ is compact, we deduce that $(1+P^2+\Delta_G)^{-1}\circ \pi(\varphi)-(1+P^2)^{-1}\circ \pi (\varphi )$ is a compact operator on ${\mathcal E}$.
\end{proof}
%

\begin{defi}\label{def:indice:NB}
The index class ${\ind (P_0)}$ of a $G$-invariant leafwise  $G$-transversally elliptic pseudodifferential operator of positive order $m$,  is the class in {$\k\k^*(C^*G,C^*(M,\mathcal{F}))$} of the $(C^*G,C^*(M,\mathcal{F}))$-unbounded Kasparov cycle $(\mathcal{E},\pi, P)$.
\end{defi}

The relation with the bounded version is obtained by using the Woronowicz transform, see \cite{baaj1983theorie}. More precisely, if $P_0$ is a $G$-invariant leafwise  $G$-transversally elliptic pseudodifferential operator of  order $1$ for instance, then the  triple $(\mathcal{E},\pi, P(1+P^2)^{-1/2})$
is a (bounded) Kasparov cycle.

%
%
The following proposition was kindly suggested to us by the referee.  It allows to avoid using the regularity of the unbounded operator $P$ and can be exploited by the interested reader to simplify many of the statements of the present paper. 
\begin{prop}
Let $P$ be a formally selfadjoint $G$-invariant leafwise  $G$-transversally elliptic operator of order $1$. Let $\widetilde{D_G}$ be the order $1$ pseudodifferential operator $d_G^*(1+\Delta)^{-1/2}d_G$. Then
\begin{enumerate}
\item $T=\begin{pmatrix}
P& \widetilde{D_G}\\
\widetilde{D_G}& -P
\end{pmatrix}$ is elliptic and therefore its closure  is regular;
\item the triple $(\maE \oplus \maE, \pi \oplus 0 , \overline{T})$ defines a $(C^*G,C^*(M,\maF))$-unbounded Kasparov cycle. In the even case with grading $\gamma$ of $\maE$, we take the  grading  $\gamma \oplus -\gamma$.
\item The Kasparov cycle $(\maE \oplus \maE,\pi\oplus 0, \frac{T}{\sqrt{1+T^2}})$  represents the index class $\ind (P)$ in $\k\k (C^*G, C^*(M, \maF))$.
\end{enumerate} 
\end{prop}
\begin{proof}
$1.$ We have
 $$\sigma_1(T)^2=\begin{pmatrix}
\sigma_1(P)^2 + \sigma_1(\widetilde{D_G})^2 & 0\\
0&\sigma_1(P)^2 + \sigma_1(\widetilde{D_G})^2
\end{pmatrix}\ \mbox{ which is invertible on}\  S^*\maF.$$ 
Indeed, $\sigma_1(T)^2$ is a sum of non-negative endomorphisms and 
$\sigma_1(P)^2$ is invertible on $F_G$ and $\sigma_1(\widetilde{D_G})^2$ is invertible along the orbits.
It follows from Proposition \ref{prop:21:vassout} that $\overline{T}$ is regular.\\
$2.$ Since $T$ is elliptic, we get that $(\id + \overline{T}^2)^{-1}$ is compact. If $\varphi \in C^\infty(G)$ then $\widetilde{D_G} \pi(\varphi)=d_G^*(\id+\Delta)^{-1/2}\pi(d\varphi)$ extends in an adjointable operator since $d_G^*(\id + \Delta)^{-1/2}$ has $0$ order. Moreover, $(\pi(\varphi) \oplus 0) (\rm{dom} \overline{T}) \subset \rm{dom} \overline{T}$ by $G$-invariance of $P$ and $\widetilde{D_G}$, therefore $\left[\overline{T},  \pi(\varphi) \oplus 0\right]$ is adjointable, and by density this is still true over $C^*G$. The operator $T$ is of course odd for the grading since $P$ is odd and $D_G$ is even.


$3.$ Any homotopy of symbols yields an operator homotopy by the standard argument.
The cycle $(\maE ,\pi, \frac{P}{\sqrt{\id + P^2+D_G^2}})$ is equivalent to $(\maE \oplus \maE,\pi\oplus 0, \frac{T}{\sqrt{1+T^2}})$ because 
$(\maE \oplus \maE, \pi\oplus 0, \frac{T_t}{\sqrt{1+T_t^2}})$ where $T_t=\begin{pmatrix}
P&tD_G\\
tD_G & -P
\end{pmatrix}$ is a homotopy with $(\maE \oplus \maE,\pi\oplus 0, \begin{pmatrix}
P&0\\0&-P
\end{pmatrix}(1+P^2+D_G^2)^{-1/2} )$ which is a Kasparov cycle because $D_G^2\pi(\varphi)$ is adjointable. Indeed, we can assume $\varphi$ non-negative in $C^*G$. Then 
$$
D_G \pi(\varphi^*\ast \varphi)=\pi(d\varphi)^*(\id +\Delta)^{-1/2} d_Gd_G^*(\id +\Delta)^{-1/2}\pi(d\varphi), \forall \varphi \in C^\infty(G).
$$  
Furthermore, $(\maE \oplus \maE,\pi\oplus 0, \begin{pmatrix}
P&0\\0&-P
\end{pmatrix} (1+T^2)^{-1/2} )=(\maE ,\pi, \frac{P}{\sqrt{\id + P^2+D_G^2}})+(\maE ,0, -\frac{P}{\sqrt{\id + P^2+D_G^2}})$, where the last cycle is degenerate.
It eventually follows that $(\maE, \pi , \frac{P}{\sqrt{\id + P^2+D_G^2}})$ is homotopy equivalent to $(\maE,\pi, \frac{P}{\sqrt{1+P^2}})$.  
\end{proof}
}

%
%
%
%

\section{The technical proposition in the non-compact case}\label{ShubinLemma}

Let us fixe a non compact foliated manifold $(U,\maF^U)$ and denote by $\overline{\Psi^0(U,\maF^U,E)}$ the $C^*$-subalgebra of the adjointable operators $\maL_{C^*(U, \maF^U)} (\maE)$, which is generated  by the closures of the zero-th order pseudodifferential operator with symbols in $C_c^\infty(S^*\maF^U,\End(E))$, see \cite{ConnesSkandalis,NWX,vassout2006unbounded}. Here $\maE$ is the Hilbert $C^*(U, \maF^U)$-module defined before for the foliation $(U, \maF^U)$. Then
setting
\begin{multline*}
\maK_U(\maE):=\{T\in \maL_{C^*(U, \maF^U)} (\maE)\ |\ Tf\ \&\ fT\in \maK_{C^*(U, \maF^U)} (\maE) ,\ \forall f\in C_0(U)\}\text{ and }\\
\overline{\Psi^0(U,\maF^U,E)}_U:=\{T\in \maL_{C^*(U, \maF^U)} (\maE)\ |\ Tf\ \&\ fT\in \overline{\Psi^0(U,\maF^U,E)} ,\ \forall f\in C_0(U)\},
\end{multline*}
the following exact sequence holds (see  \cite[Proposition 4.6]{ConnesSkandalis}): 
\begin{equation}\label{Prop:4.6:CS}
0 \to \xymatrix{\maK_U(\maE) \ar[r]& \overline{\Psi^0(U,\maF^U,E)}_U\ar[r]& C_b(S^*\maF^U,\End(E))}\to 0.
\end{equation}

\begin{prop}\label{prop:ineq:non:compact}
{Let $(U,\maF^U)$ be a (non compact) foliated manifold. Let $A \in \overline{\Psi^0(U,\maF^U,E)}_U$ be selfadjoint. Suppose that the principal symbol $ {\sigma(A)}$ of $A$ satisfies
\begin{equation}\label{equ:prop:ineg:non:compact:symbol}
\forall \varepsilon >0, \exists c>0\text{ such that }\forall (x, \xi)\in S^*\maF^U:\; 
\| {\sigma(A)}(x,\xi)\|\leq c~ {\sigma(Q)}(x,\xi)+\varepsilon.
\end{equation}
Then $\forall \varepsilon >0$, there exist  two selfadjoint operators $R_1$ and $R_2 \in \maK_{{U}} (\maE)$ such that:
$$
-(c\, Q+\varepsilon +R_1)\leq A\leq c\, Q+\varepsilon +R_2 \text{  as self-adjoint operators on }\maE.
$$}
\end{prop}

\begin{proof}
{By  \eqref{equ:prop:ineg:non:compact:symbol}, we deduce that 
$0\leq  {\sigma(A)}+ c {\sigma(Q)}+ \varepsilon$  in the $C^*$-algebra $C_b(S^*\maF^U,\End(E))$}. Now using the exact sequence \eqref{Prop:4.6:CS}, we get that there is {an operator} $R_1 \in {\maK_U(\maE)}$ such that 
$0 \leq A+cQ +\varepsilon +R_1$, in other words we get 
$$ -(R_1+cQ +\varepsilon) \leq A.$$
Now replacing $\sigma(A)$ by $-\sigma(A)$ we also get the existence of an $R_2\in\maK_{{U}}(\maE)$ such that 
$0 \leq -A +cQ +\varepsilon +R_2$, in other words we get 
$$  A\leq (R_2+cQ +\varepsilon) .$$
Gathering {these inequalities  we hence obtain}
\begin{equation}
-(c Q+\varepsilon +R_1)\leq A\leq c Q+\varepsilon +R_2 \text{  as self-adjoint operators on }\maE
\end{equation}
with each $R_i \in \maK_U(\maE)$ as allowed. 
\end{proof}

\footnotesize
\bibliographystyle{plain}
\bibliography{Transversalement_elliptique}

\end{document}